\documentclass[a4paper,10pt]{amsart}

\usepackage{dsfont}
\usepackage{graphicx,color,subfig}
\usepackage{mathrsfs}
\usepackage{amscd}
\usepackage{amssymb,amsmath,amsfonts}
\usepackage[mac]{inputenc}
 \usepackage[T1]{fontenc}
 \usepackage[normalem]{ulem}
 \usepackage[english]{babel}
  \usepackage{amsthm}
 \usepackage{bbm}
 \usepackage{verbatim}
 \usepackage{epsfig}
 \usepackage[only,llbracket,rrbracket]{stmaryrd}
 \usepackage{enumerate}
\usepackage{fullpage}
\usepackage{mathrsfs}
\usepackage[colorlinks=true]{hyper ref}
\usepackage{float}

\newtheorem{lem}{Lemma}[section]
\newtheorem{thm}{Theorem}[section]
\newtheorem{deft}{Definition}[section]

\newtheorem{prop}{Proposition}[section]
\newtheorem{coro}{Corollary}[section]
\numberwithin{equation}{section}
\newtheorem{rques}{\textbf{Remarks}}[section]{\vskip 0.5cm}

\newtheorem{rque}{\textbf{Remark}}[section]{\vskip 0.5cm} 
{\vskip 0.5cm}

\newcommand{\T}{{\mathbb T}}
\newcommand{\N}{{\mathbb N}}
\newcommand{\Z}{{\mathbb Z}}

\newcommand{\R}{{\mathbb R}}
\newcommand{\W}{W}
\renewcommand{\S}{{\mathbb S}}

\newcommand{\pa}{{\partial}}
\newcommand{\na}{{\nabla}}

\newcommand{\eps}{{\varepsilon}}

\newcommand{\Ac}{\mathcal{A}}

\newcommand{\CC}{\mathcal{C}\mathcal{C}}
\newcommand{\LL}{\mathcal{L}}
\newcommand{\Nc}{\mathcal{N}}
\newcommand{\Rc}{\, \mathcal{R}  }
\renewcommand{\Mc}{\mathcal{M}}
\newcommand{\Lb}{\mathbb{L}}

 \newcommand{\geqsim}{\,\raisebox{-0.6ex}{$\buildrel > \over \sim$}\,}

\newcommand{\om}{\omega}

\newcommand{\e}{\varepsilon}

\newcommand{\supp}{{\rm supp}\,}

\def\epsilon{\varepsilon}

\DeclareMathOperator{\Leb}{Leb}

\DeclareMathOperator{\V}{Vol}
\let \ker \relax
\DeclareMathOperator{\ker}{Ker}

\newcommand{\Black}{\color{black}}

 \title{Geometric analysis of the linear Boltzmann equation I. \\
Trend to equilibrium
}
\author{Daniel Han-Kwan}\address{CNRS and \'Ecole Polytechnique, Centre de
Math\'ematiques Laurent Schwartz UMR7640,  91128 Palaiseau cedex France}\email{daniel.han-kwan@math.polytechnique.fr} 
\author{Matthieu L\'eautaud}\address{Université Paris Diderot, Institut de Mathématiques de Jussieu UMR7586, Paris Rive Gauche Bâtiment Sophie Germain, 75205 Paris Cedex 13 France} \email{leautaud@math.univ-paris-diderot.fr}
 

\begin{document}

 \begin{abstract}
This work is devoted to the analysis of the linear Boltzmann equation in a bounded domain, in the presence of a force deriving from a potential. The collision operator is allowed to be degenerate in the following two senses: (1) the associated collision kernel may vanish in a large subset of the phase space;  (2) we do not assume that it is bounded below by a Maxwellian at infinity in velocity.

We study how the association of transport and collision phenomena can lead to convergence to equilibrium, using concepts and ideas from control theory. We prove two main classes of results. On the one hand, we show that convergence towards an equilibrium is equivalent to an almost everywhere geometric control condition. The equilibria (which are not necessarily Maxwellians with our general assumptions on the collision kernel) are described in terms of the equivalence classes of an appropriate equivalence relation.
On the other hand, we characterize the exponential convergence to equilibrium in terms of the Lebeau constant, which involves some averages of the collision frequency along the flow of the transport. We handle several cases of phase spaces, including those associated to specular reflection in a bounded domain, or to a compact Riemannian manifold. 

 \end{abstract}
 
 \maketitle
 
\setcounter{tocdepth}{4}
\tableofcontents

 \section{Introduction}
  
This paper is concerned with the study of the linear Boltzmann equation 
\begin{equation}
\label{B}
\partial_t f + v \cdot \nabla_x f - \nabla_x V \cdot \nabla_v f= \int_{\R^d} \left[k(x,v' ,  v) f(v') - k(x,v ,  v') f(v)\right] \, dv', 
\end{equation}
for $x \in \Omega, \, v \in \R^d$,  $d \in \N^*$, where $\Omega$ is either the flat torus $\T^d := \R^d/\Z^d$ or an open bounded subset of $\R^d$, in which case we add some appropriate boundary conditions to the equation.  The linear Boltzmann equation is a classical model of statistical physics, allowing to describe the interaction between particles and a fixed background \cite{Cer-book,DL1english,DL}. Among many possible applications, we mention the modeling of semi-conductors, cometary flows, or neutron transport. We refer the reader interested by further physical considerations or by a discussion of the validity of \eqref{B} in these contexts to~\cite[Chapter IV, \textsection 3]{Cer-book} or \cite[Chapter I, \textsection 5]{DL1english}. We also point out that this equation can be derived in various settings: see for instance \cite{EY} in the context of quantum scattering, or \cite{BGSR} in the context of a gas of interacting particles.

In \eqref{B}, the unknown function $f = f(t,x,v)$ is the so-called distribution function; the quantity $f(t,x,v) \, dv dx$ can be understood as the (non-negative) density at time $t$ of particles whose position is close to $x$ and velocity close to $v$.
The function $V$ is a potential which drives the dynamics of particles; we shall assume throughout this work that $V$ is smooth, more precisely that $ V \in W^{2,\infty}(\Omega)$. 

The linear Boltzmann equation \eqref{B} is a typical example of a so-called hypocoercive equation, in the sense of Villani \cite{V-Hypo}. It is made of a conservative part, namely the kinetic transport operator $v \cdot \nabla_x - \nabla_x V \cdot \nabla_v$ associated to the hamiltonian $H(x,v)= \frac{1}{2} |v|^2 + V(x)$, and a degenerate dissipative part which is the collision operator (i.e. the right hand-side of \eqref{B}). According to the hypocoercivity mechanism of \cite{V-Hypo}, only the interaction between the two parts can lead to convergence to some global equilibrium.

\bigskip
The function $k$ is the so-called collision kernel, which describes the interaction between the particles and the background.
In the following, we shall denote by $C(x,v,f)$ the collision operator, which can be split as
$$
C(x,v,f) = C^+(x,v,f) + C^-(x,v,f),
$$
where 
$$
C^+(x,v,f)= \int_{\R^d}k(x,v' ,  v) f(v') \, dv' ,\quad C^-(x,v,f)=- \left( \int_{\R^d} k(x,v ,  v') \, dv' \right)  f(v)
$$
are respectively the \emph{gain} and the \emph{loss} term.  A first property of this operator is that, due to symmetry reasons, the formal identity holds:
\begin{equation}
\label{ConvMass}
\text{for all }  x \in \Omega, \quad \int_{\R^d}  \int_{\R^d} \left[k(x,v' ,  v) f(v') - k(x,v ,  v') f(v)\right] \, dv' \, dv  = 0.
\end{equation}
This, together with the fact that the vector field $v \cdot \nabla_x  - \nabla_x V \cdot \nabla_v$ is divergence free, implies that the mass is conserved: any solution $f$ of \eqref{B} satisfies
\begin{equation}
\text{ for all } t \geq 0, \quad \frac{d}{dt} \int_{\Omega \times \R^d} f(t,x,v) \, dv dx =0.
\end{equation}

We shall now list the assumptions we make on the collision kernel $k$.

\noindent {\bf A1.} The collision kernel  $k$ belong to the class $C^0(\overline{\Omega} \times \R^d \times \R^d)$ and is nonnegative.

\noindent {\bf A2.} Introducing the Maxwellian distribution:
$$
 \Mc(v) := \frac{1}{(2\pi)^{d/2} }e^{-\frac{|v|^2}{2}},
$$
we assume that $\Mc$ cancels the collision operator, that is
\begin{equation}
\label{Mannule}
\text{for all }  (x,v) \in \Omega \times \R^d, \quad \int_{\R^d} \left[k(x,v' ,  v) \Mc(v') - k(x,v ,  v') \Mc(v)\right] \, dv'  = 0.
\end{equation}

\noindent {\bf A3.} We assume that  
\begin{equation*}
x \mapsto \int_{\R^d \times \R^d} k^2(x,v',v) \frac{\Mc(v')}{\Mc(v)} \, dv' dv \in L^\infty(\Omega). 
\end{equation*} 
It will sometimes be convenient to work with the function
\begin{equation}
\label{bornek}
\tilde{k}(x,v',v) := \frac{k(x,v' ,  v) }{\Mc(v)} .   
\end{equation}
With this notation, Assumptions {\bf A2} and {\bf A3} may be rephrased in a more symmetric way as 
$$
\text{for all }(x,v) \in \Omega \times \R^d ,\quad 
\int_{\R^d}\tilde{k}(x,v',v) \Mc(v')\, dv' = \int_{\R^d}\tilde{k}(x,v,v') \Mc(v')\, dv' ; 
$$ 
\begin{equation*}
x \mapsto \int_{\R^d \times \R^d} \tilde{k}^2(x,v',v) \Mc(v') \Mc(v)\, dv' dv \in L^\infty(\Omega). 
\end{equation*} 
Assumption {\bf A3} is in particular satisfied if $\tilde{k}$ is bounded or has a polynomial growth in the variables $v$ and $v'$.

Note that with assumption {\bf A2}, the function $(x,v)\mapsto \Mc(v)e^{-V}=\frac{1}{(2\pi)^{d/2} }e^{-H}$, which we shall call the \emph{Maxwellian equilibrium}, cancels both the transport operator and the collision operator and thus is a stationary solution of \eqref{B}.

\bigskip
Before going further, let us present usual classes of examples of collision kernels covered by Assumptions {\bf A1}--{\bf A3} and addressed in the present article.

\noindent {\bf E1. ``Symmetric'' collision kernels.} Let $k$ be a collision kernel verifying {\bf A1} and {\bf A3}. We moreover require $\tilde{k}$ to be symmetric with respect to $v$ and $v'$, i.e. $\tilde{k}(x,v,v') = \tilde{k}(x,v',v)$ for all $(x,v,v') \in \Omega \times \R^d \times \R^d$. Notice that for these kernels, {\bf A2} is automatically satisfied.

A classical example of such a kernel is the following.

\noindent{\bf E1'. Linear relaxation kernel.} Taking ${k}(x,v,v')=\sigma(x)\Mc(v')$, with $\sigma \geq 0,  \, \sigma \neq 0$ and $\sigma \in C^0(\overline{\Omega})$ \Black provides the simplest example of kernel in the class {\bf E1}.
This corresponds to the following equation (often called linearized BGK):
$$
\partial_t f + v \cdot \nabla_x f - \nabla_x V \cdot \nabla_v f= \sigma(x) \left( \left(\int_{\R^d} f \, dv\right) \Mc(v) - f  \right) .
$$
This example also belongs to the following class.

\noindent {\bf E2. ``Factorized'' collision kernels} Let $k$ be a collision kernel verifying {\bf A1}--{\bf A3}. We require $k$ to be of the form
$$
k(x,v,v')= \sigma(x) {k}^*(x,v,v'),
$$
with $\sigma \in C^0(\overline{\Omega})$, $\sigma \geq 0,  \, \sigma \neq 0$ and $k^* \in C^0(\overline{\Omega}\times \R^d \times \R^d)$, satisfying for some $\lambda>0$, for all $x \in \Omega$, $v,v' \in \R^d$,
$$
\frac{{k}^*(x,v',v) }{\Mc(v)}+ \frac{{k}^*(x,v,v') }{\Mc(v')} \geq \lambda.
$$
The sub-class of {\bf E2} which is the most studied in the literature (see e.g. \cite{DMS}) consists in the following non-degenerate case.

\noindent {\bf E2'. Non-degenerate collision kernels.} Let $k$ be a collision kernel verifying {\bf A1}--{\bf A3}.
The classical \emph{non-degeneracy} condition consists in assuming that there exists $\lambda>0$ such that for all $x \in \Omega$, $v,v' \in \R^d$
$$
\frac{{k}(x,v',v) }{\Mc(v)}+ \frac{{k}(x,v,v') }{\Mc(v')}= \tilde{k}(x,v',v) + \tilde{k}(x,v,v') \geq \lambda .
$$

Later in the paper (see Section~\ref{secexamples}), we will introduce other classes of collision kernels, that are interesting for our purposes.

\bigskip

Under assumptions {\bf A1}-{\bf A3}, the linear Boltzmann equation~\eqref{B} is well-posed in appropriate Lebesgue spaces, and the weighted $L^2$ norm of its solutions, that is $\int_{\Omega \times \R^d}|f(t,x,v)|^2 \frac{e^{V(x)}}{\Mc(v)} \, dv \, dx $, is dissipated (i.e. decreasing with respect to time, see Lemma~\ref{lemdissip}).

This work aims at describing the large time behavior of solutions of \eqref{B}, under assumptions {\bf A1}--{\bf A3}. The main feature is that we thus allow the collision operator $k$ to be degenerate in the following two senses:

\noindent  $\bullet$ the collision kernel $k$ may vanish in a large subset of the phase space $\Omega \times \R^d$;

\noindent  $\bullet$ we do not assume that $\tilde{k}$ is bounded below by a fixed positive constant at infinity in velocity.

\Black 
However, still in the spirit of Villani's hypocercivity, one may hope that the transport term in~\eqref{B} compensates for this strong degeneracy.
Our goal is to find geometric criteria (on the hamiltonian $H$ and the collision kernel $k$) to characterize:

\noindent $\bullet$ {\bf P1} convergence to a global equilibrium,

\noindent $\bullet$ {\bf P1'} exponential convergence to this global equilibrium.

The study of these questions naturally leads to another problem:

\noindent $\bullet$ {\bf P2} describe the structure and the localization properties of the spectrum of the underlying linear Boltzmann operator.

\bigskip

In recent works \cite{DS, BS1,BS2,BS3}, Bernard, Desvillettes and Salvarani investigated {\bf P1} and {\bf P1'} in a framework close to that of {\bf E2}. In particular, in \cite{BS1}, the authors have shown that in the case where $V=0$,  $(x, v) \in \T^d\times \S^{d-1}$, and $k^*(x,v,v') = k^*(v,v')$ (where $k^*$ is defined in {\bf E2}), the exponential convergence to equilibrium (in the Lebesgue space $L^1$) was equivalent to a \emph{geometric control condition} (similar to that of Bardos-Lebeau-Rauch-Taylor in control theory \cite{RT:74,BLR:92}). 
\Black

Previous works on this topic, for the non-degenerate class of collision kernels {\bf E2'} include \cite{Vid,Vid70,U74,UPG,MK} (spectral approach), \cite{DVcpam,V-Hypo,CCG,MN,DMS} (hypocoercivity methods), \cite{Her} (Lie techniques), and references therein. There are also several related works which concern the \emph{non-linear} Boltzmann equation, but we do not mention them since that equation is not studied in this paper.

\bigskip
In this article, we introduce another point of view on these questions (in particular different from \cite{DS, BS1,BS2,BS3}), by implementing in this context different methods coming from control theory. 
We borrow several ideas from the seminal paper of Lebeau \cite{Leb}, which concerns the decay rates for the damped wave equation.

The goal of this paper is to give necessary and sufficient \emph{geometric} conditions ensuring {\bf P1} and {\bf P1'}, in several settings: we mostly focus on the torus case $\Omega = \T^d$ and on the case of specular reflection in bounded domains. We also show that the methods we develop here are sufficiently robust to handle a general Riemannian setting. 
The related question {\bf P2} is studied in the companion paper \cite{HKL2}.

\bigskip

We now give a more detailed overview of the main results of this work.

\section{Overview of the paper}
\label{main}

In this Section, we give an overview of the results contained in this paper. For readability, we focus on the torus case, i.e. when the phase space is $\T^d \times \R^d$.
Several generalizations (bounded domains with specular reflection, Riemannian manifolds) are actually provided in the following.

  \subsection{Some definitions}

In this section, we introduce the notions needed to characterize convergence and exponential convergence to equilibrium. 
Given a collision kernel $k$ satisfying Assumptions~{\bf A1}--{\bf A3}, we first introduce the set $\omega$ where the collisions are effective.
 
 \begin{deft}
 \label{def-om}
 Define the open set of $\T^d \times \R^d$
 \begin{equation}
\label{omega}
\omega : =  \left\{(x,v) \in \T^d \times \R^d, \, \int_{\R^d} k(x,v,v') \, dv' >0 \right\}.
\end{equation}
 \end{deft}

Note that because of~{\bf A1}--{\bf A2}, we also have
 \begin{equation}
\label{omegabis}
\omega =  \left\{(x,v) \in \T^d \times \R^d, \, \exists v' \in \R^d, k(x,v,v') >0 \right\}= \left\{(x,v) \in \T^d \times \R^d, \, \exists v' \in \R^d, k(x,v',v) >0  \right\}.
\end{equation}

Let us recall the definition of the hamiltonian flow associated to $H$, and associated characteristic curves (or characteristics) in the present setting.
 
  \begin{deft}
  \label{def-carac}
 The hamiltonian flow $(\phi_t)_{t\in\R}$ associated to $H(x,v) = \frac{|v|^2}{2} + V(x)$ is the one parameter family of diffeomorphisms on $\T^d \times \R^d$ defined by $\phi_t(x,v) := (X_t (x,v), \, \Xi_t(x,v))$ with $(x,v) \in \T^d \times \R^d$ and
  \begin{equation}
  \label{hamilflow}
 \left\{
 \begin{aligned}
 &\frac{dX_t(x,v)}{dt} = \Xi_t(x,v), \\
 &\frac{d\Xi_t(x,v)}{dt} = - \nabla_x V(X_t(x,v)), \\
  &X_{t=0}=x,  \quad \Xi_{t=0}=v.
 \end{aligned}
 \right.
 \end{equation}
 The characteristic curve stemming from $(x,v) \in \T^d \times \R^d$ is the curve $\{\phi_t(x,v), t \in \R^+\}$. 
 \end{deft}
Recall that throughout the paper, we assume that $V \in  W^{2,\infty} (\T^d)$, so that the Cauchy-Lipshitz theorem ensures the local existence and uniqueness of the solutions of~\eqref{hamilflow}. Global existence follows from the fact that $H$ is preserved along any characteristic curve. {Note in particular that each energy level $\{H = R\}$ is compact ($V$ being continuous on $\T^d$, it is bounded from below). Hence, each characteristic curve is contained in a compact set of $\T^d \times \R^d$.}

The notions needed to understand the interaction between collisions and transport are of two different nature. We start by expressing purely geometric definitions. Then, we formulate structural-geometric definitions. We finally introduce the weighted Lebesgue spaces used in this paper, as well as a definition of a ``unique continuation type'' property.

  \subsubsection{Geometric definitions}
 \label{defgeo}
We start by introducing the following definitions:
\begin{itemize}
\item The Geometric Control Condition of ~\cite{BLR:92,RT:74}, in Definition~\ref{def: GCC},
\item The Lebeau constants of ~\cite{Leb}, $C^-(\infty)$ and $C^+(\infty)$, in Definition~\ref{definitionCinfini},
\item The almost everywhere in infinite time Geometric Control Condition, in Definition \ref{defaeitgcc}.
\end{itemize}
 
  \bigskip

  Let us first recall the Geometric Control Condition, which is a classical notion in the context of control theory. It is due to Rauch-Taylor~\cite{RT:74} and Bardos-Lebeau-Rauch~\cite{BLR:92}.

 \begin{deft}
 \label{def: GCC}
 Let $U$ be an open subset of $\T^d \times \R^d$ and $T>0$.
 We say that $(U,T)$ satisfies the Geometric Control Condition (GCC) with respect to the hamiltonian $H(x,v) = \frac{|v|^2}{2} + V(x)$ if for any $(x,v) \in \T^d \times \R^d$, there exists $t \in [0,T]$ such that $\phi_t(x,v) = (X_t(x,v),\Xi_t(x,v)) \in U$.
 
 We shall say that $U$ satisfies the Geometric Control Condition with respect to the hamiltonian $H(x,v) = \frac{|v|^2}{2} + V(x)$ if there exists $T>0$ such that the couple $(U,T)$ does.
 \end{deft}
 We now define two important constants in view of the study of the large time behavior of the linear Boltzmann equation, which involve averages of the \emph{damping function} (usually called \emph{collision frequency} in kinetic theory) $b(x,v):= \int_{\R^d} k(x,v,v') \, dv'$ along the flow $\phi_t$. 
 
  \begin{deft}
 \label{definitionCinfini}
 Define the Lebeau constants (\cite{Leb}) in $\R^+ \cup \{+\infty\}$ by 
  \begin{align}C^-(\infty) &:= \sup_{T\in \R^+}  C^-(T), \qquad C^-(T) = \inf_{(x,v) \in \T^d \times \R^d} \frac{1}{T} \int_0^T \left( \int_{\R^d} k(\phi_t (x,v), v')\, dv'\right)\, dt, \\
 C^+(\infty) &:= \inf_{T\in \R^+}  C^+(T), \qquad C^+(T) = \sup_{(x,v) \in \T^d \times \R^d} \frac{1}{T} \int_0^T \left( \int_{\R^d} k(\phi_t (x,v), v')\, dv'\right)\, dt,  
  \end{align}
  where $\phi_t$ denotes the hamiltonian flow of Definition \ref{def-carac}.
\end{deft}
 It is not clear at first sight that $C^-(\infty)$ and $C^+(\infty)$ are well defined: see \cite{Leb} and the beginning of Section~\ref{expocon} for a short explanation. It turns out that only $C^-(\infty)$ will be useful in this paper (but $C^+(\infty)$ will be interesting in the companion paper \cite{HKL2}).

Finally, we introduce a weaker version of the Geometric Control Condition, which will also play an important role in this work.

 \begin{deft}
 \label{defaeitgcc}
 Let $U$ be an open subset of $\T^d \times \R^d$.
 We say that $U$ satisfies the almost everywhere infinite time (a.e.i.t.) Geometric Control Condition  with respect to the hamiltonian $H(x,v) = \frac{|v|^2}{2} + V(x)$ if for almost any $(x,v) \in \T^d \times \R^d$, there exists $s\geq 0$ such that the characteristics $(X_t(x,v), \, \Xi_t(x,v))_{t \geq 0}$ associated to $H$ satisfy $(X_{t=s},\Xi_{t=s}) \in U$.
 \end{deft}
 
 Using this terminology, the usual Geometric Control Condition of Definition~\ref{def: GCC} could be called ``everywhere finite time'' GCC.

 \begin{rque}
 \label{rem-reecriture}
We have the following characterization of the different geometric properties introduced here.
 \begin{itemize}
 
 \item The couple $(U, T)$ satisfies GCC if and only if $\bigcup_{s  \in (0,T)}\phi_{-s}(U)= \T^d \times \R^d$. 
  
 \item The set $U$ satisfies the a.e.i.t. GCC if and only if there exists $\Nc \subset \T^d \times \R^d$ with zero Lebesgue measure such that $\bigcup_{s  \in \R^+}\phi_{-s}(U)\cup \Nc= \T^d \times \R^d$. Note that this implies in particular that $\Nc$ is a closed subset of $\T^d \times \R^d$ satisfying $\phi_s(\Nc) \subset \Nc$ for all $s\geq0$.
\end{itemize}

\end{rque}

   \subsubsection{Structural-geometric definitions}
   \label{defstrugeo}
  
In the sequel, the above geometric definitions are used with $U = \omega$. They hence involve joint properties of the flow $\phi_t$ together with the open set $\omega$, i.e. of the damping function $b =\int_{\R^d} k( \cdot , \cdot, v')\, dv'$. As such, they do not take into account the fine structure of the Boltzmann operator, and in particular the non-local property with respect to the velocity variable of the gain operator $f \mapsto \Big( (x,v) \mapsto \int_{\R^d} k( x ,  v', v) f(x,v')\, dv' \Big)$. 

The next definitions aim at describing how the information may travel between the different connected component of $\omega$. Let us first define two basic binary relations on the open sets of $\T^d \times \R^d$. 

  \begin{deft}
Let $U_1$ and $U_2$ be two open subsets of $\T^d \times \R^d$. We say that $U_1 \Rc_{\phi} \, U_2$ if there exist $s \in \R$ such that $\phi_s(U_1) \cap U_2 \neq \emptyset$.
  \end{deft}
  \begin{deft}
Let $U_1$ and $U_2$ be two open subsets of $\T^d \times \R^d$. We say that $U_1 \Rc_k \, U_2$ if there exist $(x,v_1,v_2) \in \T^d \times \R^d \times \R^d$ with $(x,v_1) \in U_1$, $(x,v_2) \in U_2$ such that $k(x,v_1,v_2)>0$ or $k(x,v_2,v_1)>0$.
  \end{deft}
  Both relations are symmetric and $\Rc_{\phi}$ is moreover reflexive. When restricted to open sets intersecting $\omega$, the relation $\Rc_k$ also becomes reflexive.
 The relation $\Rc_{\phi}$ expresses the fact that the open sets are ``connected through'' the flow $\phi_s$, whereas the relation $\Rc_{k}$ means that the open sets are ``connected through'' a collision.
 
 We also define another convenient $x$-dependent binary relation.
   \begin{deft}
    Let $x \in \T^d$ and $O_1$ and $O_2$ be two open subsets of $\R^d$. We say that $O_1 \Rc_k^x \, O_2$ if there exists $(v_1,v_2) \in \R^d \times \R^d$ with $v_1 \in O_1$, $v_2\in O_2$ such that $k(x,v_1,v_2)>0$ or $k(x,v_2,v_1)>0$. 
  \end{deft}
 Given now an open subset $U$ of $\T^d \times \R^d$, we define $U(x) = \{v \in \R^d, (x,v) \in U\}$.
 With this notation, notice that $U_1 \Rc_k \, U_2$ if and only if there exists $x \in \T^d$ such that $U_1(x) \Rc_k^x \, U_2(x)$.

 \bigskip
 Given $U$ an open set of $\T^d \times \R^d$, we denote by $\CC(U)$ the set of connected components of $U$.
Note that from the separability of $\T^d \times \R^d$, it follows that for any open set $U \subset \T^d \times \R^d$, the cardinality of the set $\CC(U)$ is at most countable.

In the sequel, the main open sets $U$ we are interested in are $\omega$ and $\bigcup_{s\in \R^+}\phi_{-s}(\omega)$.
We now define the key equivalence relation on $\CC(\omega)$.

 \begin{deft}
 \label{def-sim-o}
 Given $\omega_1$ and $\omega_2$ two connected components of $\omega$, we say that $\omega_1 \Bumpeq  \omega_2$ if  there is $N \in \N$ and $N$ connected components $(\omega^{(i)})_{1\leq i \leq N}$ of $\omega$ such that 
 \begin{itemize}
 \item we have $\omega_1 \Rc_{\phi} \, \omega^{(1)}$ or $\omega_1 \Rc_k \, \omega^{(1)}$,
 \item for all $1\leq i \leq N-1$, we have $\omega^{(i)}  \Rc_{\phi} \, \omega^{(i+1)}$ or $\omega^{(i)}  \Rc_k \, \omega^{(i+1)}$,
  \item we have $\omega^{(N)}  \Rc_{\phi} \, \omega_2$ or $\omega^{(N)}  \Rc_k \, \omega_2$.
\end{itemize}
  The relation $\Bumpeq$ is an equivalence relation on the set $\CC(\omega)$ of connected components of $\omega$. For $\omega_1\in \CC(\omega)$, we denote its equivalence class for $\Bumpeq$ by $[\omega_1]$.
 \end{deft}
 This definition means that the two connected components $\omega_1$ and $\omega_2$ are linked by $\Rc_{\phi}$ or $\Rc_k$ through a chain of connected components of $\omega$.

\begin{rque}
We will introduce later in Section \ref{autre-rel-equiv} a related equivalence relation on $\CC\left( \bigcup_{t \geq 0} \phi_{-t} \omega\right)$.
\end{rque}

\subsubsection{Weighted Lebesgue spaces and a unique continuation type property}
 
Let us now introduce the weighted Lebesgue spaces that will be used throughout this paper.
  
    \begin{deft}[Weighted $L^p$ spaces]
    \label{weightedLp}
We define the Banach spaces $\LL^p(\T^d \times \R^d)$ (for $p \in [1,+\infty)$) and $\LL^\infty(\T^d \times \R^d)$  by 
 \begin{equation*}
 \begin{aligned}
 &\LL^p(\T^d \times \R^d) := \Big\{f \in L^1_{loc}(\T^d \times \R^d), \, \int_{\T^d \times \R^d} | f|^p \frac{e^V}{\Mc(v)} \, dv \, dx < + \infty \Big\}, \\
 &\qquad\qquad\qquad\qquad\qquad\qquad\qquad\qquad \|f\|_{\LL^p} = \left(\int_{\T^d \times \R^d} | f|^p \frac{e^V}{\Mc(v)} \, dv \, dx \right)^{1/p}.
 \\
 & \LL^\infty(\T^d \times \R^d) : =  \Big\{f \in L^1_{loc}(\T^d \times \R^d), \, \sup_{\T^d \times \R^d} | f| \frac{e^V}{\Mc(v)}  < + \infty \Big\},  
 \quad \|f\|_{\LL^\infty} = \sup_{\T^d \times \R^d} | f| \frac{e^V}{\Mc(v)} .
  \end{aligned}
  \end{equation*}
 The space $\LL^2$ is a (real) Hilbert space endowed with the inner product
 $$
 \langle f, g \rangle_{\LL^2} := \int_{\T^d \times \R^d} e^{V} \frac{f \, g}{\Mc(v)} \, dv \, dx.
 $$
 \end{deft}

We finally define a Unique Continuation Property for~\eqref{B}. 
 \begin{deft}
 \label{def:UCP}
 We say that the set $\omega$ satisfies the Unique Continuation Property if the only solution $f \in C^0_t(\LL^2)$ to
 \begin{equation}
 \label{eq:UCP}
 \left\{
 \begin{array}{l}
\partial_t f + v \cdot \nabla_x f - \nabla_x V \cdot \nabla_v f= 0, \\
 C(f)=0,
\end{array}
\right.
\end{equation}
is $f= \left(\int_{\T^d \times \R^d} f  \, dv \,dx\right) \frac{e^{-V}}{\int_{\T^d} e^{-V} \, dx} \Mc$.
 \end{deft}
  
  It is actually possible to reformulate in a more explicit form the second equation in~\eqref{eq:UCP}, involving the value of $f$ on connected components of $\omega$ (see Remark~\ref{def:UCP-expli}).

  \subsection{Convergence to equilibrium}
  
Recall that the main goal of this paper is to provide necessary and sufficient geometric conditions to ensure {\bf P1} and {\bf P1'}.
In the case of the torus,  these results can be formulated as follows. 

We first give a general characterization of convergence to some equilibrium.

\begin{thm}
\label{thmconvgene-intro}
The following statements are equivalent.
\begin{enumerate}
\item  The set $\omega$ satisfies the a.e.i.t. GCC with respect to $H$.
\item For all $f_0 \in \LL^2$, there exists a stationary solution $Pf_0$ of~\eqref{B} such that
$$
\| f(t)- Pf_0 \|_{\LL^2} \to_{t \to + \infty} 0,
$$
where $f(t)$ is the solution of~\eqref{B} with initial datum $f_0$.
\end{enumerate}
\end{thm}
Theorem~\ref{thmconvgene-intro} is actually a weak version of Theorem~\ref{thmconv-general}, which is our main result in this direction.
If $(1)$ or $(2)$ holds, we can actually describe precisely the stationary solution $Pf_0$. This involves the equivalence classes of another equivalence relation, which is related to $\Bumpeq$ (see Definition~\ref{def-sim} and Lemma~\ref{equiv-equiv}): we refer to the statement of Theorem~\ref{thmconv-general}. In particular, in several cases, the stationary solution $Pf_0$ is not the Maxwellian equilibrium
\begin{equation}
\label{glomax}
\left(\int_{\T^d \times \R^d} f_0  \, dv \,dx\right) \frac{e^{-V(x)}}{\int_{\T^d} e^{-V} \, dx} \Mc(v) .
\end{equation}
As a matter of fact, we will see that the dimension of the vector space of stationary solutions is equal to the number of equivalence classes for $\Bumpeq$. An explicit example of collision kernel for which this is relevant is given in Section~\ref{secE3'}.

\bigskip

Among all possible stationary solutions of the linear Boltzmann equation, the Maxwellian equilibrium~\eqref{glomax} of course particularly stands out. 
In the next theorem (which is actually a particular case of Theorem~\ref{thmconv-general}), we characterize the situation for which the stationary solution ultimately reached is precisely the projection to the Maxwellian.

\begin{thm}
\label{thmconv-intro}
The following statements are equivalent.
\begin{enumerate}[(i.)]
\item  The set $\omega$ satisfies the Unique Continuation Property.

\item  The set $\omega$ satisfies the a.e.i.t. GCC and there exists one and only one equivalence class for $\Bumpeq$.
\item  For all $f_0 \in \LL^2(\T^d \times \R^d)$, denote by $f(t)$ the unique solution to \eqref{B} with initial datum $f_0$. We have
\begin{equation}
 \label{convergeto0}
 \left\|f(t)- \left(\int_{\T^d \times \R^d} f_0 \, dv \,dx \right)\frac{e^{-V}}{\int_{\T^d} e^{-V} \, dx} \Mc\right\|_{\LL^2} \to_{t \to +\infty} 0.
 \end{equation}

\end{enumerate}

\end{thm}

 As already mentioned in the introduction, our proofs are inspired by ideas which originate from control theory.
For the sake of brevity, we shall not give a detailed explanation of the proof of these results in this introduction. Nevertheless, we would like to comment on an important aspect of the proof of (i.) implies (iii.) in Theorem~\ref{thmconv-intro}. Our approach is based on the fact that the square of the $\LL^2$ norm of a solution $f(t)$ of \eqref{B}, which we shall sometimes refer to as the \emph{energy}, is damped via an explicit dissipation identity, see Lemma \ref{lemdissip}:
  \begin{equation}
  \label{dissip-intro}
\text{ for all } t \geq 0, \quad \frac{d}{dt} \|f(t)\|_{\LL^2}^2 = -D(f),
 \end{equation}
with $D(f) \geq0$, which we shall call the \emph{dissipation} term. 

The idea of the proof is to assume by contradiction that there exists an initial condition $g_0$ in $\LL^2$, with zero mean, such that the associated solution $g(t)$ to~\eqref{B} does not decay to $0$. This yields the existence of $\eps>0$ and of a sequence of times $(t_n)_{n \geq 0}$ going to $+\infty$ such that $\| g(t_n)\|_{\LL^2} \geq \eps$.

We  then study the sequence of shifted functions $h_n(t):= g(t+t_n)$. This is the core of our analysis, which basically consists in a \emph{uniqueness-compactness} argument. We study the weak limit of $h_n$ and show, using the identity~\eqref{dissip-intro} and the \emph{unique continuation property}, that it is necessarily  trivial. Then, we consider the associated sequence of \emph{defect measures} and prove that it is also necessarily trivial, yielding a contradiction.

A difficulty in the analysis comes from the fact that in general, the dissipation term does not control neither the $\LL^2$ distance to the projection on the set of stationary solutions, nor the $\LL^2$ norm of the collision operator. However, what holds true is the \emph{weak coercivity} property
  \begin{equation}
  \label{wc-intro}
\text{for } f \in \LL^2, \quad D(f)=0 \implies C(f)=0,
 \end{equation}
see Lemma~\ref{collannule}. This turns out to be sufficient for our needs. Denoting by 
\begin{equation}
\label{boltzop}
A := T + C , 
\end{equation}
the linear Boltzmann operator, where $T f =  (v\cdot \na_x - \na_x V\cdot \na_v) f$ and $C$ is the collision operator, the property \eqref{wc-intro} together with the skew-adjointness of $T$ then implies that $\ker (A) = \ker(T) \cap \ker (C)$.
This precise structure, together with the equivalence relation $\Bumpeq$, allows to identify $\ker (A)$, i.e. the space of stationary solutions of~\eqref{B}.

Besides, when studying defect measures, the analysis relies on another peculiar structure of \eqref{B}, which is, loosely speaking, made of a propagative and dissipative part (transport and the loss term) and a \emph{relatively compact} part (the gain term). That the gain term is relatively compact is proved via \emph{averaging lemmas} (see Appendix \ref{section-AL} and the references therein).

  \subsection{Exponential convergence to equilibrium}

For what concerns exponential convergence, we need to introduce a technical assumption, which is slightly stronger than {\bf A3}:

\noindent{\bf A3'.}
Assume that there exists a continuous function $\varphi(x,v):= \Theta \circ H(x,v)$ with $\Theta : \R \to [ 1, +\infty)$, such that for all $(x,v) \in \T^d \times \R^d$, we have
$$
\int_{\R^d} k(x,v,v’) \, dv’ \leq \varphi(x,v) ,
$$
and 
$$
 \sup_{x \in \T^d} \int_{\R^d \times \R^d} k^2(x,v’,v) \frac{\Mc(v’)}{\Mc(v)} \left(\frac{\varphi(x,v)}{\varphi(x,v')}-1\right)^2 \, dv dv’ <+\infty .
$$

This assumption is for instance satisfied in the standard case where $\tilde{k}$ has a polynomial growth in the variables $v$ and $v'$ (taking for example $\Theta(t)= \lambda \exp(\frac{1}{4} t)$ and $\lambda>0$ large enough).

We have the following criterion, assuming that {\bf A3'} is satisfied in addition to {\bf A1--A3}.

\begin{thm}[Exponential convergence to equilibrium]
\label{thmexpo-intro} Assume that the collision kernel satisfies {\bf A3'}. 
The following statements are equivalent:
\begin{enumerate}[(a.)]
\item $C^-(\infty) > 0$. 
\item There exist $C>0, \gamma>0$ such that for any $f_0 \in \LL^2(\T^d \times \R^d)$, the unique solution to \eqref{B} with initial datum $f_0$ satisfies for all $t \geq 0$,
\begin{multline}
\label{decexpo-intro} 
\left\|f(t)-\left(\int_{\T^d \times \R^d} f_0 \, dv \,dx \right) \frac{e^{-V}}{\int_{\T^d} e^{-V} \, dx} \Mc \right\|_{\LL^2} \\
\leq C e^{-\gamma t} \left\|f_0-\left(\int_{\T^d \times \R^d} f_0 \, dv \,dx\right)\frac{e^{-V}}{\int_{\T^d} e^{-V} \, dx} \Mc\right\|_{\LL^2}.
\end{multline}

\item There exists $C>0, \gamma>0$ such that for any $f_0 \in \LL^2(\T^d \times \R^d)$, there exists a stationary solution $Pf_0$ of ~\eqref{B} such that the unique solution to \eqref{B} with initial datum $f_0$ satisfies for all $t \geq 0$,
\begin{equation}
\label{decexpogene-intro} 
\left\|f(t)-Pf_0\right\|_{\LL^2} 
\leq C e^{-\gamma t} \left\|f_0-Pf_0\right\|_{\LL^2}.
\end{equation}
\end{enumerate}
\end{thm}

 \begin{rque}
  If we do not assume that {\bf A3'} is satisfied, then we still have that (a.) implies (b.) and (c.). 
 \end{rque}

 If $C^-(\infty)>0$, note in particular that the Geometric Control Condition of Definition~\ref{def: GCC} is satisfied.

One interesting consequence of Theorem~\ref{thmexpo-intro} is a rigidity property of the Maxwellian equilibrium with respect to exponential convergence: loosely speaking, given a linear Boltzmann equation with a collision kernel satisfying {\bf A1-A2-A3-A3'}, if (c.) holds, then the stationary solution ultimately reached is necessarily the projection of the initial datum on the Maxwellian equilibrium.

As an immediate consequence of Theorem~\ref{thmexpo-intro}, we deduce the following result.
\begin{coro}
\label{coroexpo-gen}
Assume that there is $x \in \T^d$ such that $\na V(x)=0$ and $\int k(x,0,v') \, dv' =0$. Then $C^-(\infty)=0$ and  there is no uniform exponential rate of convergence to equilibrium.
\end{coro}
In particular we get  in the free transport case:
\begin{coro}
\label{coroexpo-libre}
Assume that $V=0$ and that $p_x(\omega) \neq \T^d$, where
\begin{equation}
\label{defpx}
p_x(\omega) = \{x \in \T^d , \text{ there exists }v \in \R^d \text{ such that }(x,v) \in \omega\} 
\end{equation}
denotes the projection on the space of positions. Then there is no uniform exponential rate of convergence to equilibrium.
\end{coro}

The proof of Theorem~\ref{thmexpo-intro} is as well inspired by ideas coming from control theory.
The proof of $(a.)\implies (b.)$ relies on the following facts.
\begin{itemize}
\item By~\eqref{dissip-intro}, the exponential decay (i.e. $(ii.)$ in Theorem \ref{thmexpo-intro}) can be rephrased as a certain \emph{observability} inequality relating the dissipation and the energy at time $0$, see Lemma \ref{lemfondamental}: there exist $K,T>0$ such that for all $f_0 \in \LL^2(\T^d\times \R^d)$ with zero mean, 
$$
K\int_0^T D(f) \, dt \geq \| f_0 \|_{\LL^2}^2,
$$
where $f$ is the solution of \eqref{B} with initial datum $f_0$.
\item This inequality is proved using a contradiction argument, following Lebeau \cite{Leb}, which also consists in a uniqueness-compactness argument. The analysis follows the same lines as those of $(i.)$ implies $(iii.)$ in the proof of Theorem~\ref{thmconv-intro}. In particular, it also relies on the weak coercivity property~\eqref{wc-intro}. The main difference is that we need to use here the fact that the Lebeau constant is positive in order to show that the sequence of defect measures  becomes trivial at the limit, yielding a contradiction. 
\end{itemize}
 
For what concerns $(b.)\implies (a.)$, the idea is to contradict the observability condition: we construct a sequence of initial conditions in $\LL^2(\T^d\times \R^d)$ for \eqref{B}, which concentrate to a \emph{trapped} ray (whose existence is guaranteed by the cancellation of the Lebeau constant). Loosely speaking, this corresponds to a geometric optics type construction. In order to justify this procedure, we need that the collision kernel satisfies {\bf A3'}. Finally, we mention that the proof of $(c.)\implies (a.)$ is similar but relies on an additional argument based on the precise version of Theorem~\ref{thmconvgene-intro}.

\begin{rques}\begin{enumerate}
\item As for the damped wave equation~\cite{RT:74, BLR:92, Leb, KoTa}, the study of asymptotic decay rates relies on ``phase space'' analysis. However, as opposed to the wave equation, the Boltzmann equation is directly set on the phase space. As a consequence, the study of associated propagation and damping phenomena only uses ``local'' analysis, whereas that of the wave equation (see~\cite{RT:74, BLR:92, Leb, KoTa}) requires the use of microlocal analysis.

\item One technical difficulty here is to handle the lack of compactness of the phase space $\T^d \times \R^d$ in the variable $v$. It is also possible to consider the equations on a compact phase space. In this case, all our proofs apply, sometimes with significant simplifications. We refer to Section \ref{compactPS}.

\end{enumerate}

\end{rques}

   \subsection{Organization of the paper}{Part \ref{LTB}} is mainly dedicated to the proof of Theorems~\ref{thmconvgene-intro}, \ref{thmconv-intro} and \ref{thmexpo-intro}. In Section \ref{preliminaries}, we give some preliminaries in the analysis; in Section \ref{sec:wp}, we start by proving the dissipation identity and recalling the classical well-posedness result for \eqref{B}, while Section \ref{sec:wc} is dedicated to a detailed study of the kernel of the collision operator, which leads to the weak coercivity property~\eqref{wc-intro}.  
Section \ref{subsectiondecroissance} is mainly devoted to the proofs of Theorems \ref{thmconv-intro} and \ref{thmconvgene-intro};  
in Section \ref{sec:conv}, we start by proving Theorem~\ref{thmconv-intro}. Then, in Section \ref{sec:conv2}, we state and prove Theorem \ref{thmconv-general}, which is the precise version of Theorem~\ref{thmconvgene-intro}. 
Section \ref{secexamples} is dedicated to the application of these results to some particular classes of collision kernels.
In Section \ref{expocon}, we prove Theorem \ref{thmexpo-intro}. Finally in Section \ref{lowerbounds}, we briefly revisit the recent work of Bernard and Salvarani \cite{BS2} in our framework, in order to give some abstract lower bounds on the convergence rate when $C^-(\infty)=0$.

\bigskip

  Part \ref{Boundary} is dedicated to the case of specular reflection in a bounded and piecewise $C^1$ domain of $\R^d$. In Section \ref{prelimbound}, we state preliminary definitions and results in this setting (including the geometric definitions and well-posedness). Then in Section \ref{convboun}, we state and sketch the proof of the exact analogues of Theorems \ref{thmconv-intro} and \ref{thmconv-general}, see Theorems \ref{thm-convboun} and \ref{thmconv-general-bord}.
  In Section \ref{expoboun}, we study exponential convergence to equilibrium. For a more restrictive class of collision kernels (namely {\bf E2} with an additional $L^\infty$ bound)  and under a technical regularity assumption on $p_x(\omega)$ near $ \partial \Omega$ (here $p_x$ denotes the projection on $\Omega$ defined in~\eqref{defpx}; the technical assumption is automatically satisfied when $\overline{p_x({\omega})} \subset  \Omega$), we prove in Theorem \ref{thmexpo-specular} the analogue of Theorem \ref{thmexpo-intro} $(a.)\implies (b.)$. This technical assumption, as well as the fact that we do not prove the analogue of Theorem \ref{thmexpo-intro}, $(b.)\implies (a.)$,  is due to the lack of compactness up to the boundary of $\Omega$ in averaging lemmas. We are nevertheless able to overcome this difficulty for proving $(a.)\implies (b.)$, by adapting an argument of Guo \cite{Guo}. The fact that the collision kernel belongs to the class {\bf E2} allows us to obtain a control of the distance of a solution to its projection on Maxwellians, see Lemma~\ref{cerci2}. This is a quantitative coercivity estimate which is much stronger than the weak coercivity property~\eqref{wc-intro} we use in torus case. Finally, in Section \ref{otherboun}, we give some remarks on some other possible boundary conditions.
 
 \bigskip 
   In Part \ref{part3}, we adapt our analysis in order to handle other geometric situations. In Section \ref{sec:manifold}, we deal with the case of a general compact Riemannian manifold (without boundary): we first explain how to express the linear Boltzmann equation in this setting and generalize Theorems \ref{thmconv-general} and  \ref{thmexpo-intro} to this context. We provide to this end a version of averaging lemmas for kinetic transport equations on a manifold, see Lemma \ref{lem-moyenne-variete}. Finally, in Section \ref{compactPS}, we explain very shortly how to adapt all these results to the case of compact phase spaces.

 \bigskip
   This paper ends with five appendices. In Appendix~\ref{stabob}, we give the equivalence between exponential decay and the observability inequality, used crucially in the proof of Theorem~\ref{thmexpo-intro}. In Appendix~\ref{section-AL}, we give a reminder about classical averaging Lemmas and adapt them to our purposes. In Appendix~\ref{GCCother}, we provide reformulations of some geometric properties. In Appendix~\ref{proofCEXucp}, we provide the proof of Proposition~\ref{CEXucp}, which relates the two equivalence relations, which are key notions in our analysis. In Appendix~\ref{Other}, to stress the robustness of our methods, we explain how the results of Part~\ref{LTB} of this paper concerning large time behavior can be adapted to other Boltzmann-like equations (e.g. relativistic Boltzmann equation or general linearized BGK equation).

  \section{Remarks and Examples}
  
  In this section, we provide several comments on the different geometric definitions introduced in Sections \ref{defgeo} and \ref{defstrugeo}.

\subsection{About a.e.i.t. GCC in the torus}
A first natural question is to understand the a.e.i.t. GCC in the usual situation of free transport (i.e. $V=0$), when $\omega$, i.e. the set where collisions are effective, is of the simple form $\omega_x \times \R^d$. We prove that this condition is satisfied for {\em any} nonempty $\omega_x \subset \T^d$. We also prove that this situation is very particular, and unstable with respect to small perturbations of the potential.

  \begin{prop}\label{coroUCP} Suppose that $V = 0$ and that $ \omega = \omega_x \times \R^d$, where $\omega_x$ is a non-empty open subset of $\T^d$.
Then $(i.)-(ii.)-(iii.)$ in Theorem~\ref{thmconv-intro} hold.
 \end{prop}
 
 Such a result is in particular relevant for the study of the linearized BGK equation (class {\bf E1'}).
Proposition \ref{coroUCP} shows that there is convergence to the Maxwellian~\eqref{glomax} equilibrium as soon as $\sigma \neq 0$.

\medskip
On the other hand, $\omega$ being fixed, we give an example of dynamics (i.e. exhibit a potential $V$) for which the a.e.i.t. GCC fails. More precisely, we prove that this property is very unstable with respect to small perturbations of the potential: for $\omega = \omega_x \times \R^d \neq \T^d \times \R^d$ there exist arbitrary small potentials (in any $C^k$-norm) such that $\omega$ does not satisfy a.e.i.t. GCC for the associated Hamiltonian.
 This illustrates the fact that the free transport on the torus is a very uncommon situation.

 \begin{prop}
 \label{CEXucp}
Assume that $\overline{p_x(\omega)} \neq \T^d$, where $p_x(\omega)$ 
denotes the projection of $\omega$ on $\T^d$ defined in~\eqref{defpx}. Then there exists a potential $V \in C^\infty(\T^d)$ such that for any $\eps>0$, $\omega$ does not satisfy a.e.i.t. GCC for the Hamiltonian $H_\eps(x,v) = \frac{|v|^2}{2}+ \eps V(x)$.
 \end{prop}

The proof of Proposition~\ref{coroUCP} and Proposition~\ref{CEXucp} are given respectively in Section~\ref{secE3pp} and Appendix~\ref{proofCEXucp}.

  \subsection{The equivalence relation on $\CC \left(\bigcup_{t \geq 0} \phi_{-t} (\omega)\right)$}
  \label{autre-rel-equiv}
 We define here another key equivalence relation $\sim$ on the set of connected components of $\bigcup_{s\in \R^+}\phi_{-s}(\omega)$. We then explain the link between the two equivalence relations $\sim$ on $\CC \left(\bigcup_{t \geq 0} \phi_{-t} (\omega)\right)$ and $\Bumpeq$ on $\CC(\omega)$.
 \begin{deft}
 \label{def-sim}
Given $\Omega_1, \Omega_2$ two connected components of $\bigcup_{s\in \R^+}\phi_{-s}(\omega)$, we say that $\Omega_1 \sim \Omega_2$ if  there is $N \in \N$ and $N$ connected components $(\Omega^{i})_{1\leq i \leq N}$ of $\bigcup_{s\in \R^+}\phi_{-s}(\omega)$ such that 
 \begin{itemize}
 \item we have $\Omega_1 \Rc_{\phi} \, \Omega^{(1)}$  or $\Omega_1 \Rc_k \, \Omega^{(1)}$,
 \item for all $1\leq i \leq N-1$, we have $\Omega^{(i)}  \Rc_{\phi} \, \Omega^{(i+1)}$ or $\Omega^{(i)}  \Rc_k \, \Omega^{(i+1)}$,
  \item we have $\Omega^{(N)} \Rc_{\phi} \, \Omega_2$ or $\Omega^{(N)}  \Rc_k \, \Omega_2$.
\end{itemize}
  The relation $\sim$ is an equivalence relation on the set of $\CC\left(\bigcup_{s\in \R^+}\phi_{-s}(\omega) \right)$ of connected components of $\bigcup_{s\in \R^+}\phi_{-s}(\omega)$.
For $\Omega_1 \in \CC\left(\bigcup_{s\in \R^+}\phi_{-s}(\omega) \right)$,  we denote its equivalence class for $\sim$ by $[\Omega_1]$.
 \end{deft}
 
\begin{rque}
Observe that the two equivalence relations $\Bumpeq$ and $\sim$ are defined the same way, except that $\Bumpeq$ is considered as a relation on $\CC(\omega)$ and $\sim$ on $\CC(\bigcup_{s\in \R^+}\phi_{-s}(\omega))$. Note also that  if the phase space is without boundary, in the definition of $\sim$ above, one can omit that $\Omega^{(i)}  \Rc_{\phi} \, \Omega^{(i+1)}$ and require only 
$$
\Omega_1 \Rc_k \, \Omega^{(1)}, \quad \text{for all } 1\leq i \leq N-1, \, \Omega^{(i)}  \Rc_k \, \Omega^{(i+1)},\quad  \text{and } \Omega^{(N)}  \Rc_k \, \Omega_2 .
$$
However, this precise Definition~\ref{def-sim} will be useful when considering the case of a boundary value problem for which the associated flow $\phi_t$ is no longer continuous in time.
\end{rque}

The following lemma gives the link between the two equivalence relations.
We define the function
\begin{equation*}
\begin{array}{rrcl}
\Psi : & \CC(\omega) & \to &\CC\Big(\bigcup_{s \in \R^+}\phi_{-s}(\omega)\Big) \\
& \omega_0 & \mapsto &  \Omega_0 , \text{ such that } \omega_0 \subset \Omega_0.
\end{array}
\end{equation*}
The application $\Psi$ maps $\omega_0 \in \CC(\omega)$ to the connected component $\Omega_0$ of $\bigcup_{s \in \R^+}\phi_{-s}(\omega)$ containing $\omega_0$. 

\begin{lem}
\label{equiv-equiv}
Given $\omega_1, \omega_2 \in \CC(\omega)$, we have $\omega_1 \Bumpeq \omega_2$ if and only if $\Psi(\omega_1) \sim \Psi(\omega_2)$. 
As a consequence, $\Psi$ goes to the quotient  defining a {\em bijection} $\tilde\Psi$ between the equivalence classes of $\Bumpeq$ and $\sim$:
\begin{equation*}
\begin{array}{rrcl}
\tilde\Psi  : & \CC(\omega) / \Bumpeq & \to &\CC\Big(\bigcup_{s \in \R^+}\phi_{-s}(\omega)\Big)/\sim. 
\end{array}
\end{equation*}
In particular, the number of equivalence classes for $\Bumpeq$ in $\CC(\omega)$ and for $\sim$ in $\CC\Big(\bigcup_{s \in \R^+}\phi_{-s}(\omega)\Big)$ are equal.
\end{lem}

The proof of Lemma~\ref{equiv-equiv} is given in Appendix~\ref{prooflemequiv}. As a consequence of this lemma, all the results of this paper (together with their proofs) can be formulated with $\Bumpeq$ or with $\sim$ equivalently.

\subsection{Comparing $C^-(\infty)>0$ and GCC} 

The statement that $C^-(\infty)>0$ is in general stronger to the fact that $\omega$ satisfies the Geometric Control Condition with respect to the hamiltonian $H(x,v) = \frac{|v|^2}{2} + V(x)$.  This is due to the the non-compactness of the phase space $\T^d \times \R^d$.

Assume for instance that $V = 0$ so that $\phi_t(x,v) = ( x+tv ,v )$: the characteristic curves are straight lines and the velocity component $v$ is preserved by the flow. Take $k(x,v',v)>0$ on the whole $\T^d \times \R^d \times \R^d$. In this situation, $\omega =\T^d \times \R^d$ and so, it satisfies automatically GCC in any positive time. Assume further that $k$ does not depend on the space variable, i.e. $k(x,v',v) = k(v',v)$ and that there exists a sequence $(v_n)$ such that $\int_{\R^d} k(v_n,v') dv' \to 0$. Then, we have $\left(\int_{\R^d} k( \cdot ,v') dv' \right) \circ \phi_t (x,v) = \int_{\R^d} k( v ,v') dv'$ (as the flow preserves $v$) and hence, for any $n \in \N$, we have $C^-(\infty) \leq \int_{\R^d} k(v_n,v') dv'$. This yields $C^-(\infty)= 0$ although $\omega$ satisfies GCC. As an explicit example, one can take $k(x, v', v ) = \Mc(v') \Mc(v)^2$.

\medskip
Note finally that if $\int_{\R^d} k( x,v ,v') dv'$ is uniformly bounded from below at infinity (i.e. there exists $C,R>0$ such that $\int_{\R^d} k( x,v ,v') dv' \geq C$ for all $(x,v) \in \T^d \times B(0,R)^c$), then $C^-(\infty)>0$ and GCC become equivalent.

The next paragraph shows that our result is indeed more general.

\subsection{Example of exponential convergence without a bound from below at infinity}
Here, we produce a simple example of dynamics and collision kernel such that $C^-(\infty)>0$, but neither $\tilde{k}$ nor $\int_{\R^d} k( x,v ,v') dv'$ are uniformly bounded from below at infinity.

For this, assume $(x,v) \in \T \times \R$ and take $V=0$, so that $\phi_t(x,v) = ( x+tv ,v )$. We identify $\T$ to $[-1/2,1/2)$ with periodic boundary conditions.
Define $\alpha \in C^0(\T; \R^+)$ with support contained in $(-1/3,1/3)$ and satisfying $\alpha = 1$ on $[-1/4,1/4]$ and $\psi \in  C^0(\R ; \R^+)$ such that $\psi(v) \to_{|v| \to +\infty} 0$ and $\psi = 1$ on $[-2,2]$.
Consider the collision kernel in the class {\bf E1}
$$
k(x,v,v') = \tilde{k}(x,v,v')\Mc(v') , \qquad \tilde{k}(x,v,v') = \left[  \alpha(x) + \psi(v) \psi(v')\right].
$$
We first readily check that $\tilde{k}>0$ on $\T^d \times \R^d \times \R^d$ and hence $\omega=\T \times \R$. We also remark that for any $R>0$,
$$
\inf_{(x,v,v') \in \T^d \times B(0,R)^c \times B(0,R)^c} \tilde{k}(x,v,v') =0, \quad \text{and} \quad \inf_{(x,v) \in \T^d \times B(0,R)^c}  \int_{\R^d} k( x,v ,v') dv' = 0 .
$$
Nevertheless, we can prove that $C^-(\infty) \geq C^-(1)>0$, and thus, by Theorem~\ref{thmexpo-intro}, there is exponential convergence to the Maxwellian equilibrium. We set $\beta := \int_{\R^d} \psi(v') \Mc(v') \, dv'>0$ and take $(x,v) \in \T \times \R$.
\begin{itemize}
\item If $v\in [-2,2]$, then we have
$$
\int_0^1 \int_{\R^d} k(\phi_t(x,v),v') \, dv' \, dt \geq \beta \int_0^1 \psi(v) \,dt = \beta>0.
$$
\item If $v\notin [-2,2]$, then, denoting by $\lfloor v \rfloor$ the integer part of $v$, we have
\begin{align*}
\int_0^1 \int_{\R^d} k(\phi_t(x,v),v') \, dv' \, dt  &\geq \int_0^1 \alpha(x+tv) \, dt 
\geq \int_0^{\frac{\lfloor v \rfloor}{|v|}} \alpha(x+tv) \, dt  
\geq \frac{1}{2} \frac{\lfloor v \rfloor}{|v|} \geq \frac{1}{4}.
\end{align*}
 \end{itemize}
 This proves that $C^-(1)>0$ and thus $C^-(\infty)>0$.


 \subsection*{Acknowledgments} 
We wish to thank Diogo Ars\'enio for several interesting and stimulating discussions related to this work.

  \part{Large time behavior of the linear Boltzmann equation on the torus}
  \label{LTB}

  Throughout this section, $\Omega = \T^d$.
  
    \section{Preliminary results}
 \label{preliminaries}
 
    \subsection{Well-posedness and dissipation}
\label{sec:wp}
  For readability, we shall sometimes denote 
  $$C(f):=C(x,v,f) = \int_{\R^d} \left[k(x,v' ,  v) f(v') - k(x,v ,  v') f(v)\right] \, dv'.$$

The following dissipation identity holds for solutions to \eqref{B}.
  \begin{lem}
  Let $k$ be collision kernel satisfying {\bf A1}--{\bf A3}.
 \label{lemdissip}Let $f \in C^0(\R ;\LL^2)$ be a solution to \eqref{B}. The following identity holds, for all $t\in \R$:
 \begin{equation}
 \label{eqdissipation}
 \frac{d}{dt}  \|f(t)\|_{\LL^2}^2 = - D(f),
 \end{equation}
 where
 \begin{align}
D(f) &= - 2\langle C(f), f \rangle_{\LL^2} \nonumber \\ 
 &=  \frac{1}{2} \int_{\T^d} e^{V} \int_{\R^d} \int_{\R^d} \left( \frac{k(x,v' ,  v)}{\Mc(v)} + \frac{k(x,v ,  v')}{\Mc(v')} \right) \Mc(v) \Mc(v') \left(\frac{f(t,x,v)}{\Mc(v)}- \frac{f(t,x,v')}{\Mc(v')}\right)^2 \, dv' \, dv \, dx .  \label{defD}
 \end{align}
\end{lem}
 The term $D(f)$ will often be referred to as the dissipation term in the following. The proof is rather classical and follows~\cite{DGP}.
 
 \begin{proof}[Proof of Lemma \ref{lemdissip}] Multiply \eqref{B} by $f \,  \frac{e^V}{\Mc(v)}$ and integrate with respect to $x$ and $v$. This yields
 \begin{align*}
\frac{1}{2}   \frac{d}{dt}  \|f(t)\|_{\LL^2}^2 + \int_{\T^d \times \R^d} \left(v \cdot \nabla_x f - \nabla_x V \cdot \nabla_v f\right) f \,  \frac{e^V}{\Mc(v)} \, dv \, dx 
= \int_{\T^d} e^V \int_{\R^d} C(f) \frac{f}{\Mc} \, dv  \, dx.
\end{align*}
On the one hand, the contribution of the transport term vanishes
  \begin{align*}
 \int_{\T^d \times \R^d} \left(v \cdot \nabla_x f - \nabla_x V \cdot \nabla_v f\right) f \,  \frac{e^V}{\Mc(v)} \, dv \, dx 
& =  \frac{1}{2}\int_{\T^d \times \R^d} \left(v \cdot \nabla_x  - \nabla_x V \cdot \nabla_v \right) |f|^2  \,  \frac{e^V}{\Mc(v)} \, dv \, dx \\
 &= -  \frac{1}{2}\int_{\T^d \times \R^d}  |f|^2   \left(v \cdot \nabla_x  - \nabla_x V \cdot \nabla_v \right)  \frac{e^V}{\Mc(v)} \, dv \, dx \\
 & =0,
 \end{align*}
 since $\left(v \cdot \nabla_x  - \nabla_x V \cdot \nabla_v \right)  \frac{e^V}{\Mc(v)}  =0$.
 On the other hand, following \cite{DGP}, we have for any $x \in \T^d$ the identity 
\begin{align}
\label{eq: DGP1}
\int_{\R^d} C(x,v,f) \frac{f}{\Mc} \, dv &= \int_{\R^d}  \int_{\R^d} \left[k(x,v' ,  v) f(v') - k(x,v ,  v') f(v)\right] \, dv' \frac{f(v)}{\Mc(v)} \, dv  \nonumber \\
&= \int_{\R^d}  \int_{\R^d} k(x,v' ,  v) f(v')\frac{f(v)}{\Mc(v)} \, dv' \, dv -  \int_{\R^d}  \int_{\R^d}  k(x,v ,  v') \frac{|f(v)|^2}{\Mc(v)} \, dv' \, dv .
\end{align}
Symmetrizing the first term in the right hand-side of \eqref{eq: DGP1} yields
\begin{multline*}
 \int_{\R^d}  \int_{\R^d} k(x,v' ,  v) f(v')\frac{f(v)}{\Mc(v)} \, dv' \, dv =\\
  \frac{1}{2}  \int_{\R^d}  \int_{\R^d} k(x,v' ,  v) f(v')\frac{f(v)}{\Mc(v)} \, dv' \, dv  
 +\frac{1}{2}   \int_{\R^d}  \int_{\R^d} k(x,v ,  v') f(v)\frac{f(v')}{\Mc(v')} \, dv' \, dv .
\end{multline*}
Concerning the second term in the right hand-side of \eqref{eq: DGP1}, we use \eqref{Mannule} to obtain
\begin{align*}
& -  \int_{\R^d}  \int_{\R^d}  k(x,v ,  v') \frac{|f(v)|^2}{\Mc(v)} \, dv' \, dv \\
&=-\frac{1}{2}  \int_{\R^d}  \int_{\R^d}  k(x,v ,  v') \frac{|f(v)|^2}{\Mc(v)} \, dv' \, dv  -\frac{1}{2}  \int_{\R^d}  \int_{\R^d}  k(x,v ,  v')\Mc(v)  \frac{|f(v)|^2}{\Mc(v)^2} \, dv' \, dv \\
&=-\frac{1}{2}  \int_{\R^d}  \int_{\R^d}  k(x,v ,  v') \frac{|f(v)|^2}{\Mc(v)} \, dv' \, dv  -\frac{1}{2}  \int_{\R^d}  \int_{\R^d}  k(x,v',  v)\Mc(v')  \frac{|f(v)|^2}{\Mc(v)^2} \, dv' \, dv \\
&=- \frac{1}{4} \int_{\R^d}  \int_{\R^d}  k(x,v ,  v') \frac{|f(v)|^2}{\Mc(v)} \, dv' \, dv  - \frac{1}{4} \int_{\R^d}  \int_{\R^d}  k(x,v',  v) \frac{|f(v')|^2}{\Mc(v')} \, dv' \, dv\\
& \quad - \frac{1}{4} \int_{\R^d}  \int_{\R^d}  k(x,v' ,  v)\Mc(v') \frac{|f(v)|^2}{\Mc(v)^2} \, dv' \, dv -  \frac{1}{4} \int_{\R^d}  \int_{\R^d}  k(x,v,  v')\Mc(v) \frac{|f(v')|^2}{\Mc(v')^2} \, dv' \, dv .
\end{align*}
Combining the last two identities, we can now collect together the terms with $k(x,v',v)$ (resp. $k(x,v, v')$) and rewrite the right hand-side of~\eqref{eq: DGP1} as a sum of two squares. Namely, this provides
\begin{align*}
\int_{\R^d} C(x,v,f) \frac{f}{\Mc} \, dv = - \frac{1}{4} \int_{\R^d} \int_{\R^d} \left( \frac{k(x,v' ,  v)}{\Mc(v)} + \frac{k(x,v ,  v')}{\Mc(v')} \right) \Mc(v) \Mc(v') \left(\frac{f(v)}{\Mc(v)}- \frac{f(v')}{\Mc(v')}\right)^2 \, dv' \, dv .
\end{align*}
This yields~\eqref{eqdissipation} and concludes the proof of the Lemma.
 \end{proof}

 We have the following standard well-posedness result for the linear Boltzmann equation \eqref{B}. 
 
  \begin{prop}[Well-posedness of the linear Boltzmann equation]
 \label{prop:WP}
Assume that $f_0 \in \LL^2$. Then there exists a unique $f\in C^0(\R ;\LL^2)$ solution of~\eqref{B} satisfying $f|_{t = 0} =f_0$, and we have
\begin{equation}
\text{for all } t \geq 0, \quad  \frac{d}{dt} \| f(t)\|_{\LL^2}^2 = - D(f(t)),
\end{equation}
 where $D(f)$ is defined in \eqref{defD}. If moreover $f_0 \geq 0$ a.e., then for all $t \in \R$ we have $f(t, \cdot,\cdot)\geq 0$ a.e. (Maximum principle).

 \end{prop}
 
 \begin{proof}[Proof of Proposition \ref{prop:WP}] 
 We denote
\begin{align*}
 (A_0 f)(x, v)  &=  (v\cdot \na_x - \na_x V\cdot \na_v) f(x,v) +
\left(\int_{\R^d}  k(x,v,v') \, dv'\right) f(x,v), \\
(K f)(x, v) &= - \int_{ \R^d} k(x,v',v) f(x,v') \, dv', \\
Af &= A_0 f + Kf
\end{align*}
with domain
$$
 D(A) = D(A_0) = \left\{f\in \LL^2 , (v\cdot \na_x - \na_x V\cdot \na_v) f \in \LL^2, \, \left(\int_{\R^d}  k(\cdot,v') \, dv'\right) f \in \LL^2 \right\}.
$$
The operator $A_0$ generates a strongly continuous group on $\LL^2$ given by 
\begin{equation}
\label{semigpA0}
e^{-tA_0} u = \exp\left(- \int_0^t \int_{\R^d}  k(\phi_{-(t-s)}(x,v),v') \, dv' \, ds \right) u \circ \phi_{-t},
\end{equation}
where $\phi_s(x,v) = (X_s(x,v),\Xi_s(x,v))$ denotes the hamiltonian flow of Definition~\ref{def-carac}.

On the other hand, the Cauchy-Schwarz inequality yields
\begin{align*}
\|Kf\|_{\LL^2}^2 & \leq \int_{\T^d \times \R^d} \frac{e^{V(x)}}{\Mc(v)}  \left(\int_{\R^d}k(x,v',v)^2\Mc(v') dv' \right)   \left(\int_{\R^d}\frac{f(x,v')^2}{\Mc(v')} dv' \right)  dx \, dv \\
& \leq\left(\sup_{x \in \T^d}\int_{\R^d}\int_{\R^d} k^2(x, v',v)\frac{\Mc(v')}{\Mc(v)}\, dv'\, dv \right) \|f\|_{\LL^2}^2
\end{align*}
The operator $K$ is hence bounded in $\LL^2$, with
$$
\|K\|_{\LL^2 \to \LL^2}\leq\left(\sup_{x \in \T^d}\int_{\R^d}\int_{\R^d} k^2(x, v',v)\frac{\Mc(v')}{\Mc(v)}\, dv'\, dv \right)^\frac12,
$$
 which is finite by {\bf A3}.

 According to~\cite[Chapter 3, Theorem 1.1]{Paz}, this implies the well-posedness of the Cauchy problem associated to~\eqref{B}.

For the maximum principle,  we recall that we have the classical representation formula:
$$
f = \sum_{n=0}^{+\infty} \mathcal{K}^n ( e^{-t A_0} f_0),
$$
where
\begin{multline*}
\mathcal{K} g = \int_0^t   \left(\int_{\R^d} [k(X_{-(t-s)} (x,v) ,v' ,  \Xi_{-(t-s)} (x,v)) g(s,X_{-(t-s)} (x,v),v') \, dv'\right)\\ 
\times \exp \left( - \int_s^t \left( \int_{\R^d} k(\phi_{-(t-u)}(x,v),v') \, dv' \right)   \, du\right) \, ds.
\end{multline*}
Recall that $k\geq 0$. If $f_0 \geq 0$ a.e., one can observe from~\eqref{semigpA0} that for all $t\geq 0$, $ e^{-t A_0} f_0 \geq 0$ a.e., and for all $n \in \N$, $\mathcal{K}^n (  e^{-t A_0}f_0) \geq 0$ a.e..
This concludes the proof.

 \end{proof}
 
 A useful consequence of the maximum principle, of the linearity of the equation, and of Assumption {\bf A2} is the following statement. If  $f_0 \in \LL^2 \cap \LL^\infty$, then  the unique solution  of~\eqref{B} starting from $f|_{t = 0} =f_0$, satisfies
\begin{equation}
\sup_{t\geq 0} \| f(t)\|_{\LL^\infty} \leq  \| f_0\|_{\LL^\infty}  .
\end{equation}

\subsection{Weak coercivity}
\label{sec:wc} 
In this section, we describe some properties of the collision kernel $C$ and associated dissipation $D$. 
In several proofs of the paper, we shall need to exploit some local coercivity properties of the dissipation. In particular, we would like to have the \emph{weak coercivity} property
  \begin{equation}
  \label{wc}
\forall f \in \LL^2, \quad D(f)=0 \implies C(f)=0
 \end{equation}
(and thus,  $D(f)=0$ is equivalent to $C(f)=0$).
A difficulty comes from the fact that in general the dissipation term does not control neither the $\LL^2$ distance to the projection on the set of stationary solutions, nor the $\LL^2$ norm of the collision operator.

The main result is the following lemma.
\begin{lem}
\label{collannule}
Let $k$ be a collision kernel satisfying {\bf A1}--{\bf A3}. 
Let $T \in (0, +\infty]$ and denote $\omega = \cup_{i \in I} \omega_i$ the partition of $\omega$ in connected components.
Then, the following three properties are equivalent 
\begin{enumerate}
\item \label{toto1}  $f \in C^0(0,T; \LL^2)$ satisfies ${C}(f(t))=0$ for all $t \in [0,T]$. 
\item  $f \in C^0(0,T; \LL^2)$ satisfies ${D}(f(t))=0$ for all $t \in [0,T]$. 
\item 
\begin{itemize}
\item for all $i \in I$,  we have $f(t,x,v)=\rho_i (t,x)\Mc(v)$ on $[0,T] \times \omega_i$;
\item for $i,j \in I$ and $x \in \T^d$, we have: $\omega_i(x) \Rc_k^x \, \omega_j(x) \Longrightarrow \rho_i(t,x) = \rho_j(t,x)$ for all $t \in [0,T]$.
\end{itemize}
\end{enumerate}
\end{lem}
This lemma only states properties of the collision kernel. As such, it is not concerned with the time dependence, that we shall drop in the proof. 
\Black
\begin{proof}[Proof of Lemma \ref{collannule}]

By definition, we have $D(f)=\langle C(f) , f\rangle_{\LL^2}$ so that $(1) \Longrightarrow (2)$. 

Then, from ${D}(f)=0$, Equation \eqref{defD} implies that 
\begin{equation}
\label{rho-i-j}
\frac{f(x,v)}{\Mc(v)} = \frac{f(x,v')}{\Mc(v')} \quad \text{almost everywhere in } S:=\{(x,v, v'), \, {k}(x,v',v) + k(x,v',v)>0\} ,
\end{equation}
Let $(x,v) \in \omega$. Thus, there exists $v' \in \R^d$ such that $(x,v,v') \in S$. By continuity of ${k}$, there exists a neighborhood $U$ of $(x,v)$ such that for all $(y,w) \in U$, we have $(y,w,v') \in  S$. Thus, for all  $(y,w) \in U$, we have
$$
\frac{f(y,w)}{\Mc(w)} = \frac{f(y,v')}{\Mc(v')}
$$
that is to say that locally, $(y,w) \mapsto \frac{f(y,w)}{\Mc(w)}$ is function of $y$ only. Therefore, for all $i \in I$, there is a function $\rho_i$ such that $\frac{f(x,v)}{\Mc(v)} = \rho_i (x)$ on $ \omega_i$. 

Furthermore, take $i,j \in I$ and $x \in \T^d$ such that $\omega_i(x)\Rc_k^x \, \omega_j(x)$. There exists $v_i,v_j \in \R^d \times \R^d$ such that $(x,v_i) \in \omega_i$,  $(x,v_j) \in \omega_j$, and $(x,v_i,v_j) \in S$. It then follows from~\eqref{rho-i-j} and $\frac{f(x,v_i)}{\Mc(v_i)} = \rho_i (x)$, $\frac{f(x,v_j)}{\Mc(v_j)} = \rho_j (x)$ that $\rho_i(x) = \rho_j(x)$. 
This concludes the proof of $(2)\Longrightarrow (3)$.

Finally, let us check that a function satisfying the two assumptions of $(3)$ cancels the collision operator, i.e. that for all $x,v$, we have $C(x,v,f)=0$.

Let $(x,v) \in \omega$ (if $(x,v) \notin \omega$, then $C(x,v,f)=0$). Let $i \in I$ such that $(x,v) \in \omega_i$. We have
\begin{align*}
C(x,v,f) &= \int \tilde{k}(x,v',v) f(x,v') \, dv' \Mc(v) - \int \tilde{k}(x,v,v') \Mc(v') \, dv' f(x,v) \\
&=  \int \tilde{k}(x,v',v) f(x,v') \, dv' \Mc(v) - \int \tilde{k}(x,v,v') \Mc(v') \, dv' \rho_i(x) \Mc(v) \\
&=  \sum_{j \in J_i} \int \tilde{k}(x,v',v) \mathds{1}_{\omega_j}(x,v') f(x,v') \, dv' \Mc(v) - \int \tilde{k}(x,v,v') \Mc(v') \, dv' \rho_i(x) \Mc(v) \\
&=  \sum_{j \in J_i} \int \tilde{k}(x,v',v) \mathds{1}_{\omega_j}(x,v') \Mc(v') \, dv' \rho_j(x) \Mc(v) - \int \tilde{k}(x,v,v') \Mc(v') \, dv' \rho_i(x) \Mc(v),
\end{align*}
where $J_i$ is the largest subset of $I$ such that for all $j \in J_i$, there exists $v' \in \R^d$ such that $(x,v') \in\omega_j$ and $\tilde{k}(x,v',v)>0$.
According to the second property satisfied by $f$, for all $j \in J_i$, 
$$\rho_j(x) =\rho_i(x),$$
and thus we deduce
\begin{align*}
C(x,v,f) &=  \sum_{j \in J_i} \int \tilde{k}(x,v',v) \mathds{1}_{\omega_j}(x,v') \Mc(v') \, dv' \rho_i(x) \Mc(v) - \int \tilde{k}(x,v,v') \Mc(v') \, dv' \rho_i(x) \Mc(v) \\
&= \left(\int \tilde{k}(x,v',v)  \Mc(v') \, dv'  - \int \tilde{k}(x,v,v') \Mc(v') \, dv' \right) \rho_i(x) \Mc(v)\\
&= 0.
\end{align*}
The last line comes from the fact that $k$ satisfies {\bf A2}.
This concludes the proof of $(3) \Longrightarrow (1)$.
\end{proof}

\begin{rque}
\label{def:UCP-expli}
Another benefit of Lemma~\ref{collannule} is that it allows to rephrase the Unique Continuation Property, in a slightly more explicit way.

Denote $\omega = \cup_{i \in I} \omega_i$ the partition of $\omega$ in connected components. The set $\omega$ satisfies the Unique Continuation Property if and only if the following holds. The only solution  $f \in C^0(\R ; \LL^2)$ to
 \begin{equation*}
 \label{eq:UCP-expli}
\partial_t f + v \cdot \nabla_x f - \nabla_x V \cdot \nabla_v f= 0,
\end{equation*}
satisfying the following properties
\begin{itemize}
\item for all $i \in I$, $f(t,x,v)=\rho_i (t,x)\Mc(v)$ on $[0,T] \times \omega_i$;
\item for $i,j \in I$ and $x \in \T^d$, $\omega_i(x) \Rc_k^x \, \omega_j(x) \Longrightarrow \rho_i(t,x) = \rho_j(t,x)$ for all $t \in [0,T]$,
\end{itemize}
is $f= \left(\int_{\T^d \times \R^d} f  \, dv \,dx\right) \frac{e^{-V}}{\int_{\T^d} e^{-V} \, dx} \Mc(v)$.
 \end{rque}

\begin{rque}
If $\tilde{k}$ is bounded in $L^\infty$ (but only in this case), we can actually prove that the dissipation controls the norm of a ``symmetrized'' collision operator. To state and prove such a result, we introduce the symmetrized collision kernel
  \begin{equation}
  \label{barrekstar}
  \overline{k}(x,v',v) := \frac{\tilde{k}(x,v',v)+ \tilde{k}(x,v,v')}{2}, \quad k^*(x,v',v) :=   \overline{k}(x,v',v)\Mc(v) .
   \end{equation}
Note in particular that $ \overline{k}(x,v',v) \in L^\infty (\T^d \times \R^d \times \R^d)$  if $\tilde{k}  \in L^\infty (\T^d \times \R^d \times \R^d)$. 

We also introduce the associated collision operator 
  \begin{equation}
  \label{mathcolop}
 \mathcal{C} (f) :=  \int_{\R^d} \left[{k^*}(x,v' ,  v) f(v') - {k^*}(x,v ,  v')) f(v)\right] \, dv'.
 \end{equation}

 Note that we have $ \mathcal{C} (f) = C(f)$ if and only if $\overline{k}(x,v',v)=\tilde{k}(x,v',v)$, i.e. if $\tilde{k}(x,v',v)$ is symmetric with respect to $v$ and $v'$ (this corresponds to the class {\bf E1}).  
 
 \begin{lem}
 \label{cerci1}
Let $k$ be a collision kernel satisfying {\bf A1}--{\bf A2}, and such that $\tilde{k} \in L^\infty$. For any $f \in \LL^2$, we have
  \begin{equation}
 \|\overline{k}\|_{L^\infty} D(f) \geq  \left\|\mathcal{C} (f) \right\|_{\LL^2}^2,
 \end{equation} 
   \end{lem}

 \begin{proof}[Proof of Lemma \ref{cerci1}]
According to the definition of the dissipation~\eqref{defD} and the symmetry of $\overline{k}$, we have   
\begin{align*}
 D(f) &\geq   \int_{\T^d} e^{V} \int_{\R^d} \int_{\R^d}  \overline{k}(x,v',v)   \Mc \Mc' \left(\frac{f}{\Mc}- \frac{f'}{\Mc'}\right)^2 \, dv' \, dv \, dx ,
 \end{align*}
and hence 
\begin{align*}
 \|\overline{k}\|_{L^\infty} D(f) \geq  \int_{\T^d} e^{V} \int_{\R^d} \int_{\R^d}  \overline{k}^2(x,v',v)   \Mc \Mc' \left(\frac{f}{\Mc}- \frac{f'}{\Mc'}\right)^2 \, dv' \, dv \, dx.
 \end{align*}
By Jensen's inequality it follows that
  \begin{align*}
 \|\overline{k}\|_{L^\infty} D(f) &\geq  \int_{\T^d} e^{V} \int_{\R^d} \Mc(v)\left(\int_{\R^d}  \overline{k}(x,v',v) \Mc(v') \left(\frac{f(v)}{\Mc(v)}- \frac{f(v')}{\Mc(v')}\right) \, dv'\right)^2 dv \, dx \\
 &=   \int_{\T^d} e^{V} \int_{\R^d}    \frac{1}{\Mc(v)}   \mathcal{C} (f)^2 \,  dv \, dx\\
 &=   \left\| \mathcal{C} (f)  \right\|_{\LL^2}^2,
 \end{align*}
 where we again used the symmetry of $\overline{k}$. This concludes the proof of the lemma.
 \end{proof}

\end{rque}

 
 \section{Characterization of convergence to equilibrium}
 \label{subsectiondecroissance}

 In this Section, we shall first give a proof of Theorem \ref{thmconv-intro}; then we will provide a proof of Theorem~\ref{thmconvgene-intro}, which will be a consequence of our main result in this direction, namely Theorem~\ref{thmconv-general}.

We start with a technical lemma concerning the evolution under the flow of the connected components of $\bigcup_{s\in \R^+}\phi_{-s}(\omega)$.

\begin{lem}
\label{leminvarflow}
Set $\tilde\Omega = \bigcup_{s\in \R^+}\phi_{-s}(\omega)$ and denote by $(\Omega_i)_{i\in I}$ the partition of $\tilde\Omega$ in connected components, and $\Ac = \T^d \times \R^d \setminus \tilde\Omega$. Then we have for all $t \geq 0$, for all $i \in I$,
\begin{equation}
\label{inclusionsphi}
 \phi_{-t}(\tilde\Omega) \subset \tilde\Omega ,  \quad  \phi_{t}(\Ac) \subset \Ac, \quad \phi_{-t}(\Omega_i) \subset \Omega_i .
\end{equation}
If moreover $\omega$ satisfies a.e.i.t. GCC (i.e. $\Ac$ has zero Lebesgue measure), then for all $i \in I$, for all $t \in \R$,
\begin{equation}
\label{eq-propconn}
\phi_{t}(\Omega_i) = \Omega_i  \quad \text{almost everywhere}. 
\end{equation}
\end{lem}

\begin{proof}[Proof of Lemma~\ref{leminvarflow}]
First, we just remark that for all $t \geq 0$, we have $\phi_{-t}(\tilde\Omega) = \bigcup_{s\in \R^+}\phi_{-t-s}(\omega) \subset \tilde\Omega$. Taking the complement of this inclusion yields for $t \geq 0$, $\phi_{t}(\Ac) \subset \Ac$.

Let us now fix $i \in I$ and prove that 
\begin{equation}
\label{omegaiphit}
\text{for all } t \geq 0 , \quad 
\phi_{-t}(\Omega_i)\subset \Omega_i.
\end{equation}
Take $(x,v) \in \Omega_i$ and $t> 0$. If $\phi_{-t}(x, v) \notin \Omega_i$, then there exists $t_0 \in (0,t]$ such that $\phi_{-t_0}(x, v) \notin \tilde\Omega$ since $\Omega_i$ is a connected component of this set. This is in contradiction with $\phi_{-t_0}(\tilde\Omega) \subset \tilde\Omega$. This implies \eqref{omegaiphit}.

Finally, if $\Ac$ has zero Lebesgue measure, then $|\phi_{-t}(\Ac)| = 0$ as well. As a consequence, the identity
$$
\T^d \times \R^d = \phi_{-t} (\T^d \times \R^d) = \phi_{-t} \left(\bigcup_{i \in I} \Omega_i  \cup \Ac \right) 
= \bigcup_{i \in I}  \phi_{-t} (\Omega_i ) \cup \phi_{-t}(\Ac) 
$$
yields $\T^d \times \R^d = \bigcup_{i \in I}  \phi_{-t} (\Omega_i )$ almost everywhere. Since for all $i \in I$ and $t\geq 0$, $\phi_{-t} (\Omega_i ) \subset \Omega_i$, we obtain (still for $t \geq 0$) $\phi_{-t} (\Omega_i ) = \Omega_i$ almost everywhere, from which the conclusion of the lemma follows.

\end{proof}

\begin{rque}
Note that if $V=0$ and $\omega$ satisfies the following property: $(x,v) \in \omega \Leftrightarrow (x,-v) \in \omega$, then the inclusions in~\eqref{inclusionsphi} become equalities. Hence, all sets considered in~\eqref{inclusionsphi} are then invariant by $\phi_t$ for all $t\in \R$.
\end{rque}

\subsection{Proof of Theorem \ref{thmconv-intro}}
\label{sec:conv}

We shall prove that $(i.) \implies (iii.)$, that $(iii.) \implies (ii.)$ and finally that $(ii.) \implies (i.)$.

\bigskip

\noindent $(i.) \implies (iii.)$  We prove that the Unique Continuation Property (of Definition~\ref{def:UCP}) implies the decay.  

We first prove the expected convergence for data enjoying more regularity, i.e. we prove
 \begin{equation}
 \label{convergeto0'}
\text{for all } f_0 \in \LL^2\cap \LL^\infty, \quad \left\|f(t)-\int_{\T^d \times \R^d} f_0 \, dv \,dx \frac{e^{-V}}{\int_{\T^d} e^{-V} \, dx} \Mc(v)\right\|_{\LL^2} \to_{t \to +\infty} 0 .
 \end{equation}

Since \eqref{B} is linear, $g(t):= f(t)-\int_{\T^d \times \R^d} f_0 \, dv \,dx \frac{e^{-V}}{\int_{\T^d} e^{-V} \, dx} \Mc(v)$ is a solution to $\eqref{B}$ with initial datum $g(0) = f(0)-\int_{\T^d \times \R^d} f_0 \, dv \,dx \frac{e^{-V}}{\int_{\T^d} e^{-V} \, dx} \Mc(v) \in \LL^2\cap \LL^\infty$, satisfying $\int g(0) \, dv \, dx =0$.  
Therefore proving \eqref{convergeto0'} is equivalent to proving that $\|g(t)\|_{\LL^2} \to 0$. 

We argue by contradiction. Assume that there exists an initial datum $g_{0}$ in $\LL^2 \cap \LL^\infty$ (with $\int g_0 \, dv \,dx = 0 $), $\eps >0$ and an increasing sequence $(t_n)_{n \in \N}$ such that:
\begin{equation}
\label{borne}
 t_n \geq e^n , \quad \text{and} \quad \left\|g_0\right\|_{\LL^2} \geq \left\|g(t_n)\right\|_{\LL^2} > \eps.
\end{equation}
From this sequence, we may extract a subsequence (still denoted $(t_n)$) satisfying 
\begin{equation}
\label{tn+1 tn}
t_{n+1}-t_n \to + \infty.
\end{equation}
According to the Maximum Principle of Proposition \ref{prop:WP},  we have
\begin{equation}
   \label{eqequi}
\text{for all } t \geq 0, \quad \|g(t) \|_{\LL^\infty} \leq \|g_0 \|_{\LL^\infty}.
 \end{equation}

 We introduce the shifted function $h_n(t,x,v) := g(t_n + t, x,v)$. By the time translation invariance of \eqref{B},  $h_n$ is still a solution to \eqref{B}, with initial datum $h_n(0)= g(t_n)$. Using \eqref{borne}, up to some extraction, we can assume that there is $\alpha \in [{\eps},  \left\|g_0\right\|_{\LL^2}]$ such that 
    \begin{equation}
    \label{limithn}
    \| h_n(0) \|_{\LL^2} \to_{n \rightarrow + \infty}  \alpha.
 \end{equation}
 Note also that by conservation of the mass, for all $n \in \N$ and all $t \geq 0$,
     \begin{equation}
    \label{eqconsmas}
 \int h_n(t) \, dv dx = \int g_0 \, dv dx =0.
 \end{equation} 
 Using the dissipation identity of Lemma \ref{lemdissip} for $g$, we have:
     \begin{equation*}
   \| g(t_{n+1}) \|_{\LL^2}^2 -    \| g(t_{n}) \|_{\LL^2}^2 = - \int_{t_n}^{t_{n+1}} D(g) \, dt,
    \end{equation*}
    that is (using the time translation invariance):
         \begin{equation*}
   \| h_{n+1}(0) \|_{\LL^2}^2 -    \| h_{n}(0) \|_{\LL^2}^2 = - \int_{0}^{t_{n+1}-t_n} D(h_n) \, dt .
    \end{equation*}
    This, together with \eqref{tn+1 tn} and \eqref{limithn}, implies that for any $T>0$,
             \begin{equation}
             \label{Dzero}
\int_{0}^{T} D(h_n) \, dt \to 0.
    \end{equation}
    Now, up to another extraction, since for any $n \in \N$,
\begin{equation} 
\label{hng0}
\text{for all } t \geq 0, \quad   \| h_n (t) \|_{\LL^2} \leq \| g_0 \|_{\LL^2},
   \end{equation}
    we can assume that $h_n \rightharpoonup h$ weakly in $L^2_{t, loc}\LL^2$. Let us now prove that $h=0$.  First, since $h_n$ is a solution to~\eqref{B}, \Black by linearity, $h$ also satisfies~\eqref{B}.

We denote now $d\lambda := \mathds{1}_{[0,T]}(t)\left( \frac{k(x,v' ,  v)}{\Mc(v)} + \frac{k(x,v ,  v')}{\Mc(v')} \right) \Mc(v) \Mc(v') \, dv dv' dx dt$. Then, introducing 
$$\tilde{h}_n(t,x,v,v') := \frac{h_n(t,x,v)}{\Mc(v)}- \frac{h_n(t,x,v')}{\Mc(v')},$$ 
we observe that $\| \tilde{h}_n \|_{L^2(d\lambda)}^2= \int_0^TD(h_n) \, dt$ and thus, by \eqref{Dzero}, we deduce that $\tilde{h}_n(t,x,v,v')$ is uniformly bounded in $L^2(d\lambda)$. Consequently, up to a another extraction, $\tilde{h}_n$ weakly converges in $L^2(d\lambda)$ to some $\tilde{h}$. By uniqueness of the limit in the sense of distributions, we have $\tilde{h}= \frac{h(t,x,v)}{\Mc(v)}- \frac{h(t,x,v')}{\Mc(v')}$.

Then, by \eqref{Dzero} and weak lower semi-continuity, we deduce that for any $T>0$, we have
$$
    \left\|{D}(h) \right\|_{L^1(0,T)} = \| \tilde{h} \|_{L^2(d\lambda)}^2\leq \liminf_{n\to +\infty}  \| \tilde{h}_n \|_{L^2(d\lambda)}^2 =  \liminf_{n\to +\infty}   \left\|{D}(h_n) \right\|_{L^1(0,T)}= 0.
$$
 Thus, by weak coercivity (see Lemma~\ref{collannule}),  we infer that ${C}(h)=0$ on $[0,T]$, for any $T>0$,  and therefore $h$ satisfies the kinetic transport equation \eqref{eq:UCP}.
Using the Unique Continuation Property (see Definition~\ref{def:UCP}), we deduce that
$$h=\left(\int h \, dv dx \right) \frac{e^V}{\Mc}.$$
 Since $h_{n} \rightharpoonup h$ weakly in $L^2_t \LL^2$, using \eqref{eqconsmas}, we obtain in particular  that  for any $T>0$
$$
\int_0^{T} \left(\int h \, dv dx \right) dt =0. 
$$
Since $\int_0^{T} \left(\int h \, dv dx \right) dt = T  \left(\int h(0) \, dv dx \right)$, we deduce that $\int h \, dv dx=0$ so that $h=0$.

    \bigskip
    
   Let us now consider the sequence of defect measures $\nu_n := |h_n|^2$, which, according to~\eqref{hng0} and~\eqref{eqequi} satisfies, for all $n \in \N$,
   $$
\text{for all } t \geq 0, \quad  \| \nu_n(t) \|_{\LL^{1}} \leq \| g(0)\|_{\LL^2}^2, \quad \| \nu_n(t) \|_{\LL^{\infty}} \leq C_0\| g(0)\|_{\LL^\infty}^2,
   $$
   for $C_0 = \max_{(x,v) \in \T^d \times \R^d} e^{-V(x)}\Mc(v)$.
We have that, up to another subsequence  $\nu_n \rightharpoonup \nu$ weakly-$\star$ in $L^{\infty}_{t,loc} \LL^\infty$. Let us compute the equation satisfied by $\nu$: to this purpose, we consider \eqref{B} satisfied by $h_n$ and multiply it by $h_n$. We obtain:
\begin{equation}
\label{eq:nunnun}
        \begin{aligned}
        \partial_t \nu_n &+ v \cdot \nabla_x \nu_n - \nabla_x V \cdot \nabla_v \nu_n \\
        &=  2 \left[ \int_{\R^d} \left[k(x,v' ,  v) h_n(v') - k(x,v ,  v') h_n(v)\right] \, dv' \right] 
        h_n\\
        &= 2 \left(\int_{\R^d} k(x,v' ,  v) h_n(v') \, dv'\right) h_n - 2 \left( \int_{\R^d} k(x,v ,  v')\, dv'  \right)  \nu_n .
            \end{aligned}
            \end{equation}
               Using the averaging lemma of Corollary \ref{lemmoyenne} and the fact that $h_n$ weakly converges to $0$, we deduce that
        \begin{equation}
        \label{termecompacite}
        \int_{\R^d} k(x,v' , v) h_n(t,x,v') \, dv' \rightarrow 0
        \end{equation}
        strongly in $L^2_{t,loc}\LL^2$. On the other hand, according to~\eqref{hng0}, the sequence $(h_n)$ is uniformly bounded in $L^2_{t,loc}\LL^2$. We hence obtain
          \begin{equation}
          \label{compacitemoyenne}
         \left(\int_{\R^d} k(x,v' , v) h_n(t,x,v') dv'\right) h_n \to 0 \quad \text{strongly in } L^1_{t,loc}\LL^1.  
        \end{equation}
        The second term, by definition of $\nu$, weakly converges to $- 2 \left( \int_{\R^d} k(x,v ,  v')\, dv'  \right)  \nu$ in the sense of distributions, so that $\nu$ satisfies the equation
    \begin{equation}
    \label{transnu}
  \partial_t \nu + v \cdot \nabla_x \nu - \nabla_x V \cdot \nabla_v \nu = - 2 \left( \int_{\R^d} k(x,v ,  v')\, dv'  \right)  \nu.
    \end{equation}
Moreover, writing
            \begin{equation}
    \label{buternu-pre}
    \left( \int_{\R^d} k(x,v ,  v')\, dv'  \right)  |h_n|^2 = -C(h_n) h_n   +    \left(\int_{\R^d} k(x,v' , v) h_n(t,x,v') dv'\right) h_n,
    \end{equation}
and using~\eqref{Dzero} together with~\eqref{compacitemoyenne}, we deduce  that
\begin{equation}
    \label{buternu}
    \left( \int_{\R^d} k(x,v ,  v')\, dv'  \right)  \nu =0.
   \end{equation}
Thus,  \eqref{transnu} combined with \eqref{buternu} entails that $\nu$ satisfies the kinetic transport equation
   $$
    \partial_t \nu + v \cdot \nabla_x \nu - \nabla_x V \cdot \nabla_v \nu =0,
    $$
which also shows that $\nu \in C^0_t(\LL^{2})$. \Black
This, combined with the fact that $\nu=0$ on $\R^+ \times \omega$ (again coming from~\eqref{buternu}) and the Unique Continuation Property, implies 
$$\nu = \left(\int \nu(0) \, dv \,dx\right) \frac{e^{-V}}{\int_{\T^d} e^{-V} \, dx} \Mc(v).$$
According to \eqref{buternu}, this means that $\nu=0$. 
    
We now prove that there is no loss of mass at infinity. Let $\delta >0$ and $R>0$.
   We have:
\begin{align*}
   \int_{\T^d \times \R^d} \nu_n(0) \frac{e^{V}}{\Mc} \mathds{1}_{v \in \R^d \setminus B(0,R)} \, dv \,dx &\leq     \| g(0)\|_{\LL^\infty}^2 \int_{\T^d \times \R^d} \frac{\Mc}{e^{V}}\mathds{1}_{v \in \R^d \setminus B(0,R)} \, dv \,dx \\
   &\leq  \| g(0)\|_{\LL^\infty}^2 \int_{\T^d } {e^{-V}}  \,dx  \int_{|v| \geq R} \Mc(v) \, dv,
\end{align*}
   which is exponentially converging to zero as $R\to \infty$. This yields
   \begin{equation}
   \label{nolossofmass}
 \lim_{R \to \infty} \int_{\T^d \times \R^d} \nu_n(0) \frac{e^{V}}{\Mc} \mathds{1}_{v \in \R^d \setminus B(0,R)} \, dv \,dx = 0 .
   \end{equation}
   Therefore, on the one hand, up to a subsequence, we can assume that $\nu_n(0)\frac{e^{V}}{\Mc} \rightharpoonup \nu(0)\frac{e^{V}}{\Mc}$ tightly in $ \mathcal{M}^+_{x,v}$ and thus
      $$
\int_{\T^d \times \R^d} \nu_n(0) \frac{e^{V}}{\Mc} dv \, dx \rightarrow  \int_{\T^d \times \R^d} \nu(0) \frac{e^{V}}{\Mc} dv \, dx =0.
   $$ 
    On the other hand, using \eqref{limithn}, we have
   $$
\int_{\T^d \times \R^d} \nu_n(0) \frac{e^{V}}{\Mc} dv \, dx  \rightarrow \alpha  \, >0.
   $$   
 This yields a contradiction, and concludes the proof of \eqref{convergeto0'}.

   \bigskip
   We finally deduce \eqref{convergeto0} by an approximation argument. Let $f_0 \in \LL^2$ and $f(t)$ be the solution associated to $f_0$. Let $\eps>0$. There exists $f_{0,\eps} \in \LL^2 \cap \LL^\infty$ such that $\|f_0 - f_{0,\eps}\|_{\LL^2} \leq \eps$. Let $f_\eps(t)$ be the solution associated to $f_{0,\eps}$. Since $f- f_\eps$ is a solution of \eqref{B} with initial datum $f_0 -f_{0,\eps}$ we also have for any $t\geq 0$, $\|f(t) - f_\eps(t)\|_{\LL^2} \leq \eps$.
   
   By \eqref{convergeto0'}, there exists $t_0 \geq 0$ such that for all $t\geq t_0$,
$$
\left\|f_\eps(t)-\int_{\T^d \times \R^d} f_{0,\eps} \, dv \,dx \frac{e^{-V}}{\int_{\T^d} e^{-V} \, dx} \Mc(v)\right\|_{\LL^2} \leq \eps.
$$
   Thus, it follows that for all $t\geq t_0$,
 \begin{align*}
&\left\|f(t)-\int_{\T^d \times \R^d} f_{0} \, dv \,dx \frac{e^{-V}}{\int_{\T^d} e^{-V} \, dx} \Mc(v)\right\|_{\LL^2} \\
&\leq  \| f(t)- f_\eps(t) \|_{\LL^2}  +  \left\|f_\eps(t)-\int_{\T^d \times \R^d} f_{0,\eps} \, dv \,dx \frac{e^{-V}}{\int_{\T^d} e^{-V} \, dx} \Mc(v)\right\|_{\LL^2} \\
& \quad+ 
 \left\| \int_{\T^d \times \R^d} ( f_{0, \eps}- f_{0})   dv \,dx \frac{e^{-V}}{\int_{\T^d} e^{-V} \, dx} \Mc(v)\right\|_{\LL^2} \\
& \leq 2 \eps + \int_{\T^d \times \R^d} |f_{0}-f_{0,\eps}| \, dv \,dx .
\end{align*}
Besides, we have 
$$
\left(\int_{\T^d \times \R^d} |f_{0}-f_{0,\eps}| \, dv \,dx \right)^2 \leq \|f_{0}-f_{0,\eps}\|_{\LL^2}^2 \left( \int_{\T^d \times \R^d}e^{-V}\Mc(v) \, dx \, dv \right).
$$
The last two inequalities together yield, for all $t \geq t_0$,
 \begin{align*}
\left\|f(t)-\int_{\T^d \times \R^d} f_{0} \, dv \,dx \frac{e^{-V}}{\int_{\T^d} e^{-V} \, dx} \Mc(v)\right\|_{\LL^2}
\leq \left(2+ \left(\int_{\T^d}e^{-V} \, dx\right)^\frac12 \right)\eps ,
\end{align*}
which concludes the proof of \eqref{convergeto0}.

   \bigskip

 \noindent $(iii) \Rightarrow (ii.)$ Assume that $(ii.)$ does not hold. 
 Either $\T^d \times \R^d \setminus \bigcup_{s  \in \R^+}\phi_{-s}(\omega)$ has positive Lebesgue measure, or the equivalence relation $\Bumpeq$ has two or more equivalence classes (or equivalently, by Lemma~\ref{equiv-equiv}, the binary relation  $\sim$ has two or more equivalence classes).

Suppose first that  $\Ac:= \T^d \times \R^d \setminus \bigcup_{s  \in \R^+}\phi_{-s}(\omega)$ has positive Lebesgue measure. We set 
$$
 f_0 (x,v) = \mathds{1}_{\Ac}(x,v)\, e^{-V(x)} \Mc(v). 
 $$
 Note that $f_0$ satisfies $\int_{\T^d \times \R^d}f_0(x,v) dx \, dv > 0$ as $\Ac$ is of positive Lebesgue measure.
We consider $f(t, \cdot)$ the solution to \eqref{eqhypo1} with initial datum $f_0$, given by
$$
f(t,x,v) = f_0 \circ \phi_{-t}(x,v) = \mathds{1}_{\phi_{t}(\Ac)}(x,v)\, e^{-V(x)} \Mc(v)
$$
Moreover, for all $t\geq0$, we have $\phi_{t}(\Ac) \cap \omega = \emptyset$ since  
$\Ac\cap \phi_{-t}(\omega) = \emptyset$. Therefore, $f =0$ on $\R^+\times \omega$ so that, according to the characterization of $\omega$ in~\eqref{omegabis}, $C(f) = 0$ and $f$ is also a solution of~\eqref{B}. Moreover, this implies that 
$$
\left\| f(t) - \left(\int_{\T^d \times \R^d} f_0 \, dv dx\right) \frac{e^{-V}}{\int_{\T^d} e^{-V} \, dx} \Mc(v) \right\|_{\LL^2} \not\to_{t \to +\infty} 0
$$
Thus, $(iii.)$ does not hold.

\bigskip

Suppose now that $\T^d \times \R^d \setminus \bigcup_{s  \in \R^+}\phi_{-s}(\omega)$ has zero Lebesgue measure and that the equivalence relation  $\sim$ has (at least) two distinct equivalence classes, say $[\Omega_1]$ and $[\Omega_2]$.

We define now a function $f(x,v)$ as follows
\begin{equation*}
f(x,v) = \sum_{\Omega' \in [\Omega_1]} \mathds{1}_{\Omega'} (x,v) e^{-V(x)}\Mc(v).
\end{equation*}
Using~\eqref{eq-propconn} in Lemma~\ref{leminvarflow}, we deduce that for all $t \geq 0$, 
$$
f \circ \phi_t (x,v) =\sum_{\Omega' \in [\Omega_1]} \mathds{1}_{\Omega'} \big(\phi_t (x,v)\big) e^{-V(x)}\Mc(v) =
\sum_{\Omega' \in [\Omega_1]} \mathds{1}_{\Omega'} (x,v) e^{-V(x)}\Mc(v) = f  (x,v) ,
$$  
so that $f$ is a stationary solution of $\partial_t f + v \cdot \nabla_x f - \nabla_x V \cdot \nabla_v f= 0$.

There remains to prove that $f$ cancels the collision operator. Denote $\omega = \cup_{i \in I} \omega_i$ the partition of $\omega$ in connected components.

By Lemma~\ref{collannule}, $f$ cancels the collision operator if and only if $f$ satisfies the following two properties. 
\begin{enumerate}
\item For all $i \in I$,
$$
f=\rho_i (x)\Mc(v) \quad \text{on  } \omega_i,
$$
\item For $i,j \in I$, and $x \in \T^d$ $\omega_i(x), \Rc_k^x \, \omega_j(x) \Longrightarrow \rho_i(x) = \rho_j(x)$.
\end{enumerate}

We check now that our function $f$ satisfies these two properties.

Let $i \in I$. If for all $\Omega' \in [\Omega_1]$, $\omega_i \cap \Omega' =\emptyset$, then $f=0$ on $\omega_i$ (and hence satisfies $(1)$ with $\rho_i=0$). If there is $\Omega' \in [\Omega_1]$ such that $\omega_i \cap \Omega' \neq \emptyset$, then since $\Omega'$ is a connected component of $\bigcup_{s\in \R^+}\phi_{-s}(\omega)$, we have $\omega_i \subset \Omega'$. Thus we have $f=e^{-V}\Mc $ on $\omega_i$ (and hence $f$ satisfies $(1)$ with $\rho_i=e^{-V}$). We deduce that for all $i \in I$, $f$ is of the form $\rho_i (x)\Mc(v)$ on $\omega_i$.

Take now $i,j \in I$ and $x\in \T^d$ such that $\omega_i(x) \Rc_k^x \, \omega_j(x)$. 
 Let $\Omega^{(i)}$ (resp. $\Omega^{(j)}$) be the connected component of $\bigcup_{s\in \R^+}\phi_{-s}(\omega)$ which contains $\omega_i$ (resp. $\omega_j$). By definition of the relations $\Rc_k$ and $\Rc_k^x$, this directly yields $\Omega^{(i)} \Rc_k \, \Omega^{(j)}$, and {\em a fortiori} we deduce that $\Omega^{(i)}\sim \Omega^{(j)}$: in other words, these are in the same equivalence class for $\sim$. According to the definition of $f$, this implies $\rho_i(x) = \rho_j(x)$ (which is equal to $ e^{-V(x)}$ if $\Omega^{(i,j)} \in [\Omega_1]$ and to $0$ if not).
 
 Therefore, the function $f$ satisfies the two properties and, by Lemma~\ref{collannule}, cancels the collision operator.

However, we have
$$
\cup_{\Omega' \in [\Omega_2]} \Omega' \subset \T^d \times \R^d \setminus \cup_{\Omega' \in [\Omega_1]} \Omega'
$$
and $\cup_{\Omega' \in [\Omega_2]} \Omega' $ has a positive Lebesgue measure.
Consequently, the measure of $\T^d \times \R^d \setminus \cup_{\Omega' \in [\Omega_1]} \Omega'$ is positive, so that $f$ is a stationary solution of~\eqref{B} which is not a uniform Maxwellian. As a consequence, $(iii.)$ does not hold.

\bigskip
   
\noindent $(ii.) \Rightarrow (i.)$ Assume that $(ii.)$ holds. Let $f \in C^0_t(\LL^2)$ be a solution to
 \begin{align}
\label{eqhypo1} \partial_t f + v \cdot \nabla_x f - \nabla_x V \cdot \nabla_v f= 0, \\
\label{eqhypo2} C(f)=0.
\end{align}
As usual, without loss of generality, we can assume that $\int_{\T^d \times \R^d} f \, dv \,dx=0$. The goal is to show that $f=0$.

  Since $f$ cancels the collision operator, by Lemma~\ref{collannule}, the restriction of $f$ to $\omega$ is necessarily of the form $f_{|\omega}(t,x,v) = \sum_{i \in I}\mathds{1}_{\omega_i}(x,v) \rho_i(t,x) \Mc(v)$, where $\omega = \bigcup_{i \in I} \omega_i$ is the partition of $\omega$ in connected components.  Furthermore, for $i,j \in I$, if there is $x \in \T^d$ such that $\omega_i(x) \Rc_k^x \, \omega_j(x)$, then $\rho_i(t,x) = \rho_j(t,x)$.

 Consider now $\tilde{\omega}$ a connected component of $\omega$. Consider for $(t,x, v) \in \R^+ \times \tilde\omega$,  $g(t,x) := \frac{e^V}{\Mc(v)} \, f$. 
  We have, in the sense of distributions in $\R^+ \times \tilde\omega$,
 $$
 \partial_t g + v \cdot \nabla_x g =  \frac{e^V}{\Mc(v)} \left[\partial_t f + v \cdot \nabla_x f + (v\cdot  \nabla_x V ) f\right].
 $$
 Since $f$ satisfies \eqref{eqhypo1}, this implies that $g$ satisfies the free transport equation 
 \begin{equation}
 \label{transg}
   \partial_t g + v \cdot \nabla_x g =0   ,
 \end{equation}
in the sense of distributions in $\R^+ \times \tilde\omega$.

Let $(x, v) \in \tilde\omega$. Since $\tilde\omega$ is open, there exists $\delta>0$ such that $B(x,\delta)\times B(v,\delta) \subset \tilde\omega$ and $\eta>0$ such that for all $t \in (-\eta,\eta)$, for all $(x', v')\in B(x,\delta)\times B(v,\delta)$, $(x'+tv' ,v') \in \tilde\omega$. Integrating~\eqref{transg} along characteristics we obtain 
\begin{equation}
\label{gpropcons}
g(t,x') = g(0,x'+tv') , \quad \text{for } (t,x',v') \in (-\eta,\eta) \times B(x,\delta) \times B(v,\delta).
\end{equation}
Setting $U_x := \{ x - \frac{\eta}{2} v + \frac{\eta}{2} v', v' \in B(v, \delta)\}$, we remark that $U_x$ is an open set containing $x$. Moreover, for all $y \in U_x$, we have $y =  x - \frac{\eta}{2} v + \frac{\eta}{2} v'$ for some $v' \in B(v, \delta)$ so that, using~\eqref{gpropcons}, we have 
$$
g(0, y) = g\left(0,x - \frac{\eta}{2} v + \frac{\eta}{2} v'\right) = g\left(\eta/2, x- \frac{\eta}{2} v\right).
$$
Hence, $g(0, \cdot)$ is constant on $U_x$, and therefore constant on $\tilde\omega$ (since $\tilde\omega$ is connected).

Using the time translation invariance of~\eqref{transg}, we also have that for all $t\geq 0$, $g(t, \cdot)$ is locally constant on $\omega$. 
As a consequence, for all $t \geq 0$, $\frac{e^V}{\Mc} \, f(t,\cdot)$ is locally constant on $\omega$, which means that $\rho_i(t,x) = \kappa_i (t)$, i.e. $f_{|\omega}(t,x,v) = e^{-V(x)} \Mc(v) \sum_{i \in I} \kappa_i(t) \mathds{1}_{\omega_i}(x,v)$. Since $f$ satisfies the transport equation \eqref{eqhypo1} on $\R^+ \times \omega_i$, we infer that $\kappa_i$ is constant, so that $f_{|\omega}(t,x,v) =f_{|\omega}(0,x,v)= e^{-V(x)} \Mc(v) \sum_{i \in I} \kappa_i \mathds{1}_{\omega_i}(x,v)$. 
Using again the transport equation \eqref{eqhypo1}, we deduce that $\frac{e^V}{\Mc} \, f(t,x,v) = \frac{e^V}{\Mc} \, f(0,\cdot) \circ \phi_{-t}(x,v)$. Since there is only one equivalence class for $\Bumpeq$, we first  deduce that $f= \kappa  e^{-V(x)} \Mc(v) $ on $\omega$ and then that $f= \kappa  e^{-V(x)} \Mc(v) $ on $\bigcup_{s\in \R^+}\phi_{-s}(\omega)$. Since $\omega$ satisfies a.e.i.t. GCC, this is a full measure set and we deduce that $f =\kappa e^{-V} \Mc(v)$. Since  $\int_{\T^d \times \R^d} f \, dv \,dx=0$, necessarily $f=0$.

\bigskip

This concludes the proof of Theorem~\ref{thmconv-intro}.

 \begin{rque}
Note that the proof of $(i.) \implies (iii.)$ relies on the maximum principle for the linear Boltzmann equation~\eqref{B} (i.e. the $\LL^\infty$ bound). This was in particular useful to prevent loss of mass at infinity for the sequence of solutions under study. This will turn out to be also very useful to overcome another issue in the proof of the analogous theorem in the case of a bounded domain of $\R^d$ with specular reflection, see Section~\ref{convboun}. 
 \end{rque}

 \bigskip
 
 If $\omega$ satisfies the Geometric Control Condition of Defintion~\ref{def: GCC}, we can actually show a slightly stronger property than the Unique Continuation Property (with the same proof as $(ii.)$ implies $(i.)$ above, up to some slight modifications), that we state for convenience as a Proposition.
 
 \begin{prop}
  \label{remUCP}
 Assume that $(\omega,T)$ satisfies the Geometric Control Condition. 
 
If $f \in C^0_t(\LL^2)$ is a solution to
 \begin{equation*}
 \left\{
 \begin{aligned}
&\partial_t f + v \cdot \nabla_x f - \nabla_x V \cdot \nabla_v f= 0, \\
&C(f) =0 \text{  on  }  I \times \omega,
\end{aligned}
\right.
\end{equation*}
where $I$ is an interval of time of length larger than $T$,
then $f= \left(\int_{\T^d \times \R^d} f \, dv \,dx\right) \frac{e^{-V}}{\int_{\T^d} e^{-V} \, dx} \Mc(v)$.

 \end{prop}

 This will be useful for the proof of Theorem \ref{thmexpo-intro}.

 \subsection{Proof of Theorem~\ref{thmconvgene-intro}}
 \label{sec:conv2}

We start by describing the vector space of stationary solutions of the linear Boltzmann equation \eqref{B},
when the associated set $\omega$ satisfies the a.e.i.t. GCC.

For the sake of readability, we set here 
$$\tilde\Omega:= \bigcup_{s\in \R^+}\phi_{-s}(\omega).$$ 
We denote $\tilde\Omega= \cup_{i \in I} \Omega_i$ the partition of $\tilde\Omega$ in connected components. 
We write $([\Omega_j])_{j\in J}$ the equivalence classes for the equivalence relation $\sim$. We denote for all $j \in J$
\begin{equation}
\label{def-Uj}
U_j := \bigcup_{\Omega' \in [\Omega_j]} \Omega'.
\end{equation}

We have the following description of the vector space of stationary solutions of \eqref{B}.
    \begin{prop}
    \label{cardinal}
Assume that $\omega$ satisfies the a.e.i.t. GCC. Then a Hilbert basis of the subspace of stationary solutions to the linear Boltzmann equation~\eqref{B} (or, equivalently of $\ker (A)$, where $A$ is the linear Boltzmann operator defined in~\eqref{boltzop}) is given by the family $(f_j)_{j \in J}$, with
\begin{equation}
\label{def-fj}
f_j = \frac{\mathds{1}_{U_j} e^{-V}\Mc}{\| \mathds{1}_{U_j} e^{-V} \Mc \|_{\LL^2}}.
\end{equation}
In particular, the cardinality of the set of equivalence classes for $\sim$ is equal to the dimension of the vector space of stationary solutions to the linear Boltzmann equation \eqref{B}, i.e.
$$
\dim (\ker (A)) = \sharp(\CC(\tilde\Omega) / \sim) =  \sharp(\CC(\omega) / \Bumpeq).
$$
    \end{prop}     
  
 We can introduce a {\bf generalized} Unique Continuation Property, as follows.
  \begin{deft}
 \label{def:UCPgene}
We say that the set $\omega$ satisfies the {\bf generalized} Unique Continuation Property if the only solutions $f \in C^0_t(\LL^2)$ to
 \begin{equation}
 \label{eq:UCPgene}
 \left\{
 \begin{array}{l}
\partial_t f + v \cdot \nabla_x f - \nabla_x V \cdot \nabla_v f= 0, \\
C(f) = 0,
\end{array}
\right.
\end{equation}
are of the form  $f=\sum_{j\in J}  \langle f, f_j \rangle_{\LL^2} \,  f_j =\sum_{j \in J} \frac{1}{\| \mathds{1}_{U_j} e^{-V} \Mc \|_{\LL^2}}  \left( \int_{U_j} f \, dv dx\right) f_j$.
 \end{deft}

  We can now state the precise version of Theorem~\ref{thmconvgene-intro}:  
   \begin{thm}
\label{thmconv-general}
  We keep the notations of Proposition~\ref{cardinal}. \Black The following statements are equivalent.
\begin{enumerate}[(i.)]

\item The set $\omega$ satisfies the {\bf generalized} Unique Continuation Property.

\item The set $\omega$ satisfies the a.e.i.t. GCC.

\item  For all $f_0 \in \LL^2(\T^d \times \R^d)$, denote by $f(t)$ the unique solution to \eqref{B} with initial datum $f_0$. We have
\begin{equation}
 \label{convergeto0-general}
 \left\|f(t)-Pf_0 \right\|_{\LL^2} \to_{t \to +\infty} 0,
 \end{equation}
where
\begin{equation}
\label{def-equimulti}
P f_0 (x,v) = \sum_{i\in J}  \frac{1}{\| \mathds{1}_{U_j} e^{-V} \Mc \|_{\LL^2}} \left( \int_{U_i} f_0 \, dv dx\right)f_j  ,
\end{equation}
with $(U_j)_{j \in J}$ defined in~\eqref{def-Uj} and $(f_j)_{j \in J}$ defined in~\eqref{def-fj}.

\item  For all $f_0 \in \LL^2(\T^d \times \R^d)$, there exists a stationary solution $Pf_0$ of~\eqref{B} such that we have
\begin{equation}
 \label{convergeto0-general'}
 \left\|f(t)-Pf_0 \right\|_{\LL^2} \to_{t \to +\infty} 0,
 \end{equation}
where $f(t)$ the unique solution to \eqref{B} with initial datum $f_0$. .
 \end{enumerate}
 \end{thm}

Note that Theorem~\ref{thmconv-intro} is a particular case of Theorem~\ref{thmconv-general}, when there is only one equivalence class for $\sim$ (or equivalently for $\Bumpeq$).

This section is organized as follows: in Paragraph~\ref{para1}, we prove Proposition \ref{cardinal}. Then, in Paragraph~\ref{para1}, we prove that (iv.) implies (ii.). Finally, in Paragraph~\ref{para3}, we show that (i.)--(ii.)--(iii.) are equivalent. Since (iii) implies (iv.) is straightforward, this will conclude the proof of Theorem~\ref{thmconv-general}.

\subsubsection{Proof of Proposition \ref{cardinal}}
\label{para1}

 We start by checking that for all $i \in J$, $f_j$ is a stationary solution of~\eqref{B}. From Lemma~\ref{leminvarflow}, we know that for any connected component $\Omega'$ of $\tilde\Omega$ and any $t\geq0$, $\phi_{-t} (\Omega') = \Omega'$ almost everywhere. Thus for all $t\geq0$, $\phi_{-t} (U_j) = U_j$ almost everywhere. The function $f_j$ hence cancels the kinetic transport part.
 
We now check that $f_j$ cancels the collision operator, i.e. $C(\mathds{1}_{U_i} e^{-V}\Mc)=0$. We use for this Lemma~\ref{collannule}.

Denote $\omega = \cup_{i \in I} \omega_i$ the partition of $\omega$ in connected components. Let $i \in I$. If $\omega_i \cap U_j =\emptyset$, then $f_j=0$ on $\omega_i$. If $\omega_i \cap U_j \neq \emptyset$, then there exists $\Omega' \in [\Omega_j]$ such that $\omega_i \cap \Omega' \neq \emptyset$.
Since $\Omega'$ is a connected component of $\bigcup_{s\in \R^+}\phi_{-s}(\omega)$, we have $\omega_i \subset \Omega'$ and thus $f_j=\frac{ e^{-V(x)}\Mc(v)}{\| \mathds{1}_{U_j} e^{-V} \Mc \|_{\LL^2}} = \rho_j(x) \Mc(v)$ on $\omega_i$, with $\rho_j(x)=\frac{ e^{-V(x)}}{\| \mathds{1}_{U_j} e^{-V} \Mc \|_{\LL^2}}$. 

Assume now that there exist $k,l \in I$ and $x\in \T^d$ such that $\omega_k(x) \Rc_k^x \, \omega_l(x)$. Denote by $\Omega' \in  [\Omega_j]$, the connected component of $\tilde\Omega$ such that $\omega_l \subset \Omega'$, and $\Omega''$ the connected component of $\tilde\Omega$ such that $\omega_k \subset \Omega''$. Note then that $\omega_k(x) \Rc_k^x \, \omega_l(x)$ implies $\Omega' \sim \Omega''$. By definition of $f_j$, this implies that $\rho_k(x)= \rho_l(x)$.
Therefore, by Lemma~\ref{collannule}, we infer that the function $f_j$ cancels the collision operator. We deduce that $f_j$ is a stationary solution of~\eqref{B}.   
 
 Furthermore, since the supports of the $(f_j)_{j \in J}$ are disjoint,  $(f_j)_{j \in J}$ is an orthonormal family of $\LL^2$.
 
Finally, let $\varphi$ be a stationary solution of \eqref{B}. Then $\varphi$ satisfies
 $$
 v\cdot \na_x \varphi - \na_x V \cdot \na_v \varphi = C(\varphi).
 $$
 Taking the $\LL^2$ scalar product with $\varphi$, we deduce that $D(\varphi)= \langle C(\varphi), \varphi \rangle_{\LL^2}=0$, so that by Lemma \ref{collannule}, ${C}(\varphi)=0$.
Then, with the same analysis as the proof of $(ii.) \implies (i.)$ in Theorem \ref{thmconv-intro}, we deduce that $\frac{e^{V}}{\Mc}\varphi$ is  constant on each $U_j$.
Using the fact that $\omega$ satisfies a.e.i.t. GCC, we deduce that we can write
$$
 \varphi= \sum_{j\in J} \lambda_j \mathds{1}_{U_j} e^{-V} \Mc(v), \quad \lambda_j \in \R,
 $$
 that is
 $$
 \varphi= \sum_{j\in J}  \langle \varphi, f_j \rangle_{\LL^2} \,  f_j,
 $$
and this concludes the proof.

\subsubsection{Necessity of the a.e.i.t. Geometric Control Condition} 
\label{para2}
We prove here that (iv.) implies (ii.).

We argue by contradiction. Assume that the a.e.i.t. Geometric Control Condition  does not hold. Then  $\Ac:= \T^d \times \R^d \setminus \bigcup_{s  \in \R^+}\phi_{-s}(\omega)$ has positive Lebesgue measure. 

We set 
$$
 f_0 (x,v) = \Psi(x) \mathds{1}_{\Ac} (x,v) \, e^{-V(x)} \Mc(v),
 $$
 with $\Psi$ to be determined later.
We define
\begin{equation}
\label{f0psi}
f(t,x,v) := f_0 \circ \phi_{-t}(x,v) = \Psi \circ \phi_{-t}(x,v)  \mathds{1}_{\phi_t(\Ac)}(x,v) \, e^{-V(x)} \Mc(v)
\end{equation}
which satisfies, by construction,
\begin{equation}
\label{f0psieq}
\pa_t f + v \cdot \na_x f - \na_x V \cdot \na_v f =0.
\end{equation}
Note that for all $t\geq0$, we have $\phi_{t}(\Ac) \cap \omega = \emptyset$ since  
$\Ac\cap \phi_{-t}(\omega) = \emptyset$, which yields $C(f(t)) = 0$ for all $t \geq 0$ and thus $f$ is also a solution of~\eqref{B}. 
We now fix $\Psi$ in order to ensure that $f(t)$ is not stationary.
\begin{itemize}
\item If $(v \cdot \na_x - \na_x V\cdot \na_v)(\mathds{1}_{\Ac}) \neq 0$, then we take $\Psi =1$.
\item If $(v \cdot \na_x  - \na_x V\cdot \na_v )(\mathds{1}_{\Ac})  = 0$, we take for $\Psi$ a Morse function on $\T^d$, so that, in particular, $\Psi$ is smooth and $\nabla \Psi(x) \neq 0$ for almost every $x \in \T^d$.
Note that with such a function $\Psi$, we have $f_0 \in \LL^2$. We compute
$$
\left[v \cdot \na_x \Psi - \na_x V \cdot \na_v \Psi\right] \mathds{1}_{\Ac} = (v \cdot \na_x \Psi(x) )\mathds{1}_{\Ac}(x,v).
$$
Therefore for almost all $(x,v) \in \Ac$, this is not null, which shows that $f(t)$ is not stationary.
\end{itemize}

Finally if there existed a stationary solution $Pf_0$ of~\eqref{B} such that
$$
\| f(t)- Pf_0 \|_{\LL^2} \to_{t \to + \infty} 0,
$$
then since for all $t\geq 0$, $f(t)$ is supported in $\Ac$, we also have $P f_0$ supported in $\Ac$. Thus $P f_0$ cancels the collision operator, i.e. $C(Pf_0)=0$. We deduce that $f(t)- Pf_0$ satisfies 
$$
\pa_t (f -Pf_0)+ v \cdot \na_x (f -Pf_0) - \na_x V \cdot \na_v (f -Pf_0) =0 .
$$
This yields for all $t\geq 0$,
$$
\| f(t) -Pf_0\|_{\LL^2} = \| f_0 -Pf_0\|_{\LL^2} . 
$$
Moreover,  the solution $f$ defined in~\eqref{f0psi} is a non-stationary solution of~\eqref{f0psieq} according to the definition of $\Psi$. In conclusion, we have $ f_0 \neq Pf_0$. This yields $\| f_0 -Pf_0\|_{\LL^2} >0$ so that we cannot have $\| f(t)- Pf_0 \|_{\LL^2} \to_{t \to + \infty} 0$. This concludes the proof.

\subsubsection{End of the proof of Theorem~\ref{thmconv-general}}
\label{para3} 
  
 We have the following key lemma.
 
 \begin{lem} 
 \label{lem-proj}
 Let $f$ be a solution in $C^0_t(\LL^2)$ of \eqref{B}. Then for all $j \in J$, 
 $$
 \frac{d}{dt} \langle f, f_j\rangle_{\LL^2} =\frac{1}{\| \mathds{1}_{U_j}e^{-V} \Mc \|_{\LL^2}} \frac{d}{dt} \int_{U_j} f \, dv dx =0.
 $$
 \end{lem}
 
 \begin{proof}[Proof of Lemma~\ref{lem-proj}]  Let $f$ be a solution in $C^0_t(\LL^2)$ of \eqref{B}; denote by $f_0$ its initial datum. Let $j \in J$. We take the $\LL^2$ scalar product with $f_j$ in~\eqref{B} to obtain
\begin{align*}
\frac{d}{dt} \langle f, f_j\rangle_{\LL^2} &= - \langle (v\cdot \na_x -\na_x V\cdot \na_v) f, f_j\rangle_{\LL^2} 
+ \langle C(f), f_j\rangle_{\LL^2}  \\
 &= \langle  f, (v\cdot \na_x -\na_x V\cdot \na_v) f_j\rangle_{\LL^2} 
+ \langle f, C^*(f_j) \rangle_{\LL^2}
\end{align*}
where $C^*$ is the collision operator of collision kernel defined by
$$
C^*(g) = \int_{\R^d} \left[\tilde{k}(x,v,v') \Mc(v) g(v') - \tilde{k}(x,v',v) \Mc(v')  g(v)\right] \, dv' 
$$
with $\tilde{k}(x,v,v') = \frac{k(x,v,v')}{\Mc(v')}$. This follows from Property {\bf A2} satisfied by $k$.

We then use the following two facts.

 \noindent {\bf 1.} We have $(v\cdot \na_x -\na_x V\cdot \na_v) f_j =0$ (see the proof of Proposition~\ref{cardinal}).

\noindent {\bf 2.} We have $C^*(f_j)=0$. This follows from Property {\bf A2}, the fact that $C(f_j)= 0$ (see again the proof of Proposition~\ref{cardinal}) and Lemma~\ref{collannule}.

We conclude that $\frac{d}{dt} \langle f, f_j\rangle_{\LL^2}=0$. 

\end{proof}
 
  We therefore infer that if $f(t)$ satisfies~\eqref{B} with an initial datum $f_0$, then for all $t\geq 0$,
 $$
 \int_{U_j} f(t) \, dv \, dx =  \int_{U_j} f_0 \, dv \,  dx
 $$
 Equipped with this result, we prove the equivalence between (i.)--(ii.)--(iii.) exactly as for Theorem~\ref{thmconv-intro}, with only minor adaptations.

 \section{Application to particular classes of collision kernels}
 \label{secexamples}
 
 In this section, we introduce the following classes of collision kernels to illustrate the main results of the previous sections. We then draw consequences of the additional assumptions made in these examples. 

\bigskip
  \noindent {\bf E3.}  
   Let $k$ be a collision kernel verifying {\bf A1}--{\bf A3}. Let $\omega$ be the set where collisions are effective, defined in \eqref{omega}. We moreover require that for all $(x,v), (x,v') \in \omega$, there exist $N \in \N^*$ and a ``chain'' $v_1,\cdots, v_N \in \R^d$ such that the following hold.
   \begin{itemize}
   \item For all $i$, $1\leq i \leq N$, $(x,v_i) \in \omega$.
   \item The points $(x,v)$ and $(x,v_1)$ belong to the same connected component of $\omega$.
   \item The points $(x,v')$ and $(x,v_N)$ belong to the same connected component of $\omega$.
   \item For all $i$, $1\leq i \leq N-1$, we have
   $$
   k(x,v_i,v_{i+1}) > 0 \text{ or }    k(x,v_{i+1},v_i) > 0.
   $$
   \end{itemize}

 As a subclass of {\bf E3}, we have
 
   \noindent {\bf E3'.}  
   Let $k$ be a collision kernel verifying {\bf A1}--{\bf A3}. We require that for all $y \in p_x(\omega)$ (where $p_x(\omega)$ is the projection of $\omega$ on $\T^d$), the set $p_x^{-1}( \{y\})$ is included in one single connected component of $\omega$.
 
 A trivial subclass of {\bf E3'} is the case where $\omega$ is connected. Another subclass of {\bf E3'} is given in the following example.
 
    \noindent {\bf E3''.}  
   Let $k$ be a collision kernel verifying {\bf A1}--{\bf A3}.
 We require that 
 $$
 \omega = \omega_x \times \R^d,
 $$
 where $\omega_x$ is an open subset of $\T^d$.

 Remark that {\bf E2} is a subclass of {\bf E3''}. 
 
 \bigskip

In what follows, we explain the interest of these classes of collision kernels regarding the geometric definitions introduced before.

\subsection{The case of collision kernels in the class {\bf E3}}

The interest of {\bf E3} lies in the simple description of the kernel of the associated collision operator $C$.

Using  Lemma~\ref{collannule} and  the ``chain'' in the definition of a collision kernel in {\bf E3}, we deduce the following result.
\begin{lem}
\label{lemE3}
Let $k$ be a collision kernel in the class {\bf E3}. Let $T \in (0, +\infty]$ and assume that $f \in L^2(0,T; \LL^2)$ satisfies ${C}(f(t))=0$ for almost every $t \in [0,T]$. 
Then there is a function $\rho \in L^2(0,T ;L^2(\T^d))$ such that
$$
f=\rho (t,x)\Mc(v) \quad \text{on  }[0,T] \times \omega.
$$
Reciprocally, any function $f$ satisfying this property satisfies $C(f) =0$.
\end{lem}
In other words, the kernel of the associated collision operator is equal to the set of functions which are Maxwellians on $\omega$:
\begin{equation}
 \label{eqkerC}
 \ker(C) = \{ f \in \LL^2 , f_{|\omega} = \rho(x)\Mc(v) \} .
 \end{equation}
We recall that this property, which is usual in the non degenerate case $\omega= \T^d \times \R^d$, is not true in general for collision kernels satisfying merely {\bf A1}, {\bf A2} and {\bf A3}.

This allows us to reformulate in a very simple way the Unique Continuation Property for collision kernels in {\bf E3}.
 \begin{lem}
 \label{def:UCP-E3}
Let $k$ be a collision kernel in the class {\bf E3}. Then the set $\omega$ satisfies the Unique Continuation Property if and only if the only solution $f \in C^0_t(\LL^2)$ to
 \begin{equation}
 \label{eq:UCP-E3}
 \left\{
 \begin{array}{l}
\partial_t f + v \cdot \nabla_x f - \nabla_x V \cdot \nabla_v f= 0, \\
f = \rho(t,x) \Mc(v) \text{  on  } \R^+ \times \omega ,
\end{array}
\right.
\end{equation}
is $f= \left(\int_{\T^d \times \R^d} f \, dv \,dx\right) \frac{e^{-V}}{\int_{\T^d} e^{-V} \, dx} \Mc(v)$.
 \end{lem}

Conversely, using again Lemma~\ref{collannule}, we have the following result.
\begin{lem}
Let $k$ be a collision kernel satisfying {\bf A1}--{\bf A3}. If any function $f \in \LL^2$ canceling the collision operator has its restriction to $\omega$ satisfying
$$
f_{\vert \omega} = \mathds{1}_\omega \rho(x) \Mc(v)
$$
for some $\rho \in L^2(\T^d)$, then $k$ necessarily belongs to the class {\bf E3}.
\end{lem}
Therefore, {\bf E3} is the largest class of collision kernels such that the kernel of the associated collision operator is equal to the set of functions which are Maxwellians on $\omega$, i.e. for which~\eqref{eqkerC} holds.

 \subsection{The case of collision kernels in the class {\bf E3'}}
 \label{secE3'}

To explain the interest of {\bf E3'}, let us introduce now another geometric condition:

\begin{enumerate}
\item[(iv.)] The set $\omega$ satisfies the a.e.i.t. GCC and $\bigcup_{s\in \R^+}\phi_{-s}(\omega)$ is connected.
\end{enumerate}

This condition is compared to other geometric conditions in Appendix~\ref{appaeitgccconnected}.
It has to be confronted to the geometric condition of Theorem~\ref{thmconv-intro} (rephrased using Lemma~\ref{equiv-equiv}):
\begin{enumerate}
\item[(ii.)] The set $\omega$ satisfies the a.e.i.t. GCC and there is only one equivalence class for $\sim$.
\end{enumerate}

In what follows, we shall adopt the notations of Theorem~\ref{thmconv-intro}.
It is clear that $(iv.)$ implies $(ii.)$, as $(iv.)$ means that there is a single equivalence class for an equivalence relation defined as in Definition~\ref{def-sim-o} with $\Rc_\phi$ only. However, it is false in general that $(i.)$--$(iii.)$ and $(iv.)$ are equivalent; see Proposition~\ref{prop-cex-connected} below for an example of collision kernel such that $(i.)$ is satisfied, but not $(iv.)$.

\begin{prop}
\label{prop-cex-connected}
For $V=0$, there exists a collision kernel $k$ in the class {\bf E1} and {\bf E3}, for which $(ii)$ holds, but not $(iv)$.
\end{prop}

\begin{proof}[Proof of Proposition~\ref{prop-cex-connected}]
Consider $(x,v) \in \T \times \R$ (a similar construction in higher dimension could also be performed). We take a function $\varphi \in C^0(\R)$, such that $\varphi(0)=0$ and $\varphi(v) >0$ for all $v \in \R \setminus \{0\}$.
Define
$$
k(x,v',v)= \Mc(v) \varphi(v) \varphi(v').
$$
By construction, {\bf A1} and {\bf A3} are satisfied and we notice that $\tilde{k}$ is  symmetric (so that ${k}$ is in {\bf E1}).
We can also readily check that $k$ is in {\bf E3} (with $N=2$).

We have
$$
\omega := \{\T \times \R^-_*\}  \cup  \{\T \times \R^+_*\}
$$
and so 
$$
\bigcup_{t\geq 0} \phi_{-t}(\omega) =\{ \T \times \R^-_*\} \cup  \{\T \times \R^+_*\},
$$
from which we deduce that a.e.i.t. GCC is satisfied, but $\cup_{t\geq 0} \phi_{-t}(\omega)$ is not connected.

On the other hand, the set $\omega$ satisfies the unique continuation property. Indeed, let $f$ satisfying
\begin{equation*}
\label{f-trans-eq}
\pa_t f + v \cdot \na_x f =0, \quad \forall (x,v) \in \T \times \R
\end{equation*}
and $C(f) =0$. Using Lemma~\ref{lemE3} and the definition of $k$, we deduce that $f= \rho(t,x) \Mc(v)$ on $\R^+ \times \T \times \R\setminus \{0\}$, and thus almost everywhere in $\R^+ \times \T \times \R$. 

As $f$ satisfies the transport equation~\eqref{f-trans-eq}, this implies that $f = C \Mc(v)$ for some $C>0$ and we conclude that the unique continuation property holds.

Therefore, by Theorem~\ref{thmconv-intro}, we deduce that $(ii.)$ holds.

\end{proof}

However, when restricting to collision kernels in the class  {\bf E3'}, (ii.) and (iv.) become equivalent.
\begin{prop}
\label{prop-E3'}
Let $k$ be a collision kernel in the class {\bf E3'}.
Then $(i.)$--$(iv.)$ are equivalent.

\end{prop}

\begin{proof}[Proof of Proposition~\ref{prop-E3'}]
We consider $k$ a collision kernel in the class {\bf E3'}.
Assume that $(ii.)$ holds. The aim is to prove that $(iv.)$ holds. By contradiction, assume that there are at least two connected components $\Omega_1$, $\Omega_2$ of $\bigcup_{t\geq 0} \phi_{-t}(\omega)$.

By $(ii.)$, $\Omega_1$ and $\Omega_2$ belong to the same equivalence class for $\sim$. Thus, there exist $x,v_1,v_2$ with $(x,v_1) \in \Omega_1$, $(x,v_2) \in \Omega_2$ and
$$
k(x,v_1,v_2)>0 \text{ or } k(x,v_2,v_1)>0.
$$
But since $k$ is in the class {\bf E3'}, the set $p_x^{-1}( \{x\})$ is included in one connected component of $\omega$. Thus we can not have $(x,v_1) \in \Omega_1$ and $(x,v_2) \in \Omega_2$. This is a contradiction and this concludes the proof.

\end{proof}
More generally, we observe that for collision kernels in {\bf E3'}, the equivalence classes for $\sim$ are exactly the connected components of $\bigcup_{t\geq 0} \phi_{-t}(\omega)$. Thus Theorem~\ref{thmconv-general} can be reformulated as follows.

 \begin{coro}
\label{thmconv-general-E3'}
 Let $k$ be a collision kernel in the class {\bf E3'}. The following statements are equivalent.
\begin{enumerate}

\item The set $\omega$ satisfies the {generalized} Unique Continuation Property.

\item The set $\omega$ satisfies the a.e.i.t. GCC.

\item Let $(\Omega_i)_{i \in I}$ be the connected components of $\cup_{t\geq 0} \phi_{-t}(\omega)$. For all $f_0 \in \LL^2(\T^d \times \R^d)$, denote by $f(t)$ the unique solution to \eqref{B} with initial datum $f_0$. We have
\begin{equation}
 \label{convergeto0-general-E3'}
 \left\|f(t)-P f_0\right\|_{\LL^2} \to_{t \to +\infty} 0,
 \end{equation}
where
$$
P f_0 = \sum_{i\in I}  \frac{1}{\| \mathds{1}_{\Omega_i} e^{-V} \Mc \|_{\LL^2}} \left( \int_{\Omega_i} f_0 \, dv dx\right)g_j,
$$
with for all $i \in I$,
$$
g_i = \frac{\mathds{1}_{\Omega_i} e^{-V}\Mc}{\| \mathds{1}_{\Omega_j} e^{-V} \Mc \|_{\LL^2}}.
$$

\end{enumerate}

\end{coro}

 \bigskip
 
We close this section by exhibiting an example of collision kernel in {\bf E3'}, for which Corollary~\ref{thmconv-general-E3'} (and thus Theorem~\ref{thmconv-general}) are relevant.

We restrict ourselves to the case $\T \times \R$ (this can be easily adapted to higher dimensions). We consider the free transport case, i.e. $V=0$. We identify $\T$ to $[-1/2,1/2)$. Consider $\alpha \in C^0(\T)$ supported in $[-1/2,0)$ and $\beta \in C^0(\T)$ supported in $[0,1/2)$.

Let $\varphi \in L^\infty\cap C^0(\R)$ such that $\varphi>0$ on $\R^-_*$ and $\varphi =0$ on $\R^+$. Likewise, let  $\Psi \in L^\infty\cap C^0(\R)$ such that $\Psi>0$ on $\R^+_*$ and $\Psi =0$ on $\R^-$.
We define the collision kernel
$$
k(x,v,v') :=  \left[ \alpha(x) \varphi(v) \varphi(v') + \beta(x) \Psi(v) \Psi(v') \right] \Mc(v').
$$
Note that $\tilde{k}(x,v,v') = \alpha(x) \varphi(v) \varphi(v') + \beta (x)\Psi(v) \Psi(v')$ is symmetric in $v$ and $v'$, and belongs to $L^\infty$. Thus $k$ is in the class {\bf E1}. Furthermore, we readily check $k$ is in the class {\bf E3'}.

Moreover, we have $\omega = \{\{\alpha>0\} \times \R^-_*\}  \cup \{\{\beta>0\} \times \R^+_*\}$, and 
$$\bigcup_{s\in \R^+}\phi_{-s}(\omega) = \{\T\times \R^-_*\}  \cup \{\T \times \R^+_*\}.$$ 
Thus $\omega$ satisfies the a.e.i.t. GCC but $\bigcup_{s\in \R^+}\phi_{-s}(\omega)$ is not connected. 

The basis of the subspace of stationary solutions of~\eqref{B} is given by $(f_j)_{j=1,2}$ with
$$
f_1 = \frac{\mathds{1}_{\T\times \R^-_*} e^{-V}\Mc}{\| \mathds{1}_{\T\times \R^-_*} e^{-V} \Mc \|_{\LL^2}}, \quad 
f_2 = \frac{\mathds{1}_{\T\times \R^+_*} e^{-V}\Mc}{\| \mathds{1}_{\T\times \R^+_*} e^{-V} \Mc \|_{\LL^2}}.
$$

\subsection{The case of collision kernels in the class {\bf E3''}}
\label{secE3pp}
We finally study collision kernels in the class {\bf E3''}.  Because of the remarkable properties of the geodesic flow on the torus $\T^d$, the following holds.

 \begin{lem}
 \label{lemUCP1} 
 Suppose that $V = 0$ and that  the collision kernel belongs to the class {\bf E3''}. 
 Then $\omega$ satisfies a.e.i.t. GCC.
 \end{lem}

 \begin{proof}[Proof of Lemma \ref{lemUCP1}]

 Define $T_v^t : \, x \mapsto x + t \, v$; then $(T_v^t \, x)_{t \geq 0}$ is dense in $\T^d$ for almost every $(x,v) \in \T^d \times \R^d$ (with respect to the Lebesgue measure). 
 This proves the lemma.
 \end{proof}

 We deduce a proof of Proposition~\ref{coroUCP}.
  \begin{proof}[Proof of Proposition \ref{coroUCP}]Take $\omega_x^0$ a connected component of $\omega_x$. According to Lemma~\ref{lemUCP1}, $\omega_x^0 \times \R^d$ satisfies a.e.i.t. GCC. The result then follows from Proposition~\ref{aeitgccconnected}, $(i.) \Rightarrow (iii.)$, and Theorem \ref{thmconv-intro}.
  \end{proof}

  \section{Characterization of exponential convergence}
\label{expocon}
Let us first briefly recall why $C^-(\infty)$ is well-defined (see~\cite{Leb}). We can first define for all $T>0$,
 $$
 C^-(T) :=  \text{inf}_{(x,v) \in \T^d \times \R^d} \frac{1}{T} \int_0^T \left( \int_{\R^d} k(\phi_t (x,v), v')\, dv'\right)\, dt,
 $$
 which is well-defined (and nonnegative) since $k$ is non-negative. We then remark that the function $T \mapsto - T C^-(T)$ is subadditive. This entails that $C^-(\infty) = \lim_{T \to +\infty} C^-(T)$ exists (one proves similarly that $C^+(\infty)$ is well defined).

In this section, we assume that $k$ satisfies {\bf A3'} and provide the proof of Theorem \ref{thmexpo-intro}.
To this end, we first prove that (a.) and (b.) are equivalent, and finally that (c.) implies (a.). This will conclude the proof, noticing that (b.) implies (c.) is straightforward. One can also readily check that in the proof of (a.) implies (b.), the assumption {\bf A3'} is not used.

 \subsection{Proof of the equivalence between (a.) and (b.)}
  
Since the equation \eqref{B} is linear, if $f(t)$ satisfies~\eqref{B} then
 $$
 g(t) := f(t)-\int_{\T^d \times \R^d} f(0) \, dv \,dx \frac{e^{-V}}{\int_{\T^d} e^{-V} \, dx} \Mc(v),
 $$
 is still a solution to \eqref{B}. Thus we can deal only with initial data which have zero mean. 

By conservation of the mass, the Boltzmann equation~\eqref{B} is well-posed in the space 
$$\LL^2_0:= \left\{f \in \LL^2, \, \int f \, dv dx =0\right\}$$ 
and we can use Lemma~\ref{lemfondamental} for solutions in $\LL^2_0$, which yields that $(b.)$ is equivalent to

\begin{enumerate}
\item[(b'.)] There exists $T>0$ and $K>0$ such that for all $f_0 \in \LL^2_0$, the associated solution $f$ to \eqref{B} satisfies
$$
K \int_0^T D(f(t)) \, dt \geq  \| f_0\|_{\LL^2}^2.
$$
\end{enumerate}

 \bigskip

 We first prove that $(a.)$ implies $(b'.)$, then that $(b'.)$ implies $(a.)$

 \noindent ${(a.) \implies (b'.)}$ Assume that $(a.)$ holds. 
  
  We argue by contradiction. Denying $(b'.)$ is equivalent to assume for all $T>0$ and all $C>0$, the existence of $g_0^{C,T} \in \LL^2_0$, such that
 $$
 C \int_0^T D(g^{C,T}(t)) \, dt <  \| g_0^{C,T}\|_{\LL^2}^2,
 $$
 where $g^{C,T}(t)$ is the unique solution to \eqref{B} with initial datum $g_0^{C,T}$.
 For all $n \in \N^*$, there exists $g_{0,n} \in \LL^2_0$, such that
  \begin{equation}
  \label{dissipgn}
  \int_0^n D(g_n(t)) \, dt <  \frac{1}{n} \| g_{0,n}\|_{\LL^2}^2.
 \end{equation}
where $g_n(t)$ is the unique solution to \eqref{B} with initial datum $g_{0,n}$. Furthermore, by linearity of \eqref{B}, we can normalize the initial data so that for all $n \in \N^*$, 
  \begin{equation}
  \label{normalization}
 \| g_{0,n}\|_{\LL^2} =1.
  \end{equation}
 Recall that by Lemma \ref{lemdissip}, we have, for all $t\geq0$,
 \begin{equation}
 \label{eqdiss}
 \| g_n(t) \|^2_{\LL^2} -  \| g_{0,n} \|^2_{\LL^2} = - \int_0^t D(g_n(s)) \, ds.
 \end{equation}
 In particular, the sequence $(g_n)_{n \in \N^*}$ is uniformly bounded in $L^\infty_t \LL^2$; thus, up to some extraction, we can assume that $g_n \rightharpoonup g$ weakly in $L^2_t \LL^2$. Let us prove that $g=0$. By linearity of \eqref{B}, $g$ still satisfies \eqref{B} since $g_n$ does.

 Note also that by conservation of the mass, for all $n \in \N$ and all $t \geq 0$,
     \begin{equation}
    \label{eqconsmas2}
 \int g_n(t) \, dv dx = \int g_{0,n} \, dv dx =0.
 \end{equation}

 Take $T'>0$ such that $(\omega,T')$ satisfies GCC. By \eqref{dissipgn}, we have $\int_0^{T'} D(g_n(t)) \, dt  \to 0$ and therefore by 
 weak lower semi-continuity, we deduce 
$$
    \left\|{D}(g) \right\|_{L^1(0,T')} \leq \liminf_{n\to +\infty}     \left\|{D}(g_n) \right\|_{L^1(0,T')}= 0.
$$
     As a consequence, by weak coercivity (see Lemma~\ref{collannule}),  we infer that ${C}(h)=0$ on $[0,T']$, and therefore $h$ satisfies the kinetic transport equation \eqref{eq:UCP}.
The Unique Continuation Property of Proposition~\ref{remUCP} then implies that
 \begin{equation}
 g= \int_{\T^d \times \R^d} g \, dv \,dx \frac{e^{-V}}{\int_{\T^d} e^{-V} \, dx} \Mc(v).
 \end{equation}
 Since $g_{n} \rightharpoonup g$ weakly in $L^2_t \LL^2$, using \eqref{eqconsmas2}, we obtain in particular  that
$$
\int_0^{T'} \left(\int g \, dv dx \right) dt =0.
$$
Since $\int_0^{T'} \left(\int g \, dv dx \right) dt = T'  \left(\int g(0) \, dv dx \right)$, we deduce that $\int g \, dv dx=0$, so that $g=0$.
 Therefore,  this leads to $g=0$.
 \Black

 Now, let us study the sequence of defect measures $\nu_n := |g_n|^2$ and $\nu_{0,n} := |g_{0,n}|^2$. Consider the equation \eqref{B} satisfied by $g_n$ and multiply it by $g_n$ to get:
        \begin{align}
        \label{eq:nun}
        \partial_t \nu_n &+ v \cdot \nabla_x \nu_n - \nabla_x V \cdot \nabla_v \nu_n \\ \nonumber
        &= 2 \left(\int_{\R^d} k(x,v' ,  v) g_n(v') \, dv'\right) g_n - 2 \left( \int_{\R^d} k(x,v ,  v')\, dv'  \right)  \nu_n
            \end{align}
            By Duhamel's formula, we infer
            \begin{multline*}
            \nu_n(t,x,v)= e^{-2 \int_0^t \int_{\R^d} k(\phi_{s-t}(x,v) ,  v')\, dv' \, ds   } \nu_{0,n} (\phi_{-t} (x,v)) \\
            + \int_0^t  2 \left(\int_{\R^d} k(X_{s-t}(x,v),v' ,  \Xi_{s-t}(x,v)) g_n(s,X_{s-t}(x,v),v') \, dv'\right) g_n(s, \phi_{s-t}(x,v)) \\
            \times e^{-2 \int_s^t \int_{\R^d} k(\phi_{\tau-t}(x,v) ,  v')\, dv' \, d\tau   } \, ds ,
               \end{multline*}
               and thus for all $t\geq 0$,
             \begin{multline}
             \label{duhamelnun}
           \| \nu_n(t)\|_{\LL^1} \leq \int_{\T^d \times \R^d}  e^{-2 \int_0^t \int_{\R^d} k(\phi_{s-t}(x,v) ,  v')\, dv' \, ds   } \nu_{0,n} (\phi_{-t}(x,v)) \frac{e^{V(x)}}{\Mc(v)} \, dv \, dx  \\
            +\int_{\T^d \times \R^d} \Bigg( \int_0^t  2 \left|\int_{\R^d} k(X_{s-t}(x,v),v' ,  \Xi_{s-t}(x,v)) g_n(s,X_{s-t}(x,v),v') \, dv'\right|   |g_n(s, \phi_{s-t}(x,v))|    \\
            \times e^{-2 \int_s^t \int_{\R^d} k(\phi_{\tau-t}(x,v) ,  v')\, dv' \, d\tau   } \, ds \Bigg)\frac{e^{V(x)}}{\Mc(v)} \, dv \, dx
               \end{multline}

By definition of $C^-(\infty)$, there exists $T_0>0$ large enough such that for all $t\geq T_0, C^-(t)\geq C^-(\infty)/2>0$.
We have, after the change of variables $\phi_{-t}(x,v) \mapsto (x,v)$, which has unit Jacobian (recall also that the hamiltonian is left invariant by this transform), for all $t\geq T_0$,
  \begin{align*}
&\int_{\T^d \times \R^d}  e^{-2 \int_0^t \int_{\R^d} k(\phi_{s-t}(x,v) ,  v')\, dv' \, ds   } \nu_{0,n} (\phi_{-t} (x,v)) \frac{e^{V(x)}}{\Mc(v)} \, dv \, dx \\
&=\int_{\T^d \times \R^d}  e^{-2 \int_0^t \int_{\R^d} k(\phi_{s}(x,v) ,  v')\, dv' \, ds   } \nu_{0,n} (x,v) \frac{e^{V(x)}}{\Mc(v)} \, dv \, dx\\
    &  \leq e^{-t C^-(t)}  \| \nu_{0,n} \|_{\LL^1}  \leq e^{-t C^-(\infty)/2}  \| \nu_{0,n} \|_{\LL^1}
      \end{align*}  
 and thus we can choose  $T_1 \geq T_0$ large enough such that the left-handside is less than $\frac{1}{4}   \| \nu_{0,n} \|_{\LL^1}$ for $t=T_1$.

       On the other hand, since $g_n \rightharpoonup 0$, by the averaging lemma of Corollary \ref{lemmoyenne}, we deduce that in $L^2(0,T_1;\LL^2)$,
\begin{equation}
\label{eq-compact}
 \left(\int_{\R^d} k(x,v' ,  v) g_n(v') \, dv'\right) \to 0 .
\end{equation}
 Hence, by the weak/strong convergence principle,
   $$
 \left(\int_{\R^d} k(x,v' ,  v) g_n(v') \, dv'\right) g_n \to 0,
 $$
in $L^1(0,T_1;\LL^1)$.
Therefore, using the change of variables $\phi_{s-t}(x,v) \mapsto (x,v)$, which has unit Jacobian, we infer that
\begin{multline*}
 \int_{\T^d \times \R^d} \Bigg( \int_0^t  2 \left|\int_{\R^d} k(X_{s-t}(x,v),v' ,  \Xi_{s-t}(x,v)) g_n(s,X_{s-t}(x,v),v') \, dv'\right|   |g_n(s, \phi_{s-t}(x,v))|    \\
            \times e^{-2 \int_s^t \int_{\R^d} k(\phi_{\tau-t}(x,v) ,  v')\, dv' \, d\tau   } \, ds \Bigg)\frac{e^{V(x)}}{\Mc(v)} \, dv \, dx \to 0
\end{multline*}
and thus, coming back to~\eqref{duhamelnun}, for $n$ large enough, we finally obtain (we recall that $ \| \nu_{0,n} \|_{\LL^1}^2 =1$)
\begin{equation}
\label{eq:beforecontract}
           \| \nu_n(T_1)\|_{\LL^1}^2 \leq \frac{1}{2} .
\end{equation}

 But integrating with respect to time \eqref{eqdiss} and using \eqref{normalization}-\eqref{dissipgn}, we also have
 \begin{align*}
\| \nu_n({T_1} )\|_{\LL^1}^2  &= 1 -   \int_0^{T_1}  D(g_n(s)) \, ds \\
&\geq \frac{3}{4} \text{  for  } n \text{  large enough},
  \end{align*} 
 which is a contradiction with \eqref{eq:beforecontract}. 
  
 \bigskip
 
  \noindent  ${(b'.) \implies (a.)}$  We show that if $(a.)$ does not hold (i.e. $C^-(\infty)=0$), then $(b'.)$ does not either. Assume that $(a.)$ does not hold.  The goal is to show that for all $T>0$, for all $\eps >0$, there exists $g_{0,\eps} \in \LL^2_0$ such that 
  \begin{equation}
  \label{nonii}
 \| g_{0,\eps} \|_{\LL^2} =1, \quad \int_0^T {D}(g_\eps)(t) \, dt  < \eps,
 \end{equation}
 where $g_\eps$ is the solution to \eqref{B} with initial datum $g_{0,\eps}$. 
  
Fix $T>0$ and $\eps>0$. Since $C^-(\infty)=0$, there exists $(x_0,v_0) \in \T^d \times \R^d$, such that 
 \begin{equation}
 \label{hyp}
 \int_0^T \int_{\R^d} k(\phi_t(x_0,v_0), v') \, dv' \,dt < \eps/2.
 \end{equation}

Let $\chi$ be a smooth compact cutoff function defined from $\R^+$ to $\R$ such that $\chi \equiv 1$ on $[0,1]$ and $\chi \equiv 0$ on $[2,\infty)$ and such that $\int_{\R^+} \chi(r) r^{d-1} \, dr =0$. Consider 
$$
\tilde{g}_{0,n} = \chi(n|x-x_0|) \chi(n|v-v_0|).
$$
Then notice that there is $\alpha>0$ such that
$$
\|\tilde{g}_{0,n} \|_{\LL^2}^2 = n^{-2d} \alpha.
$$
There, in order to normalize, we take $g_{0,n} := \frac{n^{d}}{\alpha} \tilde{g}_{0,n}$. Note that by construction,
$$
\int g_{0,n} \, dv dx = \frac{n^{d}}{\alpha} \left(\int \chi(n|x-x_0|) \, dx\right) \left(\int \chi(n|v-v_0|) \, dv\right)  =0,
$$
and thus $g_{0,n} \in \LL^2_0$.

We call $g_n$ the solution to \eqref{B} with initial datum $g_{0,n}$. By construction, we observe that $g_{0,n} \rightharpoonup 0$ weakly in $\LL^2$ and we deduce that $g_n\rightharpoonup 0$ weakly in $L^2_{t,loc} \LL^2$. As in the previous proofs, by the averaging lemma of Corollary \ref{lemmoyenne}, this implies that
\begin{equation}
\label{convforte}
\int k(x,v',v) g_n(t,x,v') \, dv' \to 0, \text{  strongly in  } L^2_{t,loc} \LL^2.
\end{equation}

Now, consider $\nu_n := |g_n|^2$. By construction, we have:
\begin{equation}
\label{eq:nu0n}
\nu_n(0) \rightharpoonup \delta_{x=x_0,v=v_0},
\end{equation}
where $\delta$ denotes as usual the Dirac measure.
 As in \eqref{eq:nun}, we get the Duhamel's formula
    \begin{multline*}
            \nu_n(t,x,v)= e^{-2 \int_0^t \int_{\R^d} k(\phi_{s-t}(x,v) ,  v')\, dv' \, ds   } \nu_{0,n} (\phi_{-t} (x,v)) \\
            + \int_0^t  2 \left(\int_{\R^d} k(X_{s-t}(x,v),v' ,  \Xi_{s-t}(x,v)) g_n(s,X_{s-t}(x,v),v') \, dv'\right) g_n(s, \phi_{s-t}(x,v)) \\
            \times e^{-2 \int_s^t \int_{\R^d} k(\phi_{\tau-t}(x,v) ,  v')\, dv' \, d\tau   } \, ds.
               \end{multline*}

                Define now the weighted $L^2$ norm as follows
               $$
               \| f \|_{\Lb^2}^2: = \int_{\T^d \times \R^d} |f|^2  \varphi^2(x,v) \frac{e^{V}}{\Mc} \, dvdx .
               $$
               We have the following $\LL^2$ estimate for the Boltzmann equation~\eqref{B}.
               \begin{lem} 
               \label{lemtech}
              For any function $f_0 \in \LL^2$ with $\| f_0 \|_{\Lb^2}<+\infty$, the solution $f(t)$ to the Boltzmann equation~\eqref{B} with initial datum $f_0$ satisfies, for all $t\geq 0$,
           \begin{equation}
           \label{tech}
\| f(t) \|_{\Lb^2}^2 \leq \| f_0 \|_{\Lb^2}^2 +   (1+ \Gamma) \int_0^t \| f(s) \|_{\Lb^2}^2 \, ds,
\end{equation}
where
               $$
               \Gamma:= \sup_{x \in \T^d} \int_{\R^d \times \R^d} k^2(x,v’,v) \frac{\Mc(v’)}{\Mc(v)} \left(\frac{\varphi(x,v)}{\varphi(x,v')}-1\right)^2\, dv dv’ ,
               $$
               which is finite by {\bf A3'}.
 \end{lem}
               
               \begin{proof}[Proof of Lemma~\ref{lemtech}] The proof follows from an energy estimate for~\eqref{B}. We first multiply~\eqref{B} by $f  \, \varphi^2(x,v) \frac{e^{V}}{\Mc}$. Recalling that $\varphi$ is a function of the hamiltonian, we can integrate and argue as in Lemma~\ref{lemdissip} to treat the terms coming from the collision operator to obtain:
               \begin{align*}
               \frac{1}{2} \frac{d}{dt} \| f(t) \|_{\Lb^2}^2 + 0  &= \int_{\T^d\times \R^d} \left( \int_{ \R^d} k(x,v',v) \left(1-\frac{\varphi(x,v')}{\varphi(x,v)}\right) f(t,x,v') \, dv' \right) f  \, \varphi^2(x,v) \frac{e^{V}}{\Mc} \, dv dx \\
              & \qquad + \langle C(f\varphi), f\varphi\rangle_{\LL^2} \\
               &= \int_{\T^d\times \R^d} \left( \int_{ \R^d} k(x,v',v) \left(1-\frac{\varphi(x,v')}{\varphi(x,v)}\right) f(t,x,v') \, dv' \right) f  \, \varphi^2(x,v) \frac{e^{V}}{\Mc} \, dv dx - \frac{1}{2} D(f \varphi) \\
               &\leq \frac{1}{2} \int_{\T^d\times \R^d} \left( \int_{ \R^d} k(x,v',v)  \left(1-\frac{\varphi(x,v')}{\varphi(x,v)}\right) f(t,x,v') \, dv' \right)^2 \, \varphi^2(x,v) \frac{e^{V}}{\Mc} \, dv dx \\
               & \qquad + \frac{1}{2} \| f(t) \|_{\Lb^2}^2 .
               \end{align*}
              Above we have used that the dissipation term $D(f \varphi)$ is non-negative. Using the Cauchy-Schwarz inequality, we thus deduce
              \begin{align*}
               \frac{1}{2} \frac{d}{dt} \| f(t) \|_{\Lb^2}^2 \leq \frac{1}{2}\Gamma \| f(t) \|_{\Lb^2}^2  + \frac{1}{2} \| f(t) \|_{\Lb^2}^2,
               \end{align*}
 which yields~\eqref{tech}.         
               \end{proof}
               
 By construction of the sequence $(g_{0,n})$, it is uniformly compactly supported and we observe that there exists $C_0>0$ such that for all $n \in \N$,
 $$
 \| g_{0,n}\|_{\Lb^2} \leq C_0.
 $$
 We thus use Lemma~\ref{lemtech} to infer that there exists $C_T>0$, such that for all $n \in \N$, for all $t \in [0,T]$, 
\begin{equation}
\label{amelior}
 \| g_n(t)\|_{\Lb^2} \leq C_T.
 \end{equation}
 We now study the term
\begin{multline*}
A_n:= \int_{\T^d \times \R^d}  \int_0^t  2 \left(\int_{\R^d} k(X_{s-t}(x,v),v' ,  \Xi_{s-t}(x,v)) g_n(s,X_{s-t}(x,v),v') \, dv'\right) \\
\times e^{-2 \int_s^t \int_{\R^d} k(\phi_{\tau-t}(x,v) ,  v')\, dv' \, d\tau} g_n(s, \phi_{s-t}(x,v))
              \int_{\R^d} k(x,v, v') \, dv' \frac{e^{V}}{\Mc} \, ds \, dvdx
\end{multline*}
We notice that since $\varphi$ is a function of the hamiltonian, we have
$$
  \int_{\R^d} k(x,v, v') \, dv'  \leq \varphi(x,v)= \varphi(\phi_{s-t}(x,v)).
 $$
 Therefore, using the change of variables $\phi_{s-t}(x,v) \mapsto (x,v)$ and the Cauchy-Schwarz inequality, we obtain
 $$
 A_n \leq C \left\|\int_{\R^d} k(x,v' ,v) g_n(s,x,v') \, dv'\right\|_{L^1([0,t]; \LL^2)} \sup_{[0,t]} \|g_n \|_{\Lb^2}.
 $$
By~\eqref{convforte} and~\eqref{amelior}, we deduce that $\|A_n \|_{L^1[0,T]} \to 0$ as $n \to + \infty$.

Consequently, by~\eqref{eq:nu0n}, we get
\begin{align*}
& \int_0^T  \int_{\T^d \times \R^d}  \nu_n(t,x,v) \int_{\R^d} k(x,v, v') \, dv' \frac{e^{V(x)}}{\Mc(v)}dx \, dv \, dt\\
&=   \int_0^T \exp\left( -2 \int_0^t \int_{\R^d}  k(\phi_s(x_0,v_0),v') \, dv' \, ds\right) \int_{\R^d} k(\phi_t(x_0,v_0), v') \, dv' \, dt + \|A_n \|_{L^1[0,T]} \\
&<\eps/2,
\end{align*}
for $n$ large enough. 
The last inequality follows by definition of $(x_0,v_0)$ and by the assumption \eqref{hyp}. But by definition of $D(g_n)$, and using again \eqref{convforte},  for $n$ large enough, we thus have
$$
 \int_0^T {D}(g_n)(t) \, dt = 2 \int_0^T  \int_{\T^d \times \R^d}  \nu_n(t,x,v) \int_{\R^d} k(x,v, v') \, dv' \frac{e^{V(x)}}{\Mc(v)}dx \, dv \, dt + o_{n\to +\infty}(1)<\eps,
$$
We finally take $g_{0,\eps} := g_{0,n}$ (with $n$ large enough), which satisfies \eqref{nonii}. This concludes the proof.

\subsection{About the rigidity with respect to exponential convergence of the Maxwellian}

We prove here that (c.) implies (a.) in Theorem~\ref{thmexpo-intro}.

Assume that (c.) holds. By (1) implies (2) in Theorem~\ref{thmconvgene-intro}, $\omega$ satisfies the a.e.i.t. GCC.
Therefore, by Theorem~\ref{thmconv-general}, this means that $Pf_0$ is of the form defined in~\eqref{def-equimulti}. We use these notations again.

Recall by Lemma~\ref{lem-proj} that given an equivalence class $[\Omega_j]$ for $\sim$, denoting as usual $U_j=  \bigcup_{\Omega' \in [\Omega_j]} \Omega'$, we have for all $t\geq 0$,
 $$
 \int_{U_j} f(t) \, dv \, dx =  \int_{U_j} f_0 \, dv \,  dx
 $$
where $f(t)$ is the solution of~\eqref{B} with initial condition $f_0$.

Thus, the linear Boltzmann equation~\eqref{B} is well-posed in the space
$$\LL^2_{00}:= \left\{f \in \LL^2, \, \forall j \in J, \,  \int_{U_j} f \, dv dx =0\right\},
$$ 
and we can use Lemma~\ref{lemfondamental} for solutions in $\LL^2_{00}$, which yields that the exponential convergence property is equivalent to
\begin{enumerate}
\item[(c'.)] There exists $T>0$ and $K>0$ such that for all $f_0 \in \LL^2_{00}$, the associated solution $f$ to \eqref{B} satisfies
$$
K \int_0^T D(f(t)) \, dt \geq  \| f_0\|_{\LL^2}^2.
$$
\end{enumerate}
We can then make the same proof as $(b'.) \implies (a.)$ in Theorem~\ref{thmexpo-intro} in order to conclude that $C^-(\infty)=0$.
We keep the notations of that proof. The only thing to check is that $g_{0,n}$ defined there belongs to $\LL^2_{00}$ for $n$ large enough. Let $j \in J$ such that $(x_0,v_0) \in U_j$. Then for $n$ large enough, $\supp g_{0,n} \subset U_j$.
Thus for all $i \neq j$, we have $ \int_{U_i} g_{0,n}\, dv \,  dx =0$ and
$$
 \int_{U_j} g_{0,n} \, dv \,  dx =  \int_{\T^d \times \R^d} g_{0,n} \, dv \,  dx =0,
$$
by definition of $g_{0,n}$. Thus $g_{0,n} \in \LL^2_{00}$ for $n$ large enough.

\section{Remarks on lower bounds for convergence when $C^-(\infty)=0$}
\label{lowerbounds}

In the situation where $\omega$ satisfies a.e.i.t. GCC but $C^-(\infty)=0$, we know by Theorem~\ref{thmconv-general} that for all data in $\LL^2$ there is convergence to some $Pf_0$ (defined in \eqref{def-equimulti}). It is natural to wonder if there is a uniform decay rate for smoother data (e.g. in the domain of the generator of the semigroup). If so, then the question of the convergence rate one can obtain becomes particularly interesting.

Let us provide here some {\em a priori} results in this direction. The following is nothing but a rephrasing in a general framework of a result of Bernard and Salvarani \cite{BS2}.
Note that they consider in their work free transport ($V=0$) and velocities on the sphere $\S^{d-1}$, but one can readily check that their methods are relevant for \eqref{B}. In their computations, one should add the weight $e^{V}/\Mc$ in the integrals. 
\begin{thm}
\label{thmabstract}
Denote $\tau(x,v) := \inf\{ t \geq 0, \phi_{-t}(x,v) \in \omega)\}$. Assume that there is a function of time $\varphi(t)$ such that
$$
\Leb \{ (x,v) \in \T^d\times \R^d, \, \tau(x,v) >t\} \geqsim \varphi(t).
$$
Then, there exists a non-negative initial datum $f_0$ of $C^\infty$ class and $C>0$ such that for any $t \geq 0$, denoting by $f(t)$ the solution of the linear Boltzmann equation \eqref{B} with initial datum $f_0$,
$$
\left\|f(t) - P f_0\right\|_{\LL^2} \geq C \varphi(t),
$$ 
where $Pf_0$ is defined in \eqref{def-equimulti}.
\end{thm}

In particular, we obtain
\begin{coro}
\label{coro-lower-free}
Assume that $V=0$ and that $\overline{p_x(\omega)} \neq \T^d$, where $p_x$ denotes the projection on the space of positions. Then there exists a non-negative initial datum $f_0$ of $C^\infty$ class and $C>0$ such that for any $t \geq 0$, denoting by $f(t)$ the solution of the linear Boltzmann equation with initial datum $f_0$,
$$
\left\|f(t) - Pf_0\right\|_{\LL^2} \geq C/(1+t)^{d/2},
$$ 
where $Pf_0$ is defined in \eqref{def-equimulti}.
\end{coro}

\begin{proof}[Proof of Corollary \ref{coro-lower-free}]
Take $x_0 \in \T^d \setminus \overline{p_x(\omega)}$. Let $\delta := \text{dist}(x_0, \overline{p_x(\omega)})$. 
Consider  $U:=B(x_0,\delta/2)\times B(0,1)$; here $\tau(x,v) := \inf\{ t \geq 0, x-tv \in p_x(\omega)\}$.
Then the crucial point is the straightforward lower bound
$$
\Leb \{ (x,v) \in B(x_0,\delta/2)\times B(0,1), \, \tau(x,v) >t\} \geqsim 1/(1+t)^{d/2}
$$
and we can thus apply Theorem~\ref{thmabstract}.
\end{proof}

Combining with Bernard-Salvarani's theorem which concerns the case with trapped trajectories \cite{BS2}, that we recall below, one may deduce that the ``worst'' lower bound in the free transport case is due to trapped trajectories, and not to low velocities.

\begin{thm}[Bernard-Salvarani \cite{BS2}]
Let $k$ a collision kernel belonging to the class {\bf E3''} and $V=0$. Assume that there is $(x,v) \in \T^d \times \R^d$ such that for all $t\in \R^+$, $x+tv \notin \omega_x$. Then there exists a non-negative initial datum $f_0$ of $C^\infty$ class and $C>0$ such that for any $t \geq 0$, denoting by $f(t)$ the solution of the linear Boltzmann equation with initial datum $f_0$,
$$
\left\|f(t) - \left(\int_{\T^d \times \R^d} f_0 \, dv \,dx\right)  \Mc(v)\right\|_{\LL^2} \geq C/(1+t).
$$ 
\end{thm}
It is natural to conjecture that in this case, the bound in $1/t$ is optimal (this is supported by numerical evidence, as shown by De Vuyst and Salvarani \cite{DVS}).

\part{The case of specular reflection in bounded domains}
\label{Boundary}

In this section, $\Omega$ is a bounded 
 and piecewise $C^1$ domain of $\R^d$. 
 For all $x \in \partial \Omega$ (except for a set of zero Lebesgue measure, referred to as $\mathcal{B}$ below), we can consider the outward unit normal to $\partial\Omega$ at the point $x$, denoted by $n(x)$. In what follows, we also denote by $d\Sigma(x)$ the standard surface measure on $\pa \Omega$.
Consider the following partition of $\partial \Omega \times \R^d$:
\begin{equation*}
 \left\{
 \begin{aligned}
&\mathcal{B} =  \left\{(x,v) \in \partial \Omega \times \R^d, \, n(x) \text{  is not well defined}  \right\}, \\
&\Sigma_- = \left\{(x,v) \in \partial \Omega \times \R^d, \, v\cdot n(x) <0 \right\},  \\
& \Sigma_+= \left\{(x,v) \in \partial \Omega \times \R^d, \, v\cdot n(x) >0 \right\}, \\
&\Sigma_0 = \left\{(x,v) \in \partial \Omega \times \R^d, \, v\cdot n(x) =0 \right\}.
\end{aligned}
\right.
\end{equation*}

There are several relevant boundary conditions that can be considered for kinetic equations. In this paper, we focus only on the specular boundary condition case (maybe the most natural one). For this, let us first define the symmetry with respect to the tangent hyperplane to $\partial \Omega$ as
\begin{equation}
\label{speculBC}
R_x v = v - 2 (v \cdot n(x)) n(x), \quad (x,v) \in (\partial \Omega\times \R^d) \setminus \mathcal{B} .
\end{equation}
Remark that for any point $(x,v) \in \Sigma_\pm$ the reflection $R_x$ associate the point $(x,R_x v) \in \Sigma_\mp$, and that $R_x R_x v = v$. 
The linear Boltzmann equation with specular boundary condition then reads as follows:
\begin{equation}
\label{B-specular}
\left\{
\begin{aligned}
&\partial_t f + v \cdot \nabla_x f - \nabla_x V \cdot \nabla_v f= \int_{\R^d} \left[k(x,v' ,  v) f(v') - k(x,v ,  v') f(v)\right] \, dv', \\
&f(t,x,R_x v) = f(t,x,v), \quad  (x ,v) \in \partial \Omega \times \R^d.
\end{aligned}
\right.
\end{equation}
Note that the boundary conditions used here translate the fact that the particles reflect against the boundary of $\Omega$ according to the laws of geometric optics.

In this equation, we still assume that $V \in W^{2,\infty}(\Omega)$.

\bigskip

We now revisit the results which were obtained previously in the torus case.

\section{Characteristics, well-posedness and dissipation}
\label{prelimbound}

We start by defining the ``broken characteristics'' \Black $(\phi_t)_{t \geq 0}$ which will allow us to express the different relevant geometric control conditions. Note first that the force $-\nabla_x V$ can be extended to $\overline \Omega$ by uniform continuity.

\bigskip
\noindent $\bullet$ Let $(x,v) \in \Omega \times \R^d$. For small enough values of $t \geq0$, we can consider the characteristics  $\psi_t(x,v) := (X_t (x,v), \, \Xi_t(x,v))$, where  
\begin{equation}
  \label{hamilflow-boun}
 \left\{
 \begin{aligned}
 &\frac{dX_t(x,v)}{dt} = \Xi_t(x,v), \quad \frac{d\Xi_t(x,v)}{dt} = - \nabla_x V(X_t(x,v)), \\
  &X_{t=0}=x,  \quad \Xi_{t=0}=v.
 \end{aligned}
 \right.
 \end{equation}
 We define $\tau(x,v) := \inf \{t \geq 0, X_t(x,v) \in \partial \Omega\}$. Note then that for all $t \in (0,\tau(x,v))$, $\psi_t(x,v) \in \Omega \times \R^d$ and $\psi_{t=\tau(x,v)} (x,v) \in \Sigma_+ \cup \Sigma_0 \cup \mathcal{B} $.

\noindent $\bullet$ Let $(x,v) \in \Sigma_-$. Then the same construction of forward characteristics $\psi_t(x,v)$ can be performed for $t \geq 0$.

\noindent $\bullet$ We now define the broken characteristics  $(\phi_t)_{t \in \R^+}$ associated to the hamiltonian $H = \frac{1}{2} |v|^2 + V(x)$ and specular boundary conditions for almost every point of $\Omega \times \R^d$ as follows.

\medskip
Let $(x,v) \in  \Omega \times \R^d$. If for any $t\geq 0$, $\psi_t(x,v) \in \Omega \times \R^d$ (i.e. $\tau(x,v)=+\infty$), then we set $\phi_t(x,v)= \psi_t(x,v)$ for all $t\geq 0$. If not, for all $t\in [0,\tau(x,v)]$, set $\phi_t(x,v)= \psi_t(x,v)$. If $(x',v') := (X_{t= \tau(x,v)}(x,v),  \Xi_{t= \tau(x,v)}(x,v)) \in \mathcal{B} \cup \Sigma_0$, then stop the construction here. Otherwise, consider $(x_1,v_1) := (x',R_{x'} v')$. Note that $(x',v') \in \Sigma_+$ and thus $(x_1,v_1) \in \Sigma_-$.

If for any $t> 0$, $\psi_t(x_1,v_1) \in \Omega \times \R^d$, then set for $t> \tau(x,v)$, $\phi_t(x,v) = \psi_{t- \tau(x,v)} (x_1,v_1)$. If $\tau(x_1,v_1)<\infty$, for all $t\in (\tau(x,v),\tau(x,v)+\tau(x_1,v_1)]$, set $\phi_t(x,v)= \psi_{t-\tau(x,v)} (x_1,v_1)$. 

Then if  $(x'_1,v'_1):= (X_{t= \tau(x_1,v_1)} (x_1,v_1), \Xi_{t= \tau(x_1,v_1)}(x_1,v_1))  \in \mathcal{B} \cup \Sigma_0$,  stop the construction here. Otherwise, consider $(x_2, v_2) := (x'_1, R_{x'_1} v'_1)$, and so on. 

There are two possibilities:
\begin{itemize}
\item either on any interval of time, there is only a finite number of such intersections with the boundary, in which case the construction can be carried on by recursion,
\item or there is an interval of time in which there is an infinite number of such intersections with the boundary.
\end{itemize}

Nevertheless, as shown in \cite[Section 1.7]{Tab} (this is an application of Poincar\'e's recurrence lemma), the measure of the points of the phase space for which the second possibility occurs is equal to $0$.

We now recall that by a classical result by Bardos \cite[Proposition 2.3]{Bar}, which is basically an elegant application of Sard's lemma, the Lebesgue measure of the set 
$$
\mathcal{S} := \left\{ (x,v) \in \Omega \times \R^d, \, \exists t >0, \, \psi_t (x,v) \in \mathcal{B} \cup \Sigma_0 \right\}$$ 
is equal to zero. This shows that that ``pathological'' trajectories can actually be neglected.
Indeed, remark that $\phi_t(x,v)$ is well defined for all $t\geq 0$ for all $(x,v) \in \Omega \times \R^d$, except the set of zero measure evoked in the construction and the set
$$
\left\{ (x,v) \in \Omega \times \R^d, \, \exists t >0, \, \phi_t (x,v) \in \mathcal{B}  \cup \Sigma_0  \right\},
$$
but by the above property, this set has zero Lebesgue measure (as the countable union of sets with zero Lebesgue measure).

We define likewise the characteristics for almost every point of the phase space on negative times.

We finally recall the following very useful lemma (see \cite{Wec} and also \cite{SS}).

\begin{lem}
For all $s\in \R$, $\phi_s$ is measure preserving.
\end{lem}

\begin{rque}
Note that in the case where $\Omega$ is $C^1$, the Hamiltonian flow of $H(x, v) = \frac{|v|^2}{2} + V(x)$ with specular boundary conditions can be made continuous on the appropriate phase space. The latter can for instance be seen as the quotient $\W = \overline{\Omega}\times \R^d / \approx$,  where $(x,v ) \approx (x,R_x v) $ for $x \in \partial \Omega$.
A continuous function $f$ on $\overline{\Omega}\times \R^d$ satisfying $f (x,v) = f(x,R_x v)$ for $(x,v) \in \partial \Omega\times \R^d$ can be identified with a continuous function $f$ on $\W$, so that Equation~\eqref{B-specular} can be viewed as an equation on $\W$.
\end{rque}

We are in position to state the relevant definitions for the problem of convergence to equilibrium. We can define
\begin{itemize}
\item the set $\omega$ where collisions are effective, as in Definition \ref{def-om},
\item a.e.i.t. GCC, as in Definition \ref{defaeitgcc}, 
\item the equivalence relation $\sim$, as in Definition \ref{def-sim},
\item the Lebesgue spaces $\LL^p = \LL^p(\Omega \times \R^d)$, as in Definition~\ref{weightedLp},
\item the Unique Continuation Property, as in Definition \ref{def:UCP}.
\end{itemize}

We can introduce as in Definition \ref{definitionCinfini} the following Lebeau constant (note that we need to consider an essential infimum here, because characteristics are defined only almost everywhere).

\begin{deft}We define the Lebeau constant:
  \begin{equation}C^-_{b}(\infty) := \sup_{T\in \R^+}  \text{ess inf}_{(x,v) \in \Omega \times \R^d} \frac{1}{T} \int_0^T \left( \int_{\R^d} k(\phi_t (x,v), v')\, dv'\right)\, dt.
  \end{equation}
\end{deft}

\Black

Our next task is to study the well-posedness of \eqref{B-specular} in $\LL^2$ spaces, which are defined as in Definition~\ref{weightedLp}. 
One important feature is that with specular reflection, the dissipation identity still holds. The key point is to observe that for symmetry reasons, it is exactly the same  as in the torus case (that is without boundary). 
 
 Besides, as checked by Weckler \cite[Theorem 3 and Lemma 3.3]{Wec}, we have the following Duhamel formula.
 \begin{lem}
 \label{repformu}
 Let $\nu \in C(\overline\Omega\times \R^d)$ and $g \in L^\infty \LL^2$. Let $f_0 \in \LL^2$. The unique weak solution to the kinetic transport equation
 \begin{equation}
\label{Liouville-specular}
\left\{
\begin{aligned}
&\partial_t f + v \cdot \nabla_x f - \nabla_x V \cdot \nabla_v f= -b(x,v) f(t,x,v) + g , \\
&f(t,x,R_x v) = f(t,x,v), \quad (x,v) \in \partial \Omega \times \R^d, \\
&f(0,x,v)= f_0(x,v)
\end{aligned}
\right.
\end{equation}
 is given by 
 \begin{align*}
 f(t,x,v)=& \exp\left(- \int_0^t b(\phi_{-(t-s)}(x,v))  \, ds \right) f_0  \circ \phi_{-t}(x,v) \\
 & + \int_0^t    g(s, \phi_{s-t}(x,v)) \exp\left(- \int_s^t b(\phi_{\tau-t}(x,v)) \, d\tau   \right) \, ds.
\end{align*}
 \end{lem}
 
 In turn, this Duhamel formula allows in particular to prove well-posedness for~\eqref{B-specular}.
 
 \begin{rque} Note that well-posedness for~\eqref{B-specular} can be proved by other means (see for instance Mischler \cite{Mis}, where much weaker assumptions on the force field are considered); however, having the representation formula of Lemma~\ref{repformu} is a key point for the subsequent analysis (as seen in the torus case).
 \end{rque}
 
  \begin{prop}[Well-posedness of the linear Boltzmann equation with specular reflection]
 \label{prop:WP-specular}
Assume that $f_0 \in \LL^2$. Then there exists a unique $f\in C^0(\R ;\LL^2)$ solution of the initial boundary value problem~\eqref{B-specular} satisfying $f|_{t = 0} =f_0$, and we have
\begin{equation}
\text{ for all } t \geq 0, \quad  \frac{d}{dt} \| f(t)\|_{\LL^2}^2 = - D(f(t)),
\end{equation}
 where $D(f)$ is defined as follows:
  \begin{equation}
 \label{defD-boun}
 D(f) =  \frac{1}{2} \int_{\Omega} e^{V} \int_{\R^d} \int_{\R^d} \left( \frac{k(x,v' ,  v)}{\Mc(v)} + \frac{k(x,v ,  v')}{\Mc(v')} \right) \Mc(v) \Mc(v') \left(\frac{f(v)}{\Mc(v)}- \frac{f(v')}{\Mc(v')}\right)^2 \, dv' \, dv \, dx.
 \end{equation}
 If moreover $f_0 \geq 0$ a.e., then for all $t \in \R$ we have $f(t, \cdot,\cdot)\geq 0$ a.e. (Maximum principle).

 \end{prop}
 We shall not dwell on the proof of Proposition \ref{prop:WP-specular}, as it follows the lines of that of Proposition \ref{prop:WP}. 
More generally, all results of Section~\ref{preliminaries} are still relevant.

\section{Convergence to equilibrium}
\label{convboun}
As in the torus case, the following holds:

\begin{thm}[Convergence to equilibrium]
\label{thm-convboun}

The following statements are equivalent.
\begin{enumerate}[(i.)]
\item  The set $\omega$ satisfies the Unique Continuation Property.

\item  
The set $\omega$ satisfies the a.e.i.t. GCC and there exists one and only one equivalence class for the equivalence relation $\sim$.

\item  For all $f_0 \in \LL^2$, denote by $f(t)$ the unique solution to \eqref{B-specular} with initial datum $f_0$. We have
\begin{equation}
 \label{convergeto0-boundary}
 \left\|f(t)-\left(\int_{\Omega \times \R^d} f_0 \, dv \,dx\right)\frac{e^{-V}}{\int_{\Omega} e^{-V} \, dx} \Mc(v)\right\|_{\LL^2} \to_{t \to +\infty} 0,
 \end{equation}

\end{enumerate}
\end{thm}

\begin{proof}[Proof of Theorem~\ref{thm-convboun}]
One can check that the methods given in Section \ref{subsectiondecroissance} for proving Theorem \ref{thmconv-intro} are still relevant. We only highlight here the main differences.

\bigskip

\noindent $(i.) \implies (iii.)$ One can check that the beginning of the proof given for Theorem \ref{thmconv-intro} also applies here, \emph{mutatis mutandis}.

The only problem comes from the fact that at some point, we want to use an averaging lemma to obtain compactness of some averages in $v$ of the solutions, see \eqref{termecompacite}. The issue is that we a priori need ``global'' compactness, that is to say on the whole open set $\Omega$; unfortunately, classical averaging lemmas (see Corollary \ref{thmmoyenne-domain}) only provide ``local'' compactness, i.e. only for restrictions of the average on compact subsets of $\Omega$.

The key point is that the proof of $(i.) \implies (iii.)$ in Theorem \ref{thmconv-intro} is performed with a sequence of solutions enjoying a uniform $\LL^\infty$ bound, which means that it also enjoys uniform equi-intregrability (in the sense of Definition \ref{thmmoyenne-domain}). Therefore, we can apply the ``improved'' averaging lemma of Corollary \ref{thmmoyenne-domain2} which does yield compactness in the term \eqref{termecompacite}. This argument allows us to complete the proof of $(i.) \implies (iii.)$ in Theorem~\ref{thm-convboun}.

\bigskip

\noindent $(iii.) \implies (ii.)$ Once again, the beginning of the previous proof is still relevant.

Suppose that $\sim$ has at least two equivalence classes. 
The property~\eqref{eq-propconn} of Lemma~\ref{leminvarflow} is not true in this context, but can replaced by the following one. For any equivalence class $[\Omega_0]$ (for $\sim$), we have, by definition of $\sim$,
$$
\text{for all } t \geq 0 , \quad 
\phi_{-t}\left(\bigcup_{\Omega' \in [\Omega_0]} \Omega' \right) = \bigcup_{\Omega' \in [\Omega_0]} \Omega' \quad \text{almost everywhere},
$$
and we can argue as before to conclude.

\bigskip

\noindent $(ii.) \implies (i.)$ The previous proof is still relevant, \emph{mutatis mutandis}.

\end{proof}

Similarly, we also have

 \begin{thm}
\label{thmconv-general-bord}
The following statements are equivalent.
\begin{enumerate}[(i.)]

\item The set $\omega$ satisfies the a.e.i.t. GCC.

\item  For all $f_0 \in \LL^2$, denote by $f(t)$ the unique solution to \eqref{B} with initial datum $f_0$. We have
\begin{equation}
 \label{convergeto0-general-bord}
 \left\|f(t)-Pf_0 \right\|_{\LL^2} \to_{t \to +\infty} 0,
 \end{equation}
where
\begin{equation}
\label{def-equimulti-bord}
P f_0 (x,v) = \sum_{j\in J}  \frac{1}{\| \mathds{1}_{U_j} e^{-V} \Mc \|_{\LL^2}} \left( \int_{U_j} f_0 \, dv dx\right)f_j.
\end{equation}
with
$([\Omega_j])_{j \in J}$ the equivalence classes of the equivalence relation $\sim$,
\begin{equation*}
U_j = \bigcup_{\Omega'  \in [\Omega_j]} \Omega', \qquad f_j := \frac{\mathds{1}_{U_j} e^{-V} \Mc }{\|\mathds{1}_{U_j} e^{-V} \Mc \|_{\LL^2} }.
\end{equation*}

 \end{enumerate}
 \end{thm}

Note that Theorem~\ref{thm-convboun} has an interesting consequence as soon as we know that the geodesic flow enjoys ergodic properties, and thus in particular a.e.i.t. GCC is satisfied by any $\omega$ of the form $\omega_x \times \R^d$, where $\omega_x$ is a non-empty subset of $\Omega$.

\begin{coro}
 \label{coroconvergeto0-boundaryre}
Assume that $V=0$ and and $\omega=\omega_x \times \R^d$, where $\omega_x $ is a non-empty subset of $\Omega$. 
 Consider  
$$
S \Omega = \left\{ (x,v) \in \Omega \times \R^d, \, \frac12 |v|^2 = 1 \right\},
$$
and assume that the dynamics $(\phi_t)_{t \geq 0}$ defined on $S\Omega$ is ergodic.

Then for all $f_0 \in \LL^2$, denoting by $f(t)$ the unique solution to \eqref{B-specular} with initial datum $f_0$, we have
\begin{equation}
 \label{convergeto0-boundaryre}
 \left\|f(t)-\left(\int_{\Omega \times \R^d} f_0 \, dv \,dx\right)\frac{1}{|\Omega|} \Mc(v)\right\|_{\LL^2} \to_{t \to +\infty} 0,
 \end{equation}
\end{coro}
Examples of domains $\Omega$ where this corollary applies are the Bunimovich stadium, the Sinai billiard or the interior of a cardioid. Note that the homogeneity of the flow implies that ergodicity on the unit sphere is equivalent to ergodicity on any sphere bundle over $\Omega$.
Note that the conclusion of Corollary~\ref{coroconvergeto0-boundaryre} remains valid in the more general case where  $\omega  = \omega_x \times \omega_v$ where $\omega_v$ is an open set in $\R^d$ such that $\omega_v \cap S(0,R) \neq \emptyset$ for all $R>0$ (where $S(0,R)\subset \R^d$ is sphere centered in $0$ of radius $R$).

\section{Exponential convergence to equilibrium}
\label{expoboun}

In this paragraph (and here only), we consider a more restricted class of collision kernels. Precisely, we assume that the collision kernel $k$ belongs to the subclass {\bf E2*} of {\bf E2}, which we define below (note that we have slightly changed the original notations of {\bf E2}).

\bigskip

\noindent {\bf E2*. ``Factorized'' collision kernels} Let $k$ be a collision kernel verifying {\bf A1}--{\bf A3}. We suppose that there exist 
$k^* \in C^0(\overline{\Omega} \times \R^d \times \R^d)$, $\sigma \in C^0(\overline{\Omega})$,  and $\lambda_0>0$ such that 
\begin{itemize}
\item for all $(x, v,v') \in \Omega \times \R^d \times \R^d$,
\begin{equation}
\label{bornek-improv}
k(x,v,v')= \sigma(x) {k}^*(x,v,v')\Mc(v'),
\quad \text{with} \quad 
{{k}^*(x,v,  v') } + {{k}^*(x,v',  v) } \geq \lambda_0.
\end{equation}
\item we have
\begin{equation*}
(x,v) \mapsto  \int k^*(x,v,v')\Mc(v') \,dv'  \in L^\infty(\Omega \times \R^d) 
\end{equation*}
\end{itemize}

In this situation, the set $\omega$ (where the collisions are effective) is of the form $\omega_x \times \R^d$, where 
$$\omega_x= \{x \in \Omega, \, \sigma(x) >0\}.$$ 
For the sake of readability, in this section, we shall write in the following $\om$ instead of $\om_x$. To (slightly) simplify the statements, we shall assume here that $\omega$ is connected. 

\Black

For collision kernels in {\bf E2*}, we have the improvement of Lemma~\ref{cerci1}:

 \begin{lem}
 \label{cerci2}
For any $f \in \LL^2$, we have
  \begin{equation}
 D(f) \geq  \lambda_0 \left\| \sqrt{\sigma(x)}  ( f- \rho_f \Mc(v)) \right\|_{\LL^2}^2
 \end{equation}
   \end{lem}

 It is much stronger than what we have used in the case without boundary.

 \begin{proof}[Proof of Lemma \ref{cerci2}]
By \eqref{defD} and by symmetry in $v$ and $v'$, using~\eqref{bornek-improv}, we have:
  \begin{equation}
 D(f) \geq  {\lambda_0} \int_{\Omega} \sigma(x)  e^{V} \int_{\R^d} \int_{\R^d}    \Mc \Mc' \left(\frac{f}{\Mc}- \frac{f'}{\Mc'}\right)^2 \, dv' \, dv \, dx.
 \end{equation}
By Jensen's inequality it follows that
  \begin{align*}
 D(f) &\geq   {\lambda_0}\int_{\Omega}\sigma(x)   e^{V} \int_{\R^d} \Mc(v)\left(\int_{\R^d}  \Mc(v') \left(\frac{f(v)}{\Mc(v)}- \frac{f(v')}{\Mc(v')}\right) \, dv'\right)^2 dv \, dx \\
 &=   {\lambda_0}\int_{\Omega} \sigma(x)  e^{V} \int_{\R^d}      \frac{1}{\Mc(v)}   \left(f- \rho_f \, \Mc(v) \right)^2 \,  dv \, dx\\
 &=  {\lambda_0} \left\| \sqrt{\sigma(x)}  ( f- \rho_f \Mc(v)) \right\|_{\LL^2}^2,
 \end{align*}
which proves our  claim.

 \end{proof}

The main result of this section is the following theorem.

\begin{thm}[Exponential convergence to equilibrium]
\label{thmexpo-specular}
Let $k$ be a collision kernel in the class {\bf E2*}.
Let $\Omega$ be a bounded and piecewise $C^1$ domain. Assume that $\omega$ is connected.  Assume also that there exists  a neighborhood $\mathcal{V}$ of $\partial \omega \cap \partial \Omega$ in $\R^d$ such that $\partial \omega $ is of class $C^2$ in $\mathcal{V}$.   

Consider the following two statements:
\begin{enumerate}[(a.)]
\item $C^-_{b}(\infty) > 0$. 
\item There exists $C>0, \gamma>0$ such that for any $f_0 \in \LL^2(\Omega \times \R^d)$, the unique solution to \eqref{B-specular} with initial datum $f_0$ satisfies for all $t\geq 0$
\begin{multline}
\label{decexpo-specular} 
\left\|f(t)- \left(\int_{\Omega \times \R^d} f_0 \, dv \,dx\right) \frac{e^{-V}}{\int_{\Omega} e^{-V} \, dx} \Mc(v)\right\|_{\LL^2} \\
 \leq C e^{-\gamma t} \left\|f_0-\left(\int_{\Omega \times \R^d} f_0 \, dv \,dx\right) \frac{e^{-V}}{\int_{\Omega} e^{-V} \, dx} \Mc(v)\right\|_{\LL^2}.
\end{multline}
\end{enumerate}
Then $(a.)$ implies $(b.)$.
\end{thm}

\begin{figure}[H]
\label{domain-ghost}
  \begin{center}
    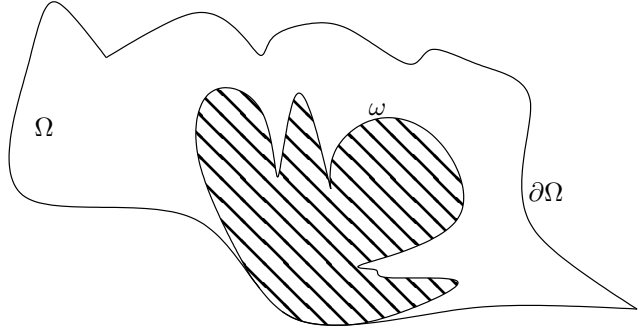 
    \caption{Illustration of the regularity assumption on $\omega$ near $\pa\omega\cap \pa\Omega$}
  \end{center}
\end{figure}

If $C^-_{b}(\infty) > 0$, we first remark that the analogue of Proposition~\ref{remUCP} holds in this setting:
 \begin{prop}
 Assume that $C^-_{b}(\infty) > 0$. Let $T>0$ such that
 $$
 \text{ess inf}_{(x,v) \in \Omega \times \R^d} \frac{1}{T} \int_0^T \left( \int_{\R^d} k(\phi_t (x,v), v')\, dv'\right)\, dt >0.
 $$
If $f \in C^0_t(\LL^2)$ is a solution to
 \begin{equation*}
 \left\{
 \begin{aligned}
&\partial_t f + v \cdot \nabla_x f - \nabla_x V \cdot \nabla_v f= 0, \\
&f(t,x,R_x v) = f(t,x,v), \quad R_x v = v - 2 (v \cdot n(x)) n(x), \quad x \in \partial \Omega, v \in \R^d, \\
&f = \rho(t,x) \Mc(v) \text{  on  }  I \times \omega,
\end{aligned}
\right.
\end{equation*}
where $I$ is an interval of length larger than $T$, then $f= \left(\int_{\Omega \times \R^d} f \, dv \,dx\right) \frac{e^{-V}}{\int_{\Omega} e^{-V} \, dx} \Mc(v)$.
 \end{prop}

As already pointed out in the proof of Theorem~\ref{thm-convboun}, the difficulty in proving Theorem~\ref{thmexpo-specular} comes  from the fact that the regularity provided by the classical averaging lemmas is a priori not valid up to the boundary.  This is the main reason why we have not been able to prove neither that $(a.)$ and $(b.)$ are equivalent, nor that the equivalence holds under the far less restrictive assumptions of Theorem~\ref{thmexpo-intro}. Nevertheless, it seems natural to conjecture that this is indeed the case. 

\medskip
In order to prove $(a.)$ implies $(b.)$, we shall overcome this difficulty by adapting some arguments due to Guo in \cite{Guo}, which forces us to make a regularity assumption on $\omega$ near $\partial \Omega$.

Note that this regularity assumption is for instance satisfied as long as $\partial \Omega$ is $C^2$ and $\sigma$ is positive on the whole boundary $\partial \Omega$.

The regularity assumption is also automatically satisfied in the case where $\overline{\omega} \subset \Omega$. As a matter of fact, if it is the case and if $\partial \Omega$ is $C^1$ then we recover the exact analogue of Theorem~\ref{thmexpo-intro}.

The paper \cite{Guo} concerns the decay of classical Boltzmann equations (that is with collisions ``everywhere''), set in bounded domains and a similar issue has to be faced at some point.
Here, we provide a slight generalization of the analysis by handling non zero potentials $V$ (in \cite{Guo}, the dynamics is dictated by free transport); furthermore, we have to modify  Guo's strategy since collisions are only effective in $\omega$ in our framework.  Loosely speaking, we will show that there can not be concentration of mass near the boundary $\pa\omega \cap \pa\Omega$. Note also that this point of the proof actually does not depend on the boundary condition chosen for the kinetic equation.

\begin{proof}[Proof of Theorem \ref{thmexpo-specular}]Assume that $C^-_{b}(\infty) > 0$. In order to show that $(b.)$ holds, the beginning of the proof is the same as that given for Theorem \ref{thmexpo-intro}, $(a.) \implies (b.)$. We keep the same notations as those of that proof. 
The only difference appears in the justification of the compactness property \eqref{eq-compact}, which we shall perform now.
To summarize, the property which remains to be shown is the following one. Given a sequence $(g_n)$ of solutions of \eqref{B-specular} in  $C^0_t([0,\infty[;\LL^2)$, satisfying
\begin{equation}
\label{assumptions-gn}
\begin{aligned}
&\sup_{t\geq 0} \| g_n(t) \|_{\LL^2}\leq 1, 
&g_n \rightharpoonup 0 \text{  weakly}-\star \text{ in  } L^\infty_t \LL^2, \\
& \int_0^{T_1} D(g_n) \, dt \to 0  ,
\end{aligned}
\end{equation}
prove that
\begin{equation}
\int k(x,v',v) g_n(t,x,v') \, dv' \to 0, \text{  strongly in  } L^2(0,T_1;\LL^2).
\end{equation}

First remark that as in the torus case, by \eqref{assumptions-gn}, Lemma~\ref{cerci2} and the averaging lemma of Corollary~\ref{thmmoyenne-domain}, for all compact sets $K \subset \Omega$, we have
$$
 \left\|\mathds{1}_{K}(x) \int k(x,v',v) g_n(t,x,v') \, dv' \right\|_{L^2(0,T_1;\LL^2)} \to 0,
$$
We shall refer to this property as \emph{interior compactness}.

\bigskip

According to the assumption on $\Omega$ and $\omega$, there exists an open subset $\tilde\omega$, of class $C^2$, included in $\omega$ such that $\pa \tilde\omega \cap \pa \Omega = \pa\omega \cap \pa\Omega$. 

We can write $\tilde\omega = \{x \in \R^d, \, \eta(x)<0\}$, where $\eta$ is a $C^2$ function such that  
\begin{equation}
\label{def-tilde-n}\tilde n(x) := \frac{\na_x \eta(x)}{|\na_x \eta(x)|}
\end{equation}
is well defined on a neighborhood of $\pa \tilde\omega = \{x \in \R^d, \, \eta(x)=0\}$.

We denote 
\begin{equation}
\tilde\omega_{\eps} =\{x \in \tilde\omega, \, \eta(x) < - \eps^4\}.
\end{equation}

\begin{figure}[H]
\label{domain-ghost-zoom}
  \begin{center}
    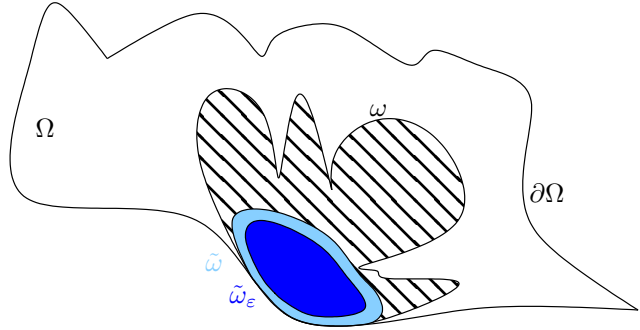 
    \caption{The open sets $\tilde\omega$ and $\tilde\omega_\eps$}
  \end{center}
\end{figure}

Then, according to the decomposition
\begin{multline*}
\left\|\int k(x,v',v) g_n(t,x,v') \, dv' \right\|_{\LL^2}^2\\
= \left\| \mathds{1}_{\tilde\omega} \int k(x,v',v) g_n(t,x,v') \, dv' \right\|_{\LL^2}^2  + \left\|\mathds{1}_{\omega \setminus \tilde\omega} \int k(x,v',v) g_n(t,x,v') \, dv' \right\|_{\LL^2}^2,
\end{multline*}
and since by interior compactness (as $\overline{\omega\setminus\tilde\omega} \subset \Omega$),
$$
 \left\|\mathds{1}_{\omega \setminus \tilde\omega} \int k(x,v',v) g_n(t,x,v') \, dv' \right\|_{L^2(0,T_1;\LL^2)} \to 0,
$$
it only remains to prove 
$$
 \left\| \mathds{1}_{\tilde\omega} \int k(x,v',v) g_n(t,x,v') \, dv' \right\|_{\LL^2}  \to 0 \text{  strongly in  } L^2(0,T_1). 
$$
We also have the decomposition
\begin{multline*}
 \left\| \mathds{1}_{\tilde\omega} \int k(x,v',v) g_n(t,x,v') \, dv' \right\|_{\LL^2}^2\\
= \left\| \mathds{1}_{\tilde\omega_\eps} \int k(x,v',v) g_n(t,x,v') \, dv' \right\|_{\LL^2}^2  + \left\|\mathds{1}_{\tilde\omega \setminus \tilde\omega_\eps} \int k(x,v',v) g_n(t,x,v') \, dv' \right\|_{\LL^2}^2  .
\end{multline*}
Again according to interior compactness  (as $\overline{\tilde\omega\setminus\tilde\omega_\eps} \subset \Omega$), we only have to prove that
\begin{equation}
\label{pregoal}
 \left\| \mathds{1}_{\tilde\omega \setminus \tilde\omega_\eps} \int k(x,v',v) g_n(t,x,v') \, dv' \right\|_{\LL^2}  \to 0 \text{  strongly in  } L^2(0,T_1). 
\end{equation}
Using~\eqref{bornek-improv} and the Cauchy-Schwarz inequality, we have

\begin{align*}
\Big\| \mathds{1}_{\tilde\omega \setminus \tilde\omega_\eps} &\int k(x,v',v) g_n(t,x,v') \, dv' \Big\|_{\LL^2} \\
& = \left\| \mathds{1}_{\tilde\omega \setminus \tilde\omega_\eps} \int \sigma(x) \Mc(v) {k}^*(x,v',v) g_n(t,x,v') \, dv' \right\|_{\LL^2}\\
& \leq \| \sqrt{\sigma}\|_\infty  \left[\int  \mathds{1}_{\tilde\omega \setminus \tilde\omega_\eps}\sigma(x) \Mc(v)\left(\int |k^*|^2(x,v',v) \Mc(v') \, dv'\right) \right.\\
& \qquad \times \left. \left(\int \frac{|g_n|^2(t,x,v')}{\Mc(v')} \, dv'\right) {e^{V}} \, \, dv  dx \right]^{1/2} \\
& \leq \| \sqrt{\sigma}\|_\infty  \left\| \int |k^*|^2(x,v',v) \Mc(v) \Mc(v') \, dv' \, dv \right\|_{L^\infty(\Omega)}^{1/2} \\
& \qquad \times \left[\int  \mathds{1}_{\tilde\omega \setminus \tilde\omega_\eps} \sigma(x)  {|g_n|^2(t,x,v')} \frac{e^{V}}{\Mc(v')}  \,dv'\, dx \right]^{1/2}
\end{align*}

Therefore, in order to prove~\eqref{pregoal}, this is sufficient to prove that
\begin{equation}
\label{goal}
\| \mathds{1}_{\tilde\omega\setminus \tilde\omega_\eps} \sqrt{\sigma(x)}  g_n \|_{L^2(0,T_1; \LL^2)} \to 0.
\end{equation}
The following is dedicated to the proof of this convergence.

\bigskip

Let $m>0$ satisfying $2m<1$. We have the following lemma, adapted from Guo \cite[Lemma 9]{Guo}, which shows that the contribution of high velocities and ``grazing'' trajectories is negligible.

\begin{lem}
\label{lem:grandesvitessesourasantes}
There exist $\eps_0>0$, $C>0$ and a nonnegative function $\varphi(n)$ going to $0$ as $n$ goes to $+\infty$, such that for any $n \in \N$ and any $\eps\in (0,\eps_0)$,
\begin{equation}
\int_0^{T_1} \int_{\tilde\omega \setminus \tilde\omega_\eps  \times \R^d } \mathds{1}_{\{ |v| > \e^{-m} \text{ or } |\tilde{n}(x)\cdot v| \leq\eps\}}(x,v) \sigma(x) |g_n (s,x,v) |^2 \frac{e^V}{\Mc(v)}\, dv \, dx \,ds \leq C \eps + \varphi(n).
\end{equation}
\end{lem}

We define now the cut-off functions
\begin{align*}
\chi_+(x,v) \: &= \: \mathds{1}_{\tilde\omega\setminus \tilde\omega_\eps}(x) \, \mathds{1}_{\{|v| \leq \e^{-m}, \, \tilde{n}(x)\cdot v >\eps\}} (x,v) \, \sqrt{\sigma(x)}, \\
\chi_-(x,v) \: &= \: \mathds{1}_{\tilde\omega\setminus \tilde\omega_\eps}(x) \,  \mathds{1}_{\{|v| \leq \e^{-m}, \, \tilde{n}(x)\cdot v <-\eps\}} (x,v)  \, \sqrt{\sigma(x)}.
\end{align*}
Let $u \in [\eps,T_1-\eps]$. Let $X(t,u,x,v)$ and $\Xi(t,u,x,v)$ be the solution to
\begin{equation}
\label{caraccarac}
\left\{
\begin{aligned}
&\frac{dX}{dt}(t,u,x,v) = \Xi(t,u,x,v), \\
&\frac{d\Xi}{dt}(t,u,x,v) = -\nabla_x V(X(t,u,x,v)),
\end{aligned}
\right.
\end{equation}
with $X(u,u,x,v)= x, \, \Xi(u,u,x,v)=v$.
We define for $t \in [\eps,s]$,
\begin{equation}
\label{chiplus}
\tilde{\chi}_+(t,x,v) =  \chi_+\Big(X(s,t,x,v), \,  \Xi(s,t,x,v)\Big),
\end{equation}
and for $t\in [s,1-\eps]$,
\begin{equation}
\label{chimoins}
\tilde{\chi}_- (t,x,v) = \chi_-\Big(X(s,t,x,v), \, \Xi(s,t,x,v)\Big), 
\end{equation}
which are built in order to satisfy the transport equation
\begin{equation*}
\pa_t \tilde\chi_\pm + v \cdot \na_x \tilde\chi_\pm - \na_x V \cdot \na_v \tilde\chi_\pm = 0, \quad \tilde\chi_\pm(s,x,v) = \chi_\pm(x,v).
\end{equation*}

The following lemma is an adaptation of \cite[Lemma 10]{Guo}, with an additional potential $V$.
\begin{lem}
\label{lem:carac}
There exists $\eps_0>0$ such that if $\eps \in (0,\eps_0)$, the following statements hold. 
\begin{enumerate}
\item For $t \in [s-\eps^2,s]$, if $\tilde\chi_+(t,x,v) \neq 0$, then $\tilde{n}(x)\cdot v>\eps/2$ and $|v|<2 \eps^{-m}$. 

Moreover, $\tilde{\chi}_+(s-\eps^2,x,v) =0$ for $x \in \tilde\omega\setminus \tilde\omega_\eps$.
\item For $t \in [s,s+ \eps^2]$, if $\tilde\chi_-(t,x,v) \neq 0$, then $\tilde{n}(x)\cdot v<-\eps/2$ and $|v|<2 \eps^{-m}$. 

Moreover, $\tilde{\chi}_-(s+\eps^2,x,v) =0$ for $x \in \tilde\omega\setminus \tilde\omega_\eps$.
\end{enumerate}
\end{lem}

For the sake of readability, we postpone the proofs of Lemmas~\ref{lem:grandesvitessesourasantes} and~\ref{lem:carac} to the end of the section.

\bigskip

We finally have the last adaptation from \cite[Lemma 11]{Guo}.

\begin{lem}
\label{lem:reste}
There exists $\eps_1>0$, $C>0$ such that for all $\eps \in (0,\eps_1)$, there exists a function $\Psi_\eps : \N \to \R^+$ satisfying $\lim_{n \to +\infty} \Psi_\eps(n) = 0$ such that
\begin{equation}
\int_0^{T_1} [\| \chi_+ g_n(s) \|_{\LL^2} ^2 + \| \chi_- g_n(s) \|_{\LL^2} ^2] \, ds \leq C \eps + \Psi_\eps(n).
\end{equation}
\end{lem}

\begin{proof}[Proof of Lemma \ref{lem:reste}]Let $s \in [\eps,T_1-\eps]$. By construction of $\tilde\chi_\pm(t,x,v)$ (see \eqref{chiplus} and \eqref{chimoins}), we have
\begin{equation}
\label{eqchign}
\pa_t (\tilde\chi_\pm g_n)+ v \cdot \na_x (\tilde \chi_\pm g_n) - \na_x V \cdot \na_v (\tilde\chi_\pm g_n) = \tilde\chi_\pm C(g_n)
\end{equation}
We deal with $\tilde\chi_+$. We multiply \eqref{eqchign} by $\tilde\chi_+ g_n \frac{e^V}{\Mc(v)}$ and integrate on $[s-\eps^2, s] \times \tilde\omega \setminus \tilde\omega_\eps \times \R^d$ to get
\begin{align*}
\| \chi_+ g_n(s) \mathds{1}_{\tilde\omega \setminus \tilde\omega_\eps \times \R^d} \|_{\LL^2} ^2 -
\| \tilde{\chi}_+ g_n (s-\eps^2)\mathds{1}_{\tilde\omega \setminus \tilde\omega_\eps \times \R^d} \|_{\LL^2}^2,\\
+ A \:
=  \: 2 \int_{s-\eps^2}^s \langle  \mathds{1}_{\tilde\omega \setminus \tilde\omega_\eps \times \R^d} \tilde\chi_+ C(g_n), g_n
\rangle_{\LL^2} \, du
\end{align*}
where $A$ denotes the contributions from the boundary of $\tilde\omega \setminus \tilde\omega_\eps$,
\begin{equation}
A:=\int_{s-\eps^2}^s \int_{\pa(\tilde\omega \setminus \tilde\omega_\eps)\times \R^d}  v\cdot N(x) (\tilde\chi_+(t,x,v) g_n(t,x,v))^2 \frac{e^V}{\Mc(v)} \, dv d\Sigma_\eps(x) dt.
\end{equation}
Here $N(x)$ is the outer normal on $\pa(\tilde\omega \setminus \tilde\omega_\eps)$ and $d\Sigma_\eps(x)$ is the surface measure on $\pa(\tilde\omega \setminus \tilde\omega_\eps)$. 

By Lemma \ref{lem:carac}, if $x \in \tilde\omega \setminus \tilde\omega_\eps$, then $\tilde\chi_+(s-\eps^2, x,v) = 0$. We deduce that 
$$
\| \tilde{\chi}_+ g_n (s-\eps^2)\mathds{1}_{\tilde\omega \setminus \tilde\omega_\eps \times \R^d} \|_{\LL^2}^2=0.
$$

Now, we partition $\pa(\tilde\omega \setminus \tilde\omega_\eps) \times \R^d$ as follows
$$
\pa(\tilde\omega \setminus \tilde\omega_\eps) \times \R^d= \gamma^- \cup \gamma^+ \cup \gamma^-_\e \cup \gamma^+_\e,
$$
where 
\begin{equation*}
\left\{
\begin{aligned}
&\gamma = \pa \tilde\omega, \\
&\gamma^-= \{ (x,v) \in \gamma\times \R^d, \, \eta(x) = 0, \,N(x)\cdot v < 0\},\\
&\gamma^+= \{ (x,v) \in \gamma\times \R^d, \, \eta(x) = 0, \,N(x)\cdot v\geq 0\},\\
&\gamma_\eps = \pa \tilde\omega_\eps, \\
&\gamma_\eps^-= \{ (x,v) \in \gamma_\eps\times \R^d, \, \eta(x) = -\eps^4, \,N(x)\cdot v <0\},\\
&\gamma_\eps^+= \{ (x,v) \in \gamma_\eps\times \R^d, \, \eta(x) = -\eps^4,\, N(x)\cdot v \geq 0\}.
\end{aligned}
\right.
\end{equation*}
Therefore we obtain
\begin{align*}
A \: = \int_{s-\eps^2}^s \int_{\gamma^+}  |v\cdot N(x)| (\tilde\chi_+(t,x,v) g_n(t,x,v))^2 \frac{e^V}{\Mc(v)} \, dv d\Sigma_\eps(x)
 \\
+ \int_{s-\eps^2}^s \int_{\gamma^+_\eps}  |v\cdot N(x)|  (\tilde\chi_+(t,x,v) g_n(t,x,v))^2 \frac{e^V}{\Mc(v)} \, dv d\Sigma_\eps(x)
 \\
 -\int_{s-\eps^2}^s \int_{ \gamma^-}  |v\cdot N(x)|  (\tilde\chi_+(t,x,v) g_n(t,x,v))^2 \frac{e^V}{\Mc(v)} \, dv d\Sigma_\eps(x)
\\
- \int_{s-\eps^2}^s \int_{ \gamma^-_\eps}  |v\cdot N(x)|  (\tilde\chi_+(t,x,v) g_n(t,x,v))^2 \frac{e^V}{\Mc(v)} \, dv d\Sigma_\eps(x).
\end{align*}
Recall the definition of $\tilde{n}$ in \eqref{def-tilde-n}. Observe that on $\gamma$, we have $N(x)=\tilde{n}(x)$, while on $\gamma_\eps$, we have $N(x)=-\tilde{n}(x)$.

Using Lemma \ref{lem:carac}, for all $t \in [s-\eps^2,s]$, if $\tilde{\chi}(t,x,v) \neq 0$, then $\tilde{n}(x)\cdot v >0$. Therefore,
$$
\int_{s-\eps^2}^s \int_{ \gamma^-}  |v\cdot N(x)|  (\tilde\chi_+(t,x,v) g_n(t,x,v))^2 \frac{e^V}{\Mc(v)} \, dv d\Sigma_\eps(x)=0.
$$
We thus obtain the bound
\begin{align}
\| \chi_+ g_n(s) \|_{\LL^2} ^2  &\leq \int_{s-\eps^2}^s \int_{ \gamma^-_\eps}  |v\cdot N(x)|  (\tilde\chi_+(t,x,v) g_n(t,x,v))^2 \frac{e^V}{\Mc(v)} \, dv d\Sigma_\eps(x) dt  \nonumber\\
&+ 2 \int_{s-\eps^2}^s \langle  \mathds{1}_{\tilde\omega \setminus \tilde\omega_\eps \times \R^d} \tilde\chi_+ C(g_n), g_n
\rangle_{\LL^2} \, du.
\end{align}
Similarly, for $\tilde\chi_-$,  we multiply \eqref{eqchign} by $\tilde\chi_- g_n \frac{e^V}{\Mc(v)}$ and integrate on $[s,s+\eps^2] \times \tilde\omega \setminus \tilde\omega_\eps \times \R^d$ to get
\begin{align}
\| \chi_- g_n(s) \|_{\LL^2} ^2  &\leq \int_{s}^{s+\eps^2}\int_{ \gamma^+_\eps}  |v\cdot N(x)|  (\tilde\chi_-(t,x,v) g_n(t,x,v))^2 \frac{e^V}{\Mc(v)} \, dv d\Sigma_\eps(x) dt \nonumber \\&- 2 \int_s^{s+\eps^2} \langle  \mathds{1}_{\tilde\omega \setminus \tilde\omega_\eps \times \R^d} \tilde\chi_- C(g_n), g_n
\rangle_{\LL^2} \, du.
\end{align}
According to Lemma \ref{lem:carac} and the definitions of $\chi_\pm$ and $\tilde\chi_\pm$, we have
\begin{align*}
\int_{s-\eps^2}^s \int_{ \gamma^-_\eps}  &|v\cdot N(x)|  (\tilde\chi_+(t,x,v) g_n(t,x,v))^2 \frac{e^V}{\Mc(v)} \, dv d\Sigma_\eps(x) dt \\
&\leq \|\sigma\|_\infty \int_{s-\eps^2}^s \int_{ \gamma_\eps}  |v\cdot N(x)|  (\mathds{1}_{\{|v| \leq 2 \eps^{-m}, \, \tilde{n}(x)\cdot v >\eps/2\}} (x,v)  g_n(t,x,v))^2 \frac{e^V}{\Mc(v)} \, dv d\Sigma_\eps(x) dt \\
&\leq \|\sigma\|_\infty \int_{s-\eps^2}^{s+\eps^2} \int_{ \gamma_\eps}  |v\cdot \tilde{n}(x)|  (\mathds{1}_{\{|v| \leq 2 \eps^{-m}, \, |\tilde{n}(x)\cdot v |>\eps/2\}} (x,v)  g_n(t,x,v))^2 \frac{e^V}{\Mc(v)} \, dv d\Sigma_\eps(x) dt,
\end{align*}
and likewise
\begin{align*}
\int_s^{s+\eps^2} \int_{ \gamma^+_\eps}  &|v\cdot N(x)|  (\tilde\chi_-(t,x,v) g_n(t,x,v))^2 \frac{e^V}{\Mc(v)} \, dv d\Sigma_\eps(x) dt \\
&\leq \|\sigma\|_\infty \int_{s-\eps^2}^{s+\eps^2} \int_{ \gamma_\eps}  |v\cdot \tilde{n}(x)|  (\mathds{1}_{\{|v| \leq 2 \eps^{-m}, \, |\tilde{n}(x)\cdot v| >\eps/2\}} (x,v)  g_n(t,x,v))^2 \frac{e^V}{\Mc(v)} \, dv d\Sigma_\eps(x) dt.
\end{align*}
We thus have to study
\begin{align*}
&\int_{s-\eps^2}^{s+\eps^2} \int_{ \gamma_\eps}  |v\cdot \tilde{n}(x)|  (\mathds{1}_{\{|v| \leq 2 \eps^{-m}, \, |\tilde{n}(x)\cdot v| >\eps/2\}} (x,v)  g_n(t,x,v))^2 \frac{e^V}{\Mc(v)} \, dv d\Sigma_\eps(x) dt \\
&\leq \int_{s-\eps^2}^{s+\eps^2}  \int_{ \gamma_\eps} (\mathds{1}_{\{|v| \leq 2 \eps^{-m}, \, |\tilde{n}(x)\cdot v| >\eps/2\}} (x,v)  g_n(t,x,v))^2 \frac{(\tilde{n}(x)\cdot v)^2}{1+|v|} \frac{4(1+2\eps^{-m})}{\eps^2}  \frac{e^V}{\Mc(v)}  \, dv d\Sigma_\eps(x) du\\
&\leq  C_\eps \int_{s-\eps^2}^{s+\eps^2}  \int_{ \gamma_\eps} g_n^2(u)  \frac{e^V}{\Mc(v)} \frac{(N(x)\cdot v)^2}{1+|v|} \, dv d\Sigma_\eps(x) du,
\end{align*}
with $C_\eps =  \frac{4(1+2\eps^{-m})}{\eps^2}$.

We recall below Cessenat's trace theorem (see \cite{Ces} or \cite[Proposition B.2]{SR}):

\begin{lem}
Let $U$ be a $C^1$ domain. Let $T_1, T_2>0$. There exists a constant $C>0$ such that, for any function $f \in L^2(T_1,T_2;L^2(U \times \R^d))$ satisfying $(\pa_t + v \cdot \na_x - \na_x V \cdot \na_v) f \in L^2(T_1,T_2;L^2(U \times \R^d))$, we have
\begin{multline*}
\| f\vert_{\pa U} \|_{L^2(T_1,T_2; L^2( |v \cdot n(x)|^2 (1+ |v|)^{-1} d\Sigma(x) dv))} \\
\leq C \left( \|f \|_{L^2(T_1,T_2;L^2(U \times \R^d))} + 
 \|(\pa_t + v \cdot \na_x - \na_x V \cdot \na_v) f \|_{L^2(T_1,T_2;L^2(U \times \R^d))} \right),
\end{multline*}
where $d\Sigma$ denotes the surface measure and $n$ is the outward unit normal on $\pa U$.
\end{lem}

This allows us to obtain the control:
\begin{align*}
\int_{s-\eps^2}^{s+\eps^2}  \int_{ \gamma_\eps} g_n^2(u)  \frac{e^V}{\Mc(v)} \frac{(N(x)\cdot v)^2}{1+|v|}  \, dv d\Sigma_\eps(x) du \leq K_\eps\int_{s-\eps^2}^{s+\eps^2} [\|g_n(u) \|^2_{\LL^2( \tilde\omega_\eps\times \R^d)} + \|C(g_n)(u) \|^2_{\LL^2( \tilde\omega_\eps\times \R^d)}] \, du.
\end{align*}

By interior compactness, since $\overline{\tilde\omega_\eps} \subset \Omega$, we know that 
$$ \int_{s-\eps^2}^{s+\eps^2} \|g_n(u) \|^2_{\LL^2( \tilde\omega_\eps\times \R^d)} \, du \to 0,$$ 
as $n \to + \infty$.

On the other hand, using the assumption {\bf A2} on the collision kernel, we have
\begin{align*}
 \int_{s-\eps^2}^{s+\eps^2} \|C(g_n)(u) \|^2_{\LL^2} \, du&=  \int_{s-\eps^2}^{s+\eps^2} \| C(g_n- \rho_{g_n} \Mc(v))(u) \|^2_{\LL^2} \, du \nonumber\\
& \leq  E_1 +E_2, 
\end{align*}
where
\begin{align*}
 E_1 &:= \int_{s-\eps^2}^{s+\eps^2} \left\| \left( \int k(x,v,v') \, dv'\right) [g_n - \rho_{g_n} \Mc(v)] \right\|_{\LL^2}^2 \, du ,\\
 E_2 &:= \int_{s-\eps^2}^{s+\eps^2}  \left\|   \int k(x,v',v)  [g_n(v') - \rho_{g_n} \Mc(v')]  \, dv' \right\|_{\LL^2}^2 \, du .
\end{align*}
We first study $E_1$. Using~\eqref{bornek-improv}, we have
\begin{align*}
 E_1 &= \int_{s-\eps^2}^{s+\eps^2} \left\| \left( \int \sigma (x) k^*(x,v,v') \Mc(v') \, dv'\right) [g_n - \rho_{g_n} \Mc(v)] \right\|_{\LL^2}^2 \, du ,\\
&\leq  \left(\sup_{(x,v) \in \Omega \times \R^d} \int k^*(x,v,v')\Mc(v') \,dv'\right)^2  \| \sqrt\sigma\|_\infty \int_{s-\eps^2}^{s+\eps^2} \left\| \sqrt{\sigma} [g_n - \rho_{g_n} \Mc(v)] \right\|_{\LL^2}^2 \, du.
\end{align*}
We argue likewise for $E_2$ and obtain a similar bound.

Therefore, using Lemma \ref{cerci2} and the dissipation bound in \eqref{assumptions-gn}, we deduce:
\begin{align}
 \int_{s-\eps^2}^{s+\eps^2} \|C(g_n)(u) \|^2_{\LL^2} \, du
& \leq  C \int_{s-\eps^2}^{s+\eps^2} \|\sqrt{\sigma(x)} (g_n- \rho_{g_n} \Mc(v)) \|^2_{\LL^2} \, du
\leq C \int_0^{T_1} D(g_n(t)) dt\nonumber \\
& \leq  \Psi_1(n), \label{bouboun}
\end{align}
with $\Psi_1(n)$ a function tending to $0$ as $n$ goes to $+\infty$. 

\Black

Moreover, we have the rough bound, obtained by Cauchy-Schwarz inequality:
\begin{align*}
 \int_{s-\eps^2}^s \langle  \mathds{1}_{\tilde\omega \setminus \tilde\omega_\eps \times \R^d} \tilde\chi_+ C(g_n), g_n \rangle_{\LL^2} \, du &-  \int_s^{s+\eps^2} \langle  \mathds{1}_{\tilde\omega \setminus \tilde\omega_\eps \times \R^d} \tilde\chi_- C(g_n), g_n
\rangle_{\LL^2} \, du \\
&\leq   \int_{s-\eps^2}^{s+\eps^2} \| C(g_n)\|_{\LL^2}  \| g_n\|_{\LL^2} \,du \\
&\leq  \sup_{t\geq 0} \| g_n\|_{\LL^2}  \int_{s-\eps^2}^{s+\eps^2} \| C(g_n)\|_{\LL^2}  \,du \\
& \leq \int_{s-\eps^2}^{s+\eps^2} \| C(g_n)\|_{\LL^2}  \,du,
\end{align*}
where we have used~\eqref{assumptions-gn} on the last line. We can use again \eqref{bouboun} to bound this by $\Psi_1(n)$.
\bigskip

Summarizing ($\eps$ being fixed), we have proved the existence of $\Psi_\eps(n)$ tending to $0$ as $n$ tends to infinity such that
$$
\int_\eps^{T_1-\eps} [\| \chi_- g_n(s) \|_{\LL^2} ^2 + \| \chi_- g_n(s) \|_{\LL^2} ^2] \, ds \leq  \Psi_\eps(n).
$$
We also have, by~\eqref{assumptions-gn},
\begin{align*}
\int_0^\eps [\| \chi_- g_n(s) \|_{\LL^2} ^2 + \| \chi_- g_n(s) \|_{\LL^2} ^2] \, ds &\leq 2 \|\sigma\|_{\infty} \eps, \\
\int_{T_1 - \eps}^{T_1} [\| \chi_- g_n(s) \|_{\LL^2} ^2 + \| \chi_- g_n(s) \|_{\LL^2} ^2] \, ds  &\leq  2 \|\sigma\|_{\infty} \eps.
\end{align*}
which yields the claimed result of Lemma~\ref{lem:reste}.
\end{proof}

\bigskip

\noindent {\bf End of the proof of Theorem~\ref{thmexpo-specular}.} We are now ready to prove \eqref{goal}, which is a consequence of Lemmas \ref{lem:grandesvitessesourasantes} and \ref{lem:reste}.
Indeed, we can write
\begin{align*}
&\| \mathds{1}_{\tilde\omega \setminus \tilde\omega_\eps} \sqrt{\sigma}  g_n \|_{L^2(0,T_1; \LL^2)} ^2 \\
&=\| \mathds{1}_{\{\tilde\omega \setminus \tilde\omega_\eps}\mathds{1}_{|v| > \e^{-m} \text{ or } |\tilde{n}(x)\cdot v| \leq\eps\}}  \sqrt{\sigma}  g_n \|_{L^2(0,T_1;\LL^2)} ^2 + \| \mathds{1}_{\{\tilde\omega \setminus \tilde\omega_\eps}\mathds{1}_{|v| \leq \e^{-m} \text{ and } |\tilde{n}(x)\cdot v| >\eps\}}  \sqrt{\sigma}  g_n \|_{L^2(0,T_1; \LL^2)} ^2 \\
&=\| \mathds{1}_{\tilde\omega \setminus \tilde\omega_\eps}\mathds{1}_{\{|v| > \e^{-m} \text{ or } |\tilde{n}(x)\cdot v| \leq\eps\}}  \sqrt{\sigma}  g_n \|_{L^2(0,T_1;\LL^2)} ^2 + \|\chi_+   g_n \|_{L^2(0,T_1; \LL^2)} ^2 + \|\chi_-   g_n \|_{L^2(0,T_1; \LL^2)} ^2 \\ 
&\leq   C \eps + \varphi(n) + \Psi_\eps(n).
\end{align*}
Let $\delta>0$. Fix $\eps>0$ small enough such that $\eps<\eps_0$ and $C\eps < \delta/2$. Once this parameter $\eps$ is fixed, choose $n_0$ large enough such that for all $n \geq n_0$, $ \varphi(n) +  \Psi_\eps(n) \leq \delta/2$. Then for all $n \geq n_0$, we have $\| \mathds{1}_{\tilde\omega \setminus \tilde\omega_\eps} \sqrt{\sigma} g_n \|_{L^2(0,T_1; \LL^2)} ^2 \leq \delta$, which proves the convergence to $0$.

Therefore, this concludes the proof of Theorem~\ref{thmexpo-specular}.
\end{proof}

Now, it only remains to prove Lemmas~\ref{lem:grandesvitessesourasantes} and~\ref{lem:carac} for the proof of Theorem~\ref{thmexpo-specular} to be complete. This is the aim of the next two sections.

\subsection{Proof of Lemma \ref{lem:grandesvitessesourasantes}}

First take $\eps_0$ small enough so that $\tilde{n}(x)$ is well defined for $x \in \tilde\omega \setminus \tilde\omega_{\eps_0}$.
We write the decomposition
\begin{align*}
\int_0^{T_1} \int_{\tilde\omega \setminus \tilde\omega_\eps \times \R^d} \mathds{1}_{\{|v| > \e^{-m} \text{ or } |\tilde{n}(x)\cdot v| \leq\eps\}} \sigma(x) |g_n (s,x,v) |^2 \frac{e^V}{\Mc(v)} \, dv \, dx \,ds 
\leq A_1 + A_2,
\end{align*}
with
\begin{align*}
A_1&:=  2\|\sigma\|_{L^\infty(\Omega)} \int_0^{T_1} \int_{\tilde\omega \setminus \tilde\omega_\eps\times \R^d} \mathds{1}_{\{|v| > \e^{-m} \text{ or } |\tilde{n}(x)\cdot v| \leq\eps\}}  |\rho_{g_n} (s,x) |^2 e^V \Mc(v) \, dv \, dx \,ds,\\
A_2&:= 2 \int_0^{T_1} \int_{\tilde\omega \setminus \tilde\omega_\eps\times \R^d} \mathds{1}_{\{|v| > \e^{-m} \text{ or } |\tilde{n}(x)\cdot v| \leq\eps\}}  \sigma(x)  |g_n (s,x,v)- \rho_{g_n}(s,x) \Mc(v) |^2 \frac{e^V}{\Mc(v)} \, dv \, dx \,ds,
\end{align*}
where 
$$\rho_{g_n} := \int g_n \, dv.$$
For $A_1$, we use the uniform bound $\sup_{t \geq 0}\| g_n \|_{\LL^2} \leq 1$ and Cauchy-Schwarz inequality to obtain
\begin{align*}
A_1 &\leq  2 \|\sigma\|_{L^\infty(\Omega)} \int_0^{T_1} \left(\int_\Omega  |\rho_{g_n}|^2 e^{V} \, dx \right) dt \sup_{x \in \tilde\omega \setminus \tilde\omega_\eps } \int_{\{|v| > \e^{-m} \text{ or } |\tilde{n}(x)\cdot v| \leq\eps\}}  \Mc(v) \, dv \\
&\leq 2 \|\sigma\|_{L^\infty(\Omega)} \int_0^{T_1} \left(\int_\Omega  \left(\int \frac{|{g_n(x,v')}|^2}{\Mc(v')} \, dv' \right)\left(\int \Mc(v') \, dv' \right) e^{V(x)} \, dx \right) dt \\
& \qquad\qquad\qquad \times \sup_{x \in \tilde\omega \setminus \tilde\omega_\eps } \int_{\{|v| > \e^{-m} \text{ or } |\tilde{n}(x)\cdot v| \leq\eps\}}  \Mc(v) \, dv \\
&\leq 2T_1 \|\sigma\|_{L^\infty(\Omega)} \left(\sup_{t \geq 0} \| g_n \|_{\LL^2} \right)^2 \sup_{x \in \tilde\omega \setminus \tilde\omega_\eps } \int_{\{|v| > \e^{-m} \text{ or } |\tilde{n}(x)\cdot v| \leq\eps\}}  \Mc(v) \, dv \\
&\leq 2T_1  \|\sigma\|_{L^\infty(\Omega)} \sup_{x \in \tilde\omega \setminus \tilde\omega_\eps} \int_{\{|v| > \e^{-m} \text{ or } |\tilde{n}(x)\cdot v| \leq\eps\}}  \Mc(v) \, dv \
\end{align*}
One can check that there exists $C>0$ independent of $\eps$ such that for all $x \in \tilde\omega \setminus \tilde\omega_\eps$, we have
$$
\int_{\{|v| > \e^{-m} \text{ or } |\tilde{n}(x)\cdot v| \leq\eps \}}  \Mc(v) \, dv  \leq C \eps,
$$
so that we have
$$
A_1 \leq 2 C T_1\|\sigma\|_{L^\infty(\Omega)} \eps.
$$
On the other hand, we have the rough bound
$$
A_2 \leq 2\int_0^{T_1} \int_{\Omega} \int_{\R^d}  \sigma(x) |g_n (s,x,v)- \rho_n(s,x) \Mc(v) |^2 \frac{e^V}{\Mc(v)} \, dv \, dx \,ds ,
$$
which also goes to $0$ as $n$ goes to infinity, thanks to Lemma \ref{cerci2} and the dissipation bound in \eqref{assumptions-gn}. This concludes the proof of Lemma~\ref{lem:grandesvitessesourasantes}.

\subsection{Proof of Lemma \ref{lem:carac}}

First take $\eps_0$ small enough such that $\tilde{n}(x)$ is well defined on $\{x \in \tilde\omega, \,  -\sqrt{\eps_0}<\eta(x) <  \sqrt{\eps_0} \}$.

We only prove Item $(1)$ (the treatment of Item $(2)$ being identical). 

For the first statement of Item $(1)$, assume that $s,t,x,v$ are such that $\tilde\chi_+(t,x,v) \neq 0$ and $t \in [s-\eps^2 ,s]$.
We have
\begin{align}
\label{B1B2B3}
& \tilde{n}(x) \cdot v = B_1 +B_2 +B_3, \quad \text{with} \quad
 B_1 = \tilde{n} (X(s,t,x,v)) \cdot \Xi(s,t,x,v)  ,\\
 & B_2 = - \tilde{n} (X(s,t,x,v)) \cdot (\Xi(t,s,x,v)- v) ,
\quad B_3 = - (\tilde{n} (X(s,t,x,v))-\tilde{n}(x))\cdot v . \nonumber
\end{align}
Since $\tilde\chi_+(t,x,v) \neq 0$, we have $B_1 \geq \eps$.
For $B_2$ and $B_3$, using~\eqref{caraccarac}
, we obtain the estimates
\begin{align*}
|B_2 |&\leq  |\Xi(s,t,x,v)- v| \leq |t-s| \|\na_x V\|_\infty \leq C \eps^2, \\
|B_3 |&\leq |\tilde{n} (X(s,t,x,v))-\tilde{n}(x)| |v|  \\
&\leq |\na \tilde{n}(\overline{x})||X(s,t,x,v))-x| |v|,
\end{align*}
where $\overline{x}$ belongs to the segment $[x,X(s,t,x,v)]$. Since $\tilde\chi_+(t,x,v) \neq 0$, we have $|\Xi(s,t,x,v)|\leq \eps^{-m}$.
We deduce, using again~\eqref{caraccarac}, that
\begin{align}
\nonumber |X(s,t,x,v))-x|  &\leq \int_t^s |\Xi (s,u,x,v)| \, du \\
\nonumber&\leq |t-s| |\Xi(s,t,x,v)| + \frac{|t-s|^2}{2} \| \na_x V\|_\infty\\
\nonumber&\leq \eps^2 \eps^{-m} + \frac{ \| \na_x V\|_\infty}{2} \eps^4 \\
\label{utileX}&\leq 2\eps^2 \eps^{-m},
\end{align}
for $\eps< \eps_0$ with $\eps_0$ small enough, as $m>0$.

Thus, we infer that $|\overline{x}-X(s,t,x,v)| \leq |{x}-X(s,t,x,v)|\leq 2 \eps^2 \eps^{-m} < \eps$, for $\eps< \eps_0$ with $\eps_0$ small enough, as $2m<1$. Moreover, since $\tilde\chi_+(t,x,v) \neq 0$, we have $X(s,t,x,v) \in \tilde\omega\setminus \tilde\omega_\eps$ and thus $\eta(\overline{x}) \in   (-\sqrt{\eps_0}, \sqrt{\eps_0})$.
Therefore, there is a constant $C>0$ (depending only on $\eps_0$ and $\eta$) such that $|\na \tilde{n}(\overline{x})|\leq C$. 
We deduce that 
\begin{align*}
|B_3 |&\leq C \eps^2 \eps^{-2m}.
\end{align*}
Coming back to~\eqref{B1B2B3}, this yields
\begin{equation}
\tilde{n}(x) \cdot v \geq \eps - C \eps^{2(1-m)} - C \eps^2 \geq \eps/2.
\end{equation}
for $\eps< \eps_0$ with $\eps_0$ small enough, since $2m<1$.

Likewise using the decomposition
$$
v = \Xi(s,t,x,v) + [v- \Xi(s,t,x,v)],
$$
we prove that $|v|<2 \eps^{-m}$ for $\eps< \eps_0$ with $\eps_0$ small enough.

\bigskip

Now let us prove the second statement of Item $(1)$. Let $x \in \tilde\omega \setminus \tilde\omega_\eps$. We argue by contradiction. Assuming that $\tilde{\chi}_+(s-\eps^2,x,v)>0$, we have  
\begin{equation}
\label{contraeq}
\begin{array}{c}
-\eps^4 < \eta(X(s,s-\eps^2,x,v)) < 0, \quad |\Xi(s,s-\eps^2,x,v)|<\eps^{-m}, \\ 
\text{and} \quad 
\tilde{n}(X(s,s-\eps^2,x,v)) \cdot \Xi(s,s-\eps^2,x,v) >\eps.
\end{array}
\end{equation}
As before, we deduce that $|v|<2 \eps^{-m}$.

We write the Taylor-Lagrange formula
\begin{align*}
\eta(X(s,s-\eps^2,x,v))
& = \eta(x) - \na \eta(X(s,s-\eps^2,x,v)) \cdot (x-X(s,s-\eps^2,x,v)) \\
&\quad  - (\na^2 \eta(\overline{x}) [x-X(s,s-\eps^2,x,v)],x-X(s,s-\eps^2,x,v)),
\end{align*}
where $\overline{x}$ belongs to the segment $[x,X(s,s-\eps^2,x,v)]$.  

The Taylor formula, together with~\eqref{caraccarac} yields
$$
x-X(s,s-\eps^2,x,v) = \eps^2 \Xi(s,s-\eps^2,x,v)  -\int_{s-\eps^2}^s (s-w) \na_x V(X(s,w,x,v)) \, dw.
$$
As a consequence, we have
\begin{equation*}
 -  \na \eta(X(s,s-\eps^2,x,v)) \cdot (x-X(s,s-\eps^2,x,v)) = D_1 +D_2,
\end{equation*}
where
\begin{align*}
D_1 &:= -\eps^2 \na \eta(X(s,s-\eps^2,x,v)) \cdot  \Xi(s,s-\eps^2,x,v) ,\\
D_2 &:=    \int_{s-\eps^2}^s  (s-w)  \na_x V(X(s,w,x,v))  \cdot  \na \eta(X(s,s-\eps^2,x,v)) \, dw.
\end{align*}
Recalling the definition of $\tilde{n}$ in~\eqref{def-tilde-n}, there is a constant $C>1$ (depending  only on $\eps_0$ and $\eta$) such that 
$$1/C\leq|\na \eta|(X(s,s-\eps^2,x,v)) \leq C, \quad |\na^2 \eta(\overline{x})|\leq C.$$
We thus have
\begin{align*}
D_1 &= -\eps^2 |\na \eta(X(s,s-\eps^2,x,v))| \tilde{n}(X(s,s-\eps^2,x,v)) \cdot \Xi(s,s-\eps^2,x,v) \\
&\leq -{\eps^3} |\na \eta(X(s,s-\eps^2,x,v))| \\ 
&\leq -1/C \eps^3,
\end{align*}
together with
$$
|D_2| \leq C \eps^4.
$$
Furthermore,  using the second equation of~\eqref{contraeq} and arguing as for~\eqref{utileX}, we obtain the bound
$$
(\na^2 \eta(\overline{x}) [X(s,s-\eps^2,x,v)-x],X(s,s-\eps^2,x,v)-x) \leq C [ \eps^2 \eps^{-m}  + \eps^4]^2 \leq C \eps^{2(2-m)}.
$$
We deduce that
$$
\eta(X(s,s-\eps^2,x,v))  < 0 -1/C  \eps^3 + {C \eps^4 + \eps^{2(2-m)}} < - \eps^4,
$$
using again $2m<1$ and taking $\eps$ small enough. This is a contradiction with the first equation of~\eqref{contraeq}. This concludes the proof of Lemma~\ref{lem:carac}.

\section{About other boundary conditions}
\label{otherboun}
\noindent {\bf 1.}
We could have considered slightly more general reflection laws of the form
\begin{equation}
\label{eq-refl}
f(t,x,T_x v) = f(t,x,v), \quad \text{for  } (x,v) \in \Sigma_+,
\end{equation}
for any family of transformations $(T_x)_{x \in \partial\Omega}$ such that $T_x: p_v(\Sigma_+) \to p_v(\Sigma_-)$ (here $p_v$ denotes the projection on the $v$ space) satisfies $\|T_x(v)\|=\|v\|$, where $\| \cdot \|$ denotes the standard euclidian norm.
The only thing to do is to modify accordingly the definition of characteristics. This includes for instance the (sometimes used) \emph{bounce-back} boundary condition:
\begin{equation}
\label{eq-bounce}
f(t,x,v) = f(t,x,-v), \quad \text{for } (x,v) \in \partial \Omega \times \R^d.
\end{equation}

\bigskip

\noindent {\bf 2.} Another important class of boundary conditions for kinetic equations is given by the:

\noindent $\bullet$ \emph{Diffusive boundary condition} (or Maxwellian diffusion), which reads:
\begin{equation}
f(t,x, v) = \frac{\int_{v'\cdot n(x) >0} f(t,x,v') \, v' \cdot n(x) \, dv'}{\int_{v'\cdot n(x) <0}\Mc_w(v') |v' \cdot n(x)| \, dv'} \Mc_w(v), \quad (x,v) \in \Sigma_-.
\end{equation}
where $\Mc_w(v) $ is some Maxwellian distribution characterizing the state of the wall (depending on its temperature).

\bigskip

We restrict ourselves to the case $\Mc_w(v) = \Mc(v)$ (which means that we consider that the wall has reached a global equilibrium compatible with the linear Boltzmann equation). Then, contrary to the specular reflection case, there is a non-trivial contribution of the boundary in the dissipation identity.

  \begin{lem}
 \label{lemdissip-diffusive}Let $f \in C^0_t(\LL^2)$ be a solution to \eqref{B} with diffusive boundary conditions. The following identity holds, for all $t\geq 0$:
 \begin{equation}
 \label{eqdissipation-diffusive}
 \frac{d}{dt}  \|f(t)\|_{\LL^2}^2 = - \tilde{D}(f),
 \end{equation}
 with:
 \begin{equation}
  \tilde{D}(f) = D(f) +  \int_{v\cdot n(x)>0} \left( f-  \frac{\int_{v'\cdot n(x) >0} f(t,x,v') \, v' \cdot n(x) \, dv'}{\int_{v'\cdot n(x) <0}\Mc(v') |v' \cdot n(x)| \, dv'} \Mc(v)\right)^2 v \cdot n(x) \frac{e^{V(x)}}{\Mc(v)} \, dv \, d\Sigma.
 \end{equation}
 where $D(f)$ is defined in \eqref{defD-boun} and $d\Sigma$ denotes the surface measure on $\partial \Omega$.
 
 \end{lem}

It is also possible to study 

\bigskip

\noindent $\bullet$ \emph{A combination of specular and diffusive boundary conditions}. Let $\alpha \in (0,1)$.
\begin{equation}
f(t,x, v) = \alpha \frac{\int_{v'\cdot n(x) >0} f(t,x,v') \, v' \cdot n(x) \, dv'}{\int_{v'\cdot n(x) <0}\Mc_w(v') |v' \cdot n(x)| \, dv'} \Mc_w(v) + (1- \alpha) f(t,x,R_x v) 
\quad (x,v) \in \Sigma_-.
\end{equation}

\bigskip

The last two boundary conditions are very relevant from the physical point of view.
As such, they are worth being studied. We leave this problem for future studies.

  \part{Other geometrical situations: manifolds and compact phase spaces}
  \label{part3}
  
\section{The case of a general compact Riemannian manifold}
\label{sec:manifold}

Let $(M,g)$ be a smooth compact connected $d$-dimensional Riemannian manifold (without boundary). In local coordinates, the metric $g$ is a symmetric positive definite matrix such that for all $x \in M$  and $u, w \in T_x M$, we have 
$$
(u,w)_{g(x)} = g_{i,j}(x) u^i w^j ,
$$
where the Einstein summation notations are used. This provides a canonical identification between the tangent bundle $TM$ and the cotangent bundle $T^*M$ {\em via} the following formula. For any vector $u \in T_x M$ there exists a unique covector $\eta \in T^*_x M$ satisfying   
$$
\langle \eta , w \rangle_{T^*_x M , T_x M} = (u,w)_{g(x)}, \quad \text{ for all } w \in  T_x M .
$$
In local coordinates, we have 
$$
\eta_i = g_{i,j}(x) u^j .
$$
We can define an inner product on $T^*_x M$ using the above identification, denoted by $(\cdot,\cdot)_{g^{-1}(x)}$. In local coordinates, we have 
$$
(\eta,\xi)_{g^{-1}(x)} = g^{i,j}(x) \eta_i \xi_j, \quad \text{ where } g^{i,j}(x) = (g(x)^{-1})^{i,j} . 
$$
For all $x \in M$ and all $\eta \in T^*_x M$, we denote by $|\eta|_x = (\eta,\eta)_{g^{-1}(x)}^{\frac12}$ the associated norm.
Let $d\V(x)$ be the canonical Riemannian measure on $M$. In local charts this reads
$$
d\V(x) = \sqrt{|\det(g(x))|} dx_1  \cdots dx_d.
$$
The cotangent bundle $T^*M$ is canonically endowed with a symplectic 2-form $\omega$ (in local charts, $\omega = \sum_{j=1}^d dx_j \wedge d\xi_j$). Let $\omega^d$ be the canonical symplectic volume form on $T^*M$ and by a slight abuse of notation $d \omega^d$ the associated {\em normalized} measure on $T^*M$ (see for instance \cite[p.274]{Hor}). In local coordinates, we have 
$$
d \omega^d = d \xi_1 \cdots d \xi_d \, d x_1 \cdots d x_d.
$$
The canonical projection $\pi : T^*M \to M$ is measurable from $(T^*M , d\omega^d)$ to $(M, d\V)$. For $f \in L^1 (T^*M , d\omega^d)$, we define $\pi_* f \in  L^1 (M , d\V)$ by
$$
\int_M \varphi(x) (\pi_* f)(x) d\V(x) = \int_{T^* M} \varphi \circ \pi (x, \xi) f(x, \xi) d\omega^d(x, \xi), \quad \text{ for all } \varphi \in C^0(M) .  
$$
In local charts, we have 
$$
(\pi_* f)(x) =  \frac{1}{\sqrt{\det(g(x))}}\int_{\R^d} f(x, \xi)  d\xi_1 \cdots d\xi_d .
$$

Note also that we have the following desintegration formula
$$
\int_{T^*M} f(x, \xi) d\omega^d(x, \xi) = \int_M d\V(x) \int_{T^*_x M} f(x, \xi)dm_x(\xi) ,
$$
where the measure $dm_x$ on $T^*_xM$ is given in local charts by
$$
dm_x = \frac{1}{\sqrt{\det(g(x))}}  d\xi_1 \cdots d\xi_d .
$$

Let $V \in W^{2, \infty}(M)$ and define on $T^*M$ the hamiltonian
$$
H(x, \xi) = \frac12|\xi|_x^2 +  V(x) , \quad x \in M , \quad  \xi \in T^*_x M .
$$
We define the associated Hamilton vector field $X_H$, given in local coordinates by
$$
X_H = \nabla_{\xi}H \cdot \nabla_x - \nabla_{x} H \cdot \na_\xi .
$$
Using the $2$-form $\omega$, we can also define the Poisson bracket $\{\cdot , \cdot\}$, see again \cite[p.271]{Hor}. We have $X_H f = \{H, f\}$.

We denote by $\Lambda = \{(x, \xi, \xi'), x \in M, (\xi, \xi') \in T^*_x M \times T^*_x M\}$ the vector bundle over $M$ whose fiber above $x$ is $T^*_x M \times T^*_x M$. 

\bigskip
With these notations, the Boltzmann equation on $T^*M$ can be written as follows, for $(t,x, \xi) \in \R \times T^*M$, 
\begin{align}
\label{Bmanifold}
\pa_t f (t,x, \xi)+ X_H f (t,x, \xi)= 
\int_{T^*_x M} \left[k(x,\xi',\xi) f(t,x,\xi') - k(x,\xi,\xi') f(t,x,\xi)\right] \, dm_x(\xi') 
.
\end{align}

We recover the key properties of the usual linear Boltzmann collision operator on flat spaces. We have, for all $x \in M$,
\begin{equation*} 
 \int_{T^*_x M}  \int_{T^*_x M} \left[k(x,\xi',\xi) f(x,\xi') - k(x,\xi,\xi') f(x,\xi)\right] \, dm_x(\xi') dm_x(\xi) = 0 .
\end{equation*}
Besides, 
$$ 
 \int_{T^* M} (X_H f) (x, \xi) d\omega^d (x, \xi) = 0 ,
$$
since $(X_H f) (x, \xi) d\omega^d$ is an exact form (since $X_H$ is hamiltonian). As a consequence, the mass is conserved: any solution $f$ of \eqref{Bmanifold} satisfies
\begin{equation}
\text{ for all } t \geq 0, \quad \frac{d}{dt} \int_{T^* M} f(t,x,\xi) \, d\omega^d (x, \xi) =0.
\end{equation}
Consider now the (generalized) Maxwellian distribution:
\begin{equation}
 \Mc(x, \xi) := \frac{1}{(2\pi)^{d/2}}e^{-\frac{|\xi|_x^2}{2}}. 
 \end{equation}
Note that for all $x \in M$, $\int_{T^*_x M}  \frac{1}{(2\pi)^{d/2}}e^{-\frac{|\xi|_x^2}{2}} \, dm_x (\xi) = 1$.
As usual, we make the following assumptions on the collision kernel $k$. 

\bigskip

\noindent {\bf A1.} The collision kernel $k \in C^0(\Lambda)$, is nonnegative.

\noindent {\bf A2.} We assume that the Maxwellian cancels the collision operator, that is: 
\begin{equation}
 \int_{T^*_x M}  \left[k(x,\xi' ,  \xi)  \Mc(x, \xi') - k(x,\xi ,  \xi')  \Mc(x, \xi)\right] \, dm_x(\xi')  = 0, \quad \text{for all  } (x,\xi) \in T^*M.
\end{equation}

\noindent {\bf A3.} We assume that
 \begin{equation*}
x \mapsto \int_{T_x^*M \times T_x^*M} k^2(x,\xi',\xi) \frac{\Mc(x,\xi')}{\Mc(x,\xi)} \, dm_x(\xi') dm_x(\xi) \in L^\infty(M). 
\end{equation*}

We can define the characteristics in this Riemannian setting as follows.

\begin{deft}
\label{def-caracmanifold}
Given $(x, \xi) \in T^*M$, we define the hamiltonian flow associated to $H$ by $s \mapsto \phi_s(x , \xi) \in T^*M $: 
\begin{equation}
  \label{eq: geodesic flow}
  \frac{d}{ds}\phi_s(x, \xi) = X_{H} \big(\phi_s(x, \xi)\big), \quad \phi_0(x, \xi) = (x, \xi)  \in T^*M, 
\end{equation}
The characteristics associated to $H$ are the integral curves of this flow.
\end{deft}
Note also for any function $g$ defined on $\R$, $g \circ H$ is preserved along these integral curves, as 
$$
 \left(\frac{d}{ds}g \circ H \circ \phi_s\right)|_{s = s_0} = X_H(g \circ H)(\phi_{s_0}) = \{H,g \circ H\}(\phi_{s_0}) =0 .
$$
In particular, this holds for the function $\frac{e^V}{\Mc} = \frac{1}{(2\pi)^{d/2}} e^{H}$.

With this definition of characteristics, we can then properly define
\begin{itemize}
\item the set $\omega$ where collisions are effective, as in Definition \ref{def-om},
\item $C^-(\infty)$, as in Definition \ref{definitionCinfini},
\item a.e.i.t. GCC, as in Definition \ref{defaeitgcc}, 
\item the Unique Continuation Property, as in Definition \ref{def:UCP},
\item the {\em generalized} Unique Continuation Property, as in Definition \ref{def:UCPgene},
\item the equivalence relations $\sim$ and $\Bumpeq$, as in Definitions \ref{def-sim} and \ref{def-sim-o},
\item the sets $U_j$ as in~\eqref{def-Uj}.
\end{itemize}

Note that in this Riemannian setting, the classes of collision operators {\bf E1}, {\bf E2}, {\bf E3} still make sense, up to some obvious adaptations. 

\bigskip

Let us now introduce the relevant weighted Lebesgue spaces.
   \begin{deft}[Weighted $L^p$ spaces]
We define the Banach spaces $\LL^2$ and $\LL^\infty$ by 
 \begin{equation*}
 \begin{aligned}
 &\LL^2 (T^*M):= \Big\{f \in L^2_{loc}(T^*M), \, \int_{T^*M} | f|^2 \frac{e^V}{\Mc} \, d\omega^d < + \infty \Big\} , 
 \quad \|f\|_{\LL^2} = \left(\int_{T^*M} | f|^2 \frac{e^V}{\Mc} \, d\omega^d \right)^{1/2}.
 \\
 & \LL^\infty(T^*M) : =  \Big\{f \in L^1_{loc}(T^*M), \, \sup_{T^*M} | f| \frac{e^V}{\Mc}  < + \infty \Big\} , 
 \quad \|f\|_{\LL^\infty} = \sup_{T^*M} | f| \frac{e^V}{\Mc} 
  \end{aligned}
  \end{equation*}
 The space $\LL^2$ is a Hilbert space endowed with the inner product
 $$
 \langle f, g \rangle_{\LL^2} := \int_{T^*M} f \, g \frac{e^{V}}{\Mc} \, d\omega^d.
 $$
 \end{deft}
 
 As usual, we have the following well-posedness result for the Boltzmann equation \eqref{Bmanifold}.
  \begin{prop}[Well-posedness of the linear Boltzmann equation]
 \label{prop:WP-manifold}
Assume that $f_0 \in \LL^2$. Then there exists a unique $f\in C^0(\R ;\LL^2)$ solution of~\eqref{Bmanifold} satisfying $f|_{t = 0} =f_0$, and we have
\begin{equation}
\text{ for all } t \geq 0, \quad  \frac{d}{dt} \| f(t)\|_{\LL^2}^2 = - D(f(t)),
\end{equation}
   where
 \begin{multline*}
 D(f) =  \frac{1}{2} \int_{M} e^{V(x)} \int_{T^*_x M}  \int_{T^*_x M}  \left( \frac{k(x,\xi' ,  \xi)}{\Mc(x,\xi)} + \frac{k(x,\xi ,  \xi')}{\Mc(x,\xi')} \right) \\ \times \Mc(x,\xi) \Mc(x,\xi') \left(\frac{f(x,\xi)}{\Mc(x,\xi)}- \frac{f(x,\xi')}{\Mc(x,\xi'))}\right)^2 \, dm_x (\xi) \,dm_x(\xi') \, d\V(x).
 \end{multline*}
 If moreover $f_0 \geq 0$ a.e., then for all $t \in \R$ we have $f(t, \cdot,\cdot)\geq 0$ a.e. (Maximum principle).

 \end{prop}

 More generally, all results of Section~\ref{preliminaries} (up to obvious adaptations) are still relevant.

The crucial point we have to check now concerns velocity averaging lemmas for kinetic transport equations on a Riemannian manifold.
\begin{lem}
\label{lem-moyenne-variete}
Let $H$ be defined as above, and $X_H$ the associated vector field. Let $T>0$ and $\Psi \in C^\infty_c(T^*M)$. There exists $C>0$ a constant such that the following holds. For any $f,g \in L^2((0,T)\times T^*M)$ satisfying
$$
\partial_t f + X_H f = g,
$$ 
we have
$$
\| \pi_*(f \Psi) \|_{H^{1/4}((0,T) \times M)} \leq C (\| f|_{t=0}\|_{L^2((0,T) \times T^*M)} +  \|g\|_{L^2((0,T) \times T^*M)} ).
$$
i.e. 
$$
\left\| \int_{T_x^*M}f \Psi dm_x \right\|_{H^{1/4}((0,T) \times M)} \leq C (\| f|_{t=0}\|_{L^2((0,T) \times T^*M)} +  \|g\|_{L^2((0,T) \times T^*M)}).
$$
\end{lem}

\begin{rque}
Assuming that $V$ is smooth enough, we may obtain the optimal Sobolev regularity $H^{1/2}$ (instead of $H^{1/4}$), see Remark \ref{rk-smooth}.
\end{rque}

\begin{proof}[Proof of Lemma \ref{lem-moyenne-variete}]
In local charts, we have
$$
 \pi_*(f \Psi)(x,\xi) = \int_{\R^d} {f}(x,\xi) \Psi(x,\xi) \frac{1}{\sqrt{\det(g(x))}} \, d\xi.
$$
and $f$ satisfies the kinetic equation 
$$
\pa_t f + g^{i,j}(x) \xi_j \pa_{x_i} f - \left(\frac12 \pa_{x_i} g^{j,k}(x) \xi_j \xi_k + \pa_{x_i} V(x) \right) \pa_{\xi_i} f = g.
$$
We use the change of variables $\overline{f}(t,x,v^i) = {f}(t,x,g_{i,j}v^j)$ (we define as well $\overline{g}$ and $\overline{\Psi}$), which satisfies  the equation 
$$
\pa_t \overline{f} + v^j \pa_{x_i} \overline{f} - \left(\Gamma^i_{j,k}(x)v^j v^k+ \pa_{x_i} V(x) \right) \pa_{v_i} \overline{f} = \overline{g},
$$
where $\Gamma^i_{j,k}(x) = \frac12 g^{i \ell}(x)\left(\pa_{x_j}g_{k \ell}(x) + \pa_{x_k}g_{j \ell}(x) - \pa_{x_\ell}g_{j k}(x) \right)$ are the Christoffel symbols. 
Using a classical averaging lemma (see \eqref{avlemmaRd} in Theorem~\ref{thmmoyenne} in Appendix~\ref{section-AL} with $m=1$ and $s=0$), we deduce that
$$
\left\|\int_{\R^d} \overline{f} \, \overline{\Psi}  \sqrt{\det(g(x))} \, dv \right\|_{H^{1/4}((0,T) \times \R^d \times \R^d)} \leq C (\| \overline{f}|_{t=0}\|_{L^2((0,T) \times \R^d\times \R^d)} +  \|\overline{g}\|_{L^2((0,T) \times\R^d\times \R^d)} ) .
$$
Going back to the original variables, we deduce that 
$$
\left\|\int_{\R^d} {f} \, {\Psi}  \frac{1}{\sqrt{\det(g(x))}} \, dv \right\|_{H^{1/4}((0,T) \times \R^d \times \R^d)} \leq C (\| {f}|_{t=0}\|_{L^2((0,T) \times\R^d\times \R^d)} +  \|{g}\|_{L^2((0,T) \times\R^d\times \R^d)} ),
$$
which proves our claim.
\end{proof}

Equipped with this tool (more generally the analogues of all averaging lemmas of Appendix~\ref{section-AL} can be obtained as well), we have the following analogue of the general convergence result of Theorem~\ref{thmconv-general} (which includes Theorems \ref{thmconv-intro} and \ref{thmconvgene-intro}). The same proof applies with only minor adaptations. Recall that the sets $(U_j)_{j \in J}$ are defined in~\eqref{def-Uj}.

   \begin{thm}
\label{thmconv-manifold}
The following statements are equivalent.
\begin{enumerate}[(i.)]
\item  The set $\omega$ satisfies the {\bf generalized} Unique Continuation Property.
\item The set $\omega$ satisfies the a.e.i.t. GCC.
\item  For all $f_0 \in \LL^2$, denote by $f(t)$ the unique solution to \eqref{Bmanifold} with initial datum $f_0$. We have
\begin{equation}
 \left\|f(t)-Pf_0 \right\|_{\LL^2} \to_{t \to +\infty} 0,
 \end{equation}
where
\begin{equation}
P f_0 (x,v) = \sum_{j\in J}  \frac{1}{\| \mathds{1}_{U_j} e^{-V} \Mc \|_{\LL^2}} \left( \int_{U_j} f_0 \, d\omega^d \right)f_j.
\end{equation}
and the $(U_j)_{j \in J}$ are defined in~\eqref{def-Uj} and the $f_j = \frac{ \mathds{1}_{U_j} e^{-V} \Mc}{\| \mathds{1}_{U_j} e^{-V} \Mc \|_{\LL^2}}$.
\item  For all $f_0 \in \LL^2$, denote by $f(t)$ the unique solution to \eqref{Bmanifold} with initial datum $f_0$. We have
\begin{equation}
 \left\|f(t)-Pf_0 \right\|_{\LL^2} \to_{t \to +\infty} 0,
 \end{equation}
 where $Pf_0$ is a stationary solution of~\eqref{Bmanifold}.
 \end{enumerate}
 \end{thm}

We obtain as well the analogue of Theorem~\ref{thmexpo-intro}. As in the torus case, we make the additional technical assumption:
\noindent{\bf A3'.}
Assume that there exists a continuous function $\varphi(x,\xi):= \Theta \circ H(x,\xi)$, with $\varphi \geq 1$, such that for all $(x,\xi) \in T^*M$, we have
$$
\int_{T^*_xM} k(x,\xi,\xi’) \, dm_x(\xi’) \leq \varphi(x,\xi)
$$
and 
$$
 \sup_{x \in M } \int_{T^*_xM \times T^*_xM} k^2(x,\xi’,\xi) \frac{\Mc(\xi’)}{\Mc(\xi)}  \left(\frac{\varphi(x,\xi)}{\varphi(x,\xi')}-1\right)^2 \, dm_x(\xi) dm_x(\xi’) <+\infty
$$

\begin{thm}[Exponential convergence to equilibrium]
\label{thmexpo-manifold}
Assume that the collision kernel satisfies {\bf A3'}. The following statements are equivalent:
\begin{enumerate}[(a.)]
\item $C^-(\infty) > 0$. 
\item There exists $C>0, \gamma>0$ such that for any $f_0 \in \LL^2$, the unique solution to \eqref{Bmanifold} with initial datum $f_0$ satisfies for all $t\geq 0$
\begin{multline}
\label{decexpo-manifold} 
\left\|f(t)-\left( \int_{T^*M} f_0 \, d\omega^d \right) \frac{e^{-V(x)}}{\int_{M} e^{-V(x)} \, d\V(x)}  \Mc\right\|_{\LL^2}  \\
\leq C e^{-\gamma t}\left\|f_0-\left( \int_{T^*M} f_0 \, d\omega^d \right) \frac{e^{-V(x)}}{\int_{M} e^{-V(x)} \, d\V(x)}  \Mc\right\|_{\LL^2} .
\end{multline}
\item There exists $C>0, \gamma>0$ such that for any $f_0 \in \LL^2$, there exists $Pf_0$ a stationary solution of~\eqref{Bmanifold} such that the unique solution to \eqref{Bmanifold} with initial datum $f_0$ satisfies for all $t \geq 0$,
\begin{equation}
\label{decexpogene-manifold} 
\left\|f(t)-Pf_0\right\|_{\LL^2} 
\leq C e^{-\gamma t} \left\|f_0-Pf_0\right\|_{\LL^2}.
\end{equation}
\end{enumerate}
\end{thm}

As a particular case of Theorem~\ref{thmconv-manifold}, we have the following corollary.
\begin{coro}
Assume that $V=0$ and $\omega=T^*\omega_x$, where $\omega_x$ is a non-empty open subset of $M$. 
Suppose that the dynamics associated to $(\phi_t)_{t \geq 0}$ on 
$$
S^* M = \left\{ (x,\xi) \in T^*M, \, \frac12 |\xi|_x^2 = 1 \right\},
$$
is ergodic. 
Then for all $f_0 \in \LL^2$, denoting by $f(t)$ the unique solution to \eqref{Bmanifold} with initial datum $f_0$, we have
\begin{equation}
 \label{convergeto0-manifold-re}
 \left\|f(t)-\left(\int_{T^*M} f_0 \, d\omega^d\right)\frac{1}{|\V(M)|} \Mc(v)\right\|_{\LL^2} \to_{t \to +\infty} 0,
 \end{equation}
\end{coro}
Note that if the dynamics of $(\phi_t)_{t \geq 0}$ is ergodic on $S^* M$, then it is also ergodic on cosphere bundles of any positive radius (since for $V=0$, the flow is homogeneous of degree one).

Classical examples of Riemannian manifolds satisfying this dynamical assumption are given by compact Riemannian manifolds with negative curvature.

\section{The case of compact phase spaces}
\label{compactPS}

Instead of studying the linear Boltzmann equation on the ``whole'' phase space $T^*M$ or $\Omega \times \R^d$, it is possible to consider this equation set on the ``reduced'' phase spaces
\begin{align*}
B_H^*M &= \{(x, \xi) \in T^*M, H(x, \xi) \leq R\},\quad 
\quad
 S_H^*M = \{(x, \xi) \in T^*M , H(x, \xi) = R\}, \\
 &\text{or} \quad  \mathcal{R}_H^*M = \{(x, \xi) \in T^*M , R\leq H(x, \xi) \leq R'\},
\end{align*}
for $R'>R>0$, or (with similar definitions) on $B_H^*\Omega$, $S_H^*\Omega$ or $ \mathcal{R}_H^*\Omega$. Note that by continuity, the potential $V$ is always bounded from below (and above), so that  $B_H^*M$, $ S_H^*M$ and  $\mathcal{R}_H^*M$ (as well as $B_H^*\Omega$, $S_H^*\Omega$ and  $\mathcal{R}_H^*\Omega$) are compact. The linear Boltzmann equation~\eqref{B} is well-posed in $L^2(B_H^*M)$ (resp. $L^2(B_H^*\Omega)$), $L^2(S_H^*M)$ (resp. $L^2(B_H^*\Omega)$) or $L^2(\mathcal{R}_H^*M)$ (resp. $L^2(\mathcal{R}_H^*\Omega)$), in particular because the hamiltonian is preserved by the dynamics.
The case of $S_H^*M$ is for instance relevant for the equations of radiative transfer or neutronics.

\bigskip

The analogues of Theorems \ref{thmconv-intro}, \ref{thmexpo-intro}, \ref{thmconv-general} still hold in this framework, replacing the former phase space by $B_H$, $S_H$ or $\mathcal{R}_H$ in the various geometric conditions. For the sake of conciseness, we do not write these results again. All proofs remain valid, with some simplifications, since the phase space is now compact. Note that the fact that $C^-(\infty)>0$ is equivalent to GCC in this compact case.


\part*{Appendices}
\appendix

\section{A stabilization criterion}
\label{stabob}
We provide here a characterization of exponential decay for dissipative evolution equations. The following lemma is very classical and we reproduce it here for the convenience of the reader.
 \begin{lem}
 \label{lemfondamental}
 Consider the evolution equation
 \begin{equation}
 \label{pde}
 \left\{
 \begin{aligned}
 \partial_t f +  L f = 0, \\
 f_{|t=0} = f_0,
 \end{aligned}
 \right.
 \end{equation}
 assumed to be:
 \begin{itemize}
 \item globally wellposed in some functional space $X$ in the sense that for any $f_0 \in X$, there is a unique $f \in C^0_t(X)$ solution to \eqref{pde},
 \item invariant by translation in time, in the sense that if $f \in C^0_t(X)$ is the solution of \eqref{pde}, then for all $t_0\geq 0$, $g(t) := f(t+t_0)$ is the unique solution of
   \begin{equation}
 \left\{
 \begin{aligned}
 \partial_t g +  L g = 0, \\
g_{|t=0} = f_{|t=t_0}.
 \end{aligned}
 \right.
 \end{equation}
 \end{itemize}
  Let $E(f)$ and $D(f)$ be two non-negative functionals defined for all $f \in X$, and such that if $f$ is a solution to \eqref{pde},  
  \begin{equation}
  \label{eqidentity}
\text{ for all } t \geq 0, \quad  \frac{d}{dt} E(f(t)) = - D(f(t)).
 \end{equation}
 
 Then, the following two properties are equivalent:
 \begin{enumerate}
\item  There exist $C,\gamma>0$ such that for all $f(0) \in X$, the associated solution $f$ to \eqref{pde} satisfies
 \begin{equation}
 \label{eqfonda1}
 \text{ for all } t \geq 0, \quad E(f(t))\leq C e^{-\gamma t} E(f(0)).
 \end{equation}
 
 \item There exists $T>0$ and $K>0$ such that for all $f(0) \in X$, the associated solution $f$ to \eqref{pde} satisfies
  \begin{equation}
  \label{eqfonda2}
K \int_0^T D(f(t)) \, dt \geq  E(f(0)).
 \end{equation}
 
 \end{enumerate}
 \end{lem}
 For the sake of completeness, we provide a short proof of this lemma.
  
 \begin{proof}[Proof of Lemma \ref{lemfondamental}] 
 \noindent $(1) \Rightarrow (2)$ Assume that $(1)$ holds. Let $T_0>0$ such that $C e^{-\gamma T_0} = \frac{1}{2}$. Then, after integrating \eqref{eqidentity} betwen $0$ and $T_0$, we have:
 $$
 E(f(T_0))- E(f(0)) = - \int_0^{T_0}  D(f(t)) \, dt,
 $$
 so that, by \eqref{eqfonda1},
 $$
 \int_0^{T_0}  D(f(t)) \, dt \geq E(f(0)) -  C e^{-\gamma T_0} E(f(0)) = \frac{1}{2} E(f(0)),
 $$
 and we can therefore take $T =T_0$ and $C=2$ in \eqref{eqfonda2}.
 
 \bigskip
 
  \noindent $(2) \Rightarrow (1)$ Assume that $(2)$ holds. Here (and only here), we need the property of invariance by time translations for~\eqref{pde}. By \eqref{eqidentity} and  \eqref{eqfonda2}, we have
  $$
  E(f(T))  \leq \left(1-\frac{1}{K} \right) E(f(0)).
  $$
  Note that the assumption $E(f) \geq 0$ implies in particular that $K\geq 1$. We may assume that $K>1$. Indeed, for $K=1$, we have $E(f(t)) = 0$ for all $t \geq T$ so that for any $\gamma>0$, there exists $C>0$ such that (1) holds.
  By invariance by translation in time of~\eqref{pde}, one likewise obtains
   $$
  E(f(2T))  \leq \left(1-\frac{1}{K} \right) E(f(T)).
  $$
  Thus, by a straightforward induction, for any $k \in \N$, we have the bound:
     $$
  E(f(kT))  \leq \left(1-\frac{1}{K} \right)^k E(f(0)).
  $$
  Defining $\gamma_0 := \frac{- \log\left(1-\frac{1}{K} \right)}{T}>0$ and $C_0 := \left(1-\frac{1}{K} \right)^{-1} = e^{\gamma_0 T}>0$, we can now check that 
   \begin{equation}
 \text{ for all } t \geq 0, \quad E(f(t))\leq C_0 e^{-\gamma_0 t} E(f(0)).
 \end{equation}
 Indeed, let $t\geq 0$ and $k \in \N$ such that $t \in [kT, (k+1) T[$; since $E(f(\cdot))$ is decreasing (see~\eqref{eqidentity}), we have
 $$
 E(f(t)) \leq   E(f(kT))   \leq \left(1-\frac{1}{K} \right)^k E(f(0)) = e^{-\gamma_0 k T} E(f(0)) \leq C_0 e^{-\gamma_0 t} E(f(0)),
 $$
which concludes the proof.
 \end{proof}

 \section{Velocity averaging lemmas}
 \label{section-AL}

Velocity averaging lemmas play an important role in many proofs of this paper. In this appendix, we recall some classical results and also state the versions precisely adapted to our needs.

Kinetic transport equations are hyperbolic partial differential equations and as it can be seen from Duhamel's formula, there is propagation of potential singularities at initial time and/or from a source in the equations. Thus there is no hope that the solution of a kinetic equation becomes more regular than the initial condition.
 
 It was nevertheless observed by Golse, Perthame and Sentis \cite{GPS} that the averages in velocity of the solution of a kinetic transport equation enjoy extra regularity/compactness properties (see also the independent paper of Agoshkov \cite{Ago}).
 We refer to the by now classical paper of  Golse, Lions, Perthame, Sentis \cite{GLPS}, DiPerna, Lions \cite{DPL1}, DiPerna, Lions, Meyer \cite{DPLM}, B\'{e}zard \cite{B94} for quantitative estimates of this compactness property in various settings of increasing complexity.

We also refer to the review paper of Jabin \cite{Jab} and to the recent work of Ars\'enio and Saint-Raymond~\cite{ASR}.

 \subsection{Velocity averaging lemmas in $\R^d$}

We start by recalling classical averaging lemmas in the whole space $\R^d$. There are also versions of these lemmas for $p\in (1,\infty)$, but we stick to the case $p=2$, which is sufficient for our needs.

  \begin{thm}[Kinetic averaging lemma \cite{GLPS,DPL1,DPLM,B94}]
  \label{thmmoyenne}
 Let $s\in [0,1)$ and $m \in \R^+$.
  \begin{enumerate}
 \item For any $T>0$ and any {\em bounded} open sets $\Omega_x ,\Omega_v \subset \R^d$, there exists a constant $C>0$ such that for all $\Psi \in C^{\infty}_c(\mathbb{R}^d)$ supported in $\Omega_v$ and all $f,g \in  L^2_{loc}(\R \times \R^d \times \R^d)$ satisfying
  $$
  \partial_t f + v \cdot \nabla_x f = (1- \Delta_{t,x})^{s/2}(1- \Delta_{v})^{m/2} g,
  $$ 
we have
\begin{equation}
\label{avlemmaRd}
\| \rho_{\Psi} \|_{H^{\alpha}([0,T] \times \Omega_x)} \leq C \left(\| f \|_{L^2([0,T] \times \Omega_x \times \Omega_v )}  + \| g \|_{L^2([0,T] \times \Omega_x \times \Omega_v )}\right),
\end{equation}
where $\rho_{\Psi}(t,x):=\int_{\R^d} f(t,x,v)\Psi(v)dv$ and $\alpha= {\frac{(1-s)}{2(1+m)}}$. 
 \item Let $T>0$ and $(f_n)_{n\in \N}$ and $(g_n)_{n\in \N}$ be two sequences of $L^2(0,T ; L^2_{loc}( \R^d \times \R^d))$ such that the following holds
 $$
 \partial_t f_n + v \cdot \nabla_x f_n = (1- \Delta_{t,x})^{s/2}(1- \Delta_{v})^{m/2} g_n,
 $$ 
 with $s \in [0,1), m \geq 0$.
Assume that for any {\em bounded} open sets $\Omega_x ,\Omega_v \subset \R^d$, there exists $C_1>0$, such that for all $n \in \N$,
 \begin{equation}
 \label{ass2.0}
\| f_n \|_{L^2((0,T) \times \Omega_x \times \Omega_v)}  + \| g_n \|_{L^2((0,T) \times \Omega_x \times \Omega_v)} \leq C_1.
 \end{equation}
Then, for any $\Psi \in C^{\infty}_c(\mathbb{R}^d)$, the sequence $(\rho_{\Psi,n})_{n\in \N}$ defined for $n \in \N$ by $$\rho_{\Psi,n}(t,x):=\int_{\R^d} f_n(t,x,v)\Psi(v)dv$$ is relatively compact in 
$L^2(0,T ; L^2_{loc}( \R^d))$.
\end{enumerate}
  \end{thm}

\begin{rque}
\label{rk-smooth}
In the main part of the paper, we apply this averaging lemma to the Boltzmann equation \eqref{B} by writing it under the form 
$$
\partial_t f + v \cdot \nabla_x f = \nabla_x V \cdot \nabla_v f +  \int_{\R^d} \left[k(x,v' ,  v) f(v') - k(x,v ,  v') f(v)\right] \, dv' .
$$
To this end, we consider the case $s=0, m=1$ in Theorem~\ref{thmmoyenne}. This implies that the averages in $v$ belong to the Sobolev space $H^{1/4}$.

Nevertheless, assuming that the potential $V$ is smooth enough, we can also use the approach of Berthelin-Junca \cite{BJ} to obtain the optimal Sobolev space $H^{1/2}$ for these averages (which is not needed in this paper). 
\end{rque}

\begin{rque}
We also have a version of these lemmas for kinetic transport equations set in general Riemannian manifolds, see Lemma \ref{lem-moyenne-variete}.
\end{rque}

We now state the result as needed in the main part of this work.  
Assuming an extra uniform integrability, we can deduce some compactness on moments of $f$ without having to consider compactly supported test functions in $v$. This is the purpose of the next result, which is actually the version of averaging lemmas used most of the time in this work.

\begin{coro}
\label{lemmoyenne}
 Let $\Omega_x$ be a bounded open set of $\R^d$, $T>0$, and $(f_n)_{n\in \N}$, $(g_n)_{n\in \N}$ be two sequences of $L^2(0,T ; L^2_{loc}( \Omega_x \times \R^d))$ satisfying
 $
 \partial_t f_n + v \cdot \nabla_x f_n = (1- \Delta_{v})^{m/2} g_n,
 $ 
 for some $ m \geq 0$.
 Suppose that there exists $V \in L^\infty$ such that for any {bounded} open set $\Omega_v \subset \R^d$, there exists $C_0>0$ such that, for any $n \in \N$, 
 \begin{equation}
 \label{eqmoybound1}
\text{ for all } t\geq 0,\quad \| f_n\|_{\LL^2(\Omega_x \times \R^d)}^2  :=  \int_{\Omega_x} \int_{\R^d} |f_n|^2 \,  \frac{e^{V(x)}}{\Mc(v)} \, dv \, dx \leq C_0, \qquad  \| g_n \|_{L^2((0,T) \times \Omega_x \times \Omega_v)} \leq C_0 .
 \end{equation}
Assume moreover that there is $f \in L^\infty(0,T ;L^2(\Omega_x \times \R^d))$ such that $f_n \rightharpoonup f$ weakly$-\star$ in $L^\infty(0,T ; \LL^2(\Omega_x \times \R^d))$. Consider  $\rho_{n}(t,x):=\int_{\R^d}  f_n(t,x,v) \,dv$. 
 Then up to a subsequence, we have
\begin{equation}
\label{eqconclu1}
\rho_n \Mc(v) \rightarrow \left(\int_{\R^d} f \, dv\right) \Mc(v) , \quad \text{strongly in  } L^2(0,T ;\LL^2(\Omega_x \times \R^d)) ,
\end{equation} 
and for any continuous kernel $k(\cdot,\cdot,\cdot): \R^d \times \R^d \times \R^d \to \R$ satisfying {\bf A3}, we have
\begin{equation}
\label{eqconclu2}
\int_{\R^d} k(x,v',v) f_n(t,x,v') \, dv'  \rightarrow \int_{\R^d} k(x,v',v) f(t,x,v') \, dv'  , \quad \text{strongly in  } L^2(0,T ; \LL^2(\Omega_x \times \R^d)) .
\end{equation}
\end{coro}

 \begin{proof}[Proof of Corollary \ref{lemmoyenne}]
 First note that by Fatou's lemma, $f$ satisfies:
   \begin{equation}
   \label{eqmoybound2}
\text{ for all } t\geq 0, \quad \|f\|_{\LL^2}^2 = \int_{\Omega_x} \int_{\R^d} |f|^2 \,  \frac{e^{V(x)}}{\Mc(v)} \, dv \, dx \leq C_0.
 \end{equation}
 
 Since $\rho_n$ does not depend on $v$, proving \eqref{eqconclu1} is equivalent to show that:
 \begin{equation}
\rho_n e^{V} \rightarrow \int_{\R^d} f \, dv \, e^{V}, \quad \text{strongly in  } L^2((0,T) \times \Omega_x).
\end{equation}
  Let $\Psi \in C^{\infty}_c(\R)$ such that $\Psi = 1$ in a neighborhood of $0$, and define $\Psi_R(v) = \Psi(\frac{|v|}{R})$, $v\in \R^d$. 
 
 By Theorem \ref{thmmoyenne}, we can assume, up to a subsequence, that 
  \begin{equation}
  \label{convcutoff}
 \rho_{\Psi_R,n} := \int_{\R^d} f_n \Psi_R \, dv  \rightarrow  \int_{\R^d} f \Psi_R \, dv, \quad \text{strongly in  } L^2((0,T) \times \Omega_x).
  \end{equation}
  Let $\eps >0$. We can write the decomposition:
 \begin{align*}
&  \rho_n - \int_{\R^d} f \, dv =  A_1 + A_2 +A_3 , \quad \text{with }\\
& A_1 = \left( \rho_{\Psi_R,n} - \int_{\R^d} f \Psi_R \, dv \right) , \quad
A_2 = \int_{\R^d} f_n (1-\Psi_R) \, dv , \quad
A_3 = -  \int_{\R^d} f (1-\Psi_R) \, dv .
 \end{align*}
 By Cauchy-Schwarz inequality,  
 using \eqref{eqmoybound1}, we have for all $n \in \N$ and all $t \in (0,T)$, 
  \begin{align*}
 \left\|  \int_{\R^d} f_n (1-\Psi_R) \, dv \,e^{V} \right \|_{L^2(\Omega_x)}^2 &\leq  \int  \left(\int_{\R^d} |f_n|^2 \frac{1}{\Mc(v)}  \, dv\right)  \left(\int_{\R^d} (1-\Psi_R)^{2}{\Mc(v)}  \, dv\right) e^{2V(x)} \, dx\\
 &\leq C_0  \|e^V\|_{L^\infty(\Omega_x)} \left( \int_{\R^d} (1-\Psi_R)^{2}{\Mc(v)}  \, dv \right).
 \end{align*}
As a consequence, there exists $R_0>0$ large enough such that for all $R\geq R_0$ and all $n \in \N$, we have
$$
\|e^V A_2\|^2_{L^2((0,T) \times \Omega_x)} \leq \frac{\eps}{3} .
$$
Likewise, we use \eqref{eqmoybound2} to  get  for all $n \in \N$, 
 $$
 \| e^{V} A_3 \|_{L^2((0,T) \times \Omega_x)}^2 \leq \frac{\eps}{3}.
 $$
 Using \eqref{convcutoff} ($R$ is now fixed), there is $N\geq 0$ such that for any $n\geq N$, 
 $$
\left\|  \left(  \rho_{\Psi_R,n} - \int_{\R^d} f \Psi_R \, dv\right)  e^{V}\right\|_{L^2((0,T)\times \Omega_x)} \leq \eps/3,
 $$
 from which we infer that 
 \begin{equation}
 \left\|  \left(  \rho_n - \int_{\R^d} f \, dv\right) e^{V}\right\|_{L^2((0,T)\times \Omega_x)} \leq \eps
 \end{equation}
 and this concludes the proof of \eqref{eqconclu1} . 
 
 \bigskip

 For the proof of \eqref{eqconclu2}, let us first assume for a while that $k$ is smooth (namely for all $x$, $k(x, \cdot, \cdot)$ belongs to the $C^\infty$ class).
 We first have to be careful about the integration in the velocity variable. The convergence in~\eqref{eqconclu2} results from the following two facts:
 \begin{itemize}
 \item For all $v \in \R^d$, we have the following  convergence
 $$
 \int_{\R^d} k(x,v',v) f_n(t,x,v') \, dv'  \rightarrow \int_{\R^d} k(x,v',v) f(t,x,v') \, dv'  \quad \text{strongly in  } L^2(0,T;L^2_{x}(\Omega_x)).
 $$
 This follows from a truncation argument and Theorem \ref{thmmoyenne}, exactly as for $\rho_n$. Keeping the same notations, the only difference is that we have to study 
 \begin{multline*}
 \int_{\R^d} (1-\Psi_R)^{2}k(x,v',v){\Mc(v')}  \, dv' \\
 \leq  \left( \int_{\R^d} (1-\Psi_R)^{2}k^2(x,v',v) {\Mc(v')}  \, dv' \right)^{1/2}  \left( \int_{\R^d} (1-\Psi_R)^{2}{\Mc(v')}  \, dv' \right)^{1/2},
 \end{multline*}
which is, using {\bf A3}, small for $R$ large enough.
 \item By the Cauchy-Schwarz inequality and the bound~\eqref{eqmoybound1}, we have
 \begin{align*}
  &\int_0^T \int_{\Omega_x}  \left(\int_{\R^d} k(x,v',v) (f_n-f)(t,x,v') \, dv' \right)^2 \frac{e^V}{\Mc(v)}\, dx \, dt \\
& \leq \int_0^T  \left(\sup_{x \in \Omega_x}  \int_{\R^d} k^2(x,v',v) \frac{\Mc(v')}{\Mc(v)} \, dv'   \right)  \int_{\Omega_x}\left(\int_{\R^d} \frac{|f_n-f|^2(t,x,v')}{\Mc(v')}\, dv' \right) e^V \, dx  \, dt \\
 & \leq   C_0 T  \sup_{x \in \Omega_x}  \int_{\R^d} k^2(x,v',v) \frac{\Mc(v')}{\Mc(v)} \, dv' ,
  \end{align*}
 which is independent of $n$ and in $L^1(dv)$, since by {\bf A3}, we have
 $$
 \sup_{x \in \Omega_x}  \int_{\R^d \times \R^d} k^2(x,v',v) \frac{\Mc(v')}{\Mc(v)} \, dv' \, dv < +\infty.
 $$
 \end{itemize}
 Hence, by Lebesgue dominated convergence theorem, we deduce \eqref{eqconclu2}.

We now use an approximation argument to handle the general case, i.e. when $k$ is only assumed to be continuous. Consider $(\phi_\delta)_{\delta>0}$ a family of mollifiers in $C^\infty_c(\R^d \times \R^{d})$ for the measure $ \Mc(v)\Mc(v') \, dv' dv$. We set for all $x,v,v'$
$$
\tilde{k}_\delta(x,v,v') =\left(\tilde{k}(x,\cdot,\cdot) \star \phi_\delta(\cdot,\cdot)\right)(v,v'), \quad k_\delta(x,v,v')= \tilde{k}_\delta(x,v,v') \Mc(v').
$$
We use the following classical properties of mollifiers: 
\begin{itemize}
\item for all $x$, $k_\delta(x, \cdot,\cdot)$ is in the $C^\infty$ class;
\item we have for all $\delta>0$
\begin{equation}
\label{prop1mol}
\sup_{x \in \Omega_x} \| \tilde k-  k_\delta \|_{L^2(\Mc(v)\Mc(v') \, dv' dv)}  \to_{\delta\to0} 0.
\end{equation}
\end{itemize}
Let $\eps>0$. We write the decomposition
$$
\left\|\int_{\R^d} k(x,v',v) f_n(t,x,v') \, dv'  - \int_{\R^d} k(x,v',v) f(t,x,v') \, dv' \right\|_{\LL^2}^2
 \leq 2A_1 + 2 A_2 ,
$$
with
\begin{align*}
A_1 &= \left\|\int_{\R^d} k_\delta(x,v',v) f_n(t,x,v') \, dv'  - \int_{\R^d} k_\delta(x,v',v) f(t,x,v') \, dv' \right\|^2_{\LL^2}, \\
A_2 &=  \left\|\int_{\R^d} (k-k_\delta)(x,v',v) (f_n-f)(t,x,v') \, dv' \right\|^2_{\LL^2}.
\end{align*}
We estimate $A_2$ as follows, using~\eqref{prop1mol}
\begin{align*}
A_2 &= \int_{\R^d} \Mc(v)  \int_{\Omega_x}  \left(\int_{\R^d} (\tilde{k}-\tilde{k}_\delta) (x,v',v) (f_n-f)(t,x,v') \, dv' \right)^2 e^V \, dx   \, dv \\
    & \leq    \sup_{x\in \Omega_x} \int_{\R^d\times \R^d} |\tilde{k}-\tilde{k}_\delta|^2 (x,v',v) \Mc(v')\Mc(v) \, dv' \, dv  \left( \int_{\Omega_x} \int_{\R^d}  \frac{|f_n-f|^2(t,x,v')}{\Mc(v')} e^V   \, dv' \, dx \right) \\
    &\leq 4 C_0^2   \sup_{x\in \Omega_x} \int_{\R^d\times \R^d} |\tilde{k}-\tilde{k}_\delta|^2 (x,v',v) \Mc(v')\Mc(v) \, dv' \, dv
\end{align*}
Using~\eqref{prop1mol}, we fix $\delta>0$ small enough so that for all $n\in \N$,
$$
A_2 \leq \big(\eps/(4T)\big)^{1/2}.
$$
and thus for all $n\in \N$, we have
$$
\| A_2 \|^2_{L^2(0,T)} \leq \eps/4.
$$
For $A_1$, we use the above analysis in the smooth case to deduce that we can take $N$ large enough to get for all $n\geq N$,
$$
\| A_1 \|^2_{L^2(0,T)}  \leq \eps/4.
$$
Finally, we have proven that for any $\eps>0$, there is $N$ such that for all $n\geq N$,
$$
\left\|\int_{\R^d} k(x,v',v) f_n(t,x,v') \, dv'  - \int_{\R^d} k(x,v',v) f(t,x,v') \, dv' \right\|_{L^2(0,T;\LL^2)}^2 \leq \eps,
$$
which concludes the proof of the convergence.

 \end{proof}

\subsection{Velocity averaging lemmas in open sets with boundary}
We now consider the case of equations set in open sets of $\R^d$ with boundary.

A first result is the following localized averaging lemma, which shows the interior regularity of velocity averages. It is obtained from the whole space case after a standard localization procedure and does not depend on the prescribed boundary conditions.

\begin{coro}
\label{thmmoyenne-domain}
Let $\Omega$ be an open set of $\R^d$ and $m\geq 0$. Let $f,g \in  L^2_{loc}(\R^+, L^2(\Omega \times \R^d))$ satisfying
\begin{align}
&\partial_t f + v \cdot \nabla_x f  = (1- \Delta_{v})^{m/2}g, \quad (x,v) \in \Omega \times \R^d. 
\end{align}

 Then for all $\psi \in C^{1}_c(\Omega \times  \mathbb{R}^d)$,  and all $T>0$, there exists $C>0$, such that 
 $$
 \rho_{\psi}(t,x):=\int_{\R^d} f(t,x,v)\psi(x,v)dv
 $$
satisfies
\begin{equation}
\| \rho_{\psi} \|_{H^{\frac{1}{2(1+m)}}((0,T) \times \Omega)} \leq C \left(\| f \|_{L^2(0,T; L^2(\Omega \times \R^d))}  + \| g \|_{L^2(0,T ;L^2(\Omega \times \R^d))}\right).
\end{equation}
  \end{coro}

We now formulate another compactness result with an additional uniform equi-integrability assumption on the sequence. Note here that only compactness is obtained, not uniform regularity up to the boundary, and that we again do not use the boundary conditions.

 The problem of finding uniform regularity up to the boundary (without the equi-integrability assumption) seems difficult and is clearly beyond the scope of this paper.
For this question, the precise boundary condition satisfied by the solution of the kinetic transport equation should play a key role (whereas it does not play any role in the results of this section).

\begin{deft}[Uniform equi-integrability]
\label{def-equi}
Let $d\mu$ be a positive measure on the phase space $\Omega \times \R^d$. We say that a sequence $(g_n)_{n \in \N}$ of $L^1(d\mu)$ is equi-integrable (with respect to $d\mu$) if for any $\eps>0$, there exists $\delta >0$ such that for any measurable subset $A \subset \Omega \times \R^d$ satisfying $\mu(A) \leq \delta$, we have
\begin{equation}
\sup_{n \in \N} \int_{A} |f_n| d \mu \leq \eps.
\end{equation}
\end{deft}

\begin{coro}
\label{thmmoyenne-domain2}
Let $\Omega$ be an open subset of $\R^d$ and $V \in L^\infty(\Omega)$. Fix $T>0$ and $m \geq 0$.
Let $(f_n)_{n\in \N}$ and $(g_n)_{n\in \N}$ be two sequences of $L^2(0,T ; L^2(\Omega \times \R^d))$ satisfying in $D'(\R^+ \times \Omega \times \R^d)$ the equation
 $$
 \partial_t f_n + v \cdot \nabla_x f_n = (1- \Delta_{v})^{m/2} g_n .
 $$ 
Assume that for all open sets $\Omega_x \subset \Omega, \Omega_v \subset \R^d$ such that $\overline{\Omega_x} \subset \Omega$, there exists $C_1>0$, such that for all $n \in \N$,
 \begin{equation}
 \label{ass1-boun}
\| f_n \|_{L^2((0,T) \times \Omega_x \times \Omega_v)}  + \| g_n \|_{L^2((0,T) \times \Omega_x \times \Omega_v)} \leq C_1
 \end{equation}
and that for any $n \in \N$, 
 \begin{equation}
 \label{eqmoybound1-boun}
\sup_{t\in (0,T)} \|f_n\|_{\LL^2}^2  = \sup_{t\in (0,T)} \int_{\Omega} \int_{\R^d} |f_n|^2 \,  \frac{e^{V(x)}}{\Mc(v)} \, dv \, dx \leq C_0.
 \end{equation}
 Assume in addition that the sequence $(|f_n|^2)_{n\in \N}$ is equi-integrable with respect to the measure $d \mu:=e^{V}/\Mc \, dv dx$.
 Consider  $\rho_{n}(t,x):=\int_{\R^d}  f_n(t,x,v) \,dv$. Suppose that $f_n \rightharpoonup f$ weakly$-\star$ in $L^\infty(0,T; \LL^2(\Omega \times \R^d))$. 
 Then up to a subsequence, we have
\begin{equation}
\label{eqconclu1-boun}
\rho_n \Mc(v) \rightarrow \left(\int_{\R^d} f \, dv\right) \Mc(v) , \quad \text{strongly in  } L^2(0,T; \LL^2(\Omega \times \R^d))
\end{equation}
and for any continuous kernel $k(\cdot,\cdot,\cdot): \overline{\Omega}\times \R^d \times \R^d \to \R$ satisfying {\bf A3}, we have
\begin{equation}
\label{eqconclu2-boun}
\int_{\R^d} k(x,v',v) f_n(t,x,v') \, dv'  \rightarrow \int_{\R^d} k(x,v',v) f(t,x,v') \, dv'  , \quad \text{strongly in  } L^2(0,T;\LL^2(\Omega \times \R^d)).
\end{equation}

\end{coro}

 \begin{proof}[Proof of Corollary \ref{thmmoyenne-domain2}]
 The result follows from a localization and approximation argument. Consider $(\Phi_k(x))_{k \in \N}$ a sequence of smooth approximations of unity in $\Omega$, such that for any $k \in \N$, there exists $C_k>0$ with
 $$
 \| \Phi_k \|_{L^{\infty}} \leq 1 , \quad \| \Phi_k \|_{W^{1,\infty}} \leq C_k.
 $$
 Fix $T>0$.  Take $\eps>0$ and write the decomposition
 $$
 \rho_n - \rho = (\rho_n -\rho)  \Phi_k  + (\rho_n- \rho) (1-\Phi_k) =: B_1 + B_2.
 $$
For $B_2$, we write
 \begin{align*}
\| \rho_n (1- \Phi_k ) \Mc(v)\|_{\LL^2(\Omega \times \R^d))}^2 &\leq \int |f_n|^2 (1-\Phi_k)^2 \frac{e^V}{\Mc} \, dv dx, \\
&\leq  \int_{\supp(1-\Phi_k)\times \R^d} |f_n|^2  \frac{e^V}{\Mc} \, dv dx.
\end{align*}
 Using the fact that $\Leb (\supp(1-\Phi_k)) \to 0$ and the equiintegrability of $|f_n|^2$, we can consider  $k$ large enough so that for all $n\in \N$,
 $$
  \| B_2\|_{ L^2(0,T;\LL^2(\Omega \times \R^d))} \leq \eps/2.
 $$
Once $k$ is fixed, consider $\tilde{f}_n^k = \Phi_k f_n$, which satisfies the transport equation
 $$\partial_t \tilde{f}_n^k + v \cdot \nabla_x \tilde{f}_n^k = (1- \Delta_{v})^{m/2} (\Phi_k g_n) + v \cdot \nabla_x \Phi_k  f_n,
 $$
 For $B_1$, we can apply Corollary \ref{lemmoyenne} and take $n$ large enough to ensure
  $$
  \| B_1\|_{ L^2(0,T;\LL^2(\Omega \times \R^d))} \leq \eps/2,
 $$
  which concludes the first part of the proof. 
 
 The rest of the proposition is proved exactly as for Corollary \ref{lemmoyenne}.

 \end{proof}

 \section{Reformulation of some geometric properties}
 \label{GCCother}

\subsection{Proof of Lemma~\ref{equiv-equiv}}
 \label{prooflemequiv}
We first define another convenient equivalence relation.
 \begin{deft}
 \label{def-sim-phi}
 Given $\omega_1$ and $\omega_2$ two connected components of $\omega$, we say that $\omega_1 \bumpeq  \omega_2$ if  there is $N \in \N$ and $N$ connected components $(\omega^{i})_{1\leq i \leq N}$ of $\omega$ such that 
 \begin{itemize}
 \item we have $\omega_1 \Rc_{\phi} \, \omega^{(1)}$,
 \item for all $1\leq i \leq N-1$, we have $\omega^{(i)}  \Rc_{\phi} \, \omega^{(i+1)}$,
  \item we have $\omega^{(N)}  \Rc_{\phi} \, \omega_2$.
\end{itemize}
  The relation $\bumpeq$ is an equivalence relation on the set of connected components of $\omega$. For $\omega_1$ a connected component of $\omega$,  we denote its equivalence class for $\bumpeq$ by $\{\omega_1\}$.
 \end{deft}

Then, the proof of Lemma~\ref{equiv-equiv} relies on the following lemma.

\begin{lem}
\label{lem-classeq}
Let $\Omega_0$ be a connected component of $\bigcup_{s \in \R^+}\phi_{-s}(\omega)$ and let $(\omega_\ell)_{\ell \in L}$ be the connected components of $\omega$ such that for all $\ell \in L$, there exists $t\geq 0$ with $\phi_{-t} (\omega_\ell) \cap \Omega_0 \neq \emptyset$. Then, for all $\ell,\ell' \in L$, we have $\omega_\ell \bumpeq \omega_{\ell'}$.
\end{lem}

\begin{proof}[Proof of Lemma~\ref{lem-classeq}]
Assume that there exist at least two equivalence classes for $\bumpeq$ among the $\omega_\ell$, $\ell \in L$. Let $\ell_0 \in L$ and consider $\{\omega_{\ell_0}\}$ the equivalence class of $\omega_{\ell_0}$ for $\bumpeq$. 
Defining
$$
U_1 := \bigcup_{U \in \CC(\omega), \,  U \in \{\omega_{\ell_0}\}} \bigcup_{t\geq 0} \phi_{-t} (U) \cap \Omega_0 \quad \text{ and } \quad U_2 := \bigcup_{U \in \CC(\omega), \, U \notin \{\omega_{\ell_0}\}}\bigcup_{t\geq 0} \phi_{-t} (U) \cap \Omega_0,
$$
we have by construction that $U_1, U_2 $ are two open non-empty subsets of $\Omega_0$ and that $U_1 \cup U_2 = \Omega_0$.

Let us check $U_1 \cap U_2 \neq \emptyset$: otherwise there would exist two connected components of $\omega$, $U_1 \in \{\omega_{\ell_0}\}$, $U_2 \notin \{\omega_{\ell_0}\}$ such that $U_1 \Rc_\phi U_2$, which is excluded by definition of the equivalence class.

This is a contradiction with the fact that $\Omega_0$ is connected.

\end{proof}

We are now in position to prove Lemma~\ref{equiv-equiv}.
\begin{proof}[Proof of Lemma~\ref{equiv-equiv}] 
Let us first prove that  $\omega_1 \Bumpeq \omega_2 \Longrightarrow \Psi(\omega_1) \sim \Psi(\omega_2)$. It suffices to prove that 
\begin{equation}
\label{equivequiv}
\Big(\omega_1 \Rc_k \,  \omega_2 \text{ or } \omega_1 \Rc_\phi \, \omega_2 \Big) \Longrightarrow \left(\Psi(\omega_1)  \Rc_k \Psi(\omega_2) \text{ or } \Psi(\omega_1) \Rc_\phi \Psi(\omega_2) \right) .
\end{equation}
The conclusion then follows from the iterative use of this argument.

If $\omega_1 \Rc_k \,  \omega_2$, then $\Psi(\omega_1)  \Rc_k \Psi(\omega_2)$. This follows from the fact that $\omega_j  \subset \Psi(\omega_j)$ and the definition of $\Rc_k$. Similarly, if $\omega_1 \Rc_\phi \,  \omega_2$, then $\Psi(\omega_1)  \Rc_\phi \, \Psi(\omega_2)$.

\medskip
Let us now prove that $\Psi(\omega_1) \sim \Psi(\omega_2) \Longrightarrow \omega_1 \Bumpeq \omega_2 $. 
According to Lemma~\ref{lem-classeq}, it is sufficient to prove that $\Omega^{(1)}, \Omega^{(2)} $ being two given connected components of $\bigcup_{t\geq 0} \phi_{-t} (\omega)$,
\begin{equation}
\label{equivequivbis}
\Omega^{(1)} \Rc_k \, \Omega^{(2)} 
\Longrightarrow 
\begin{array}{l}
\text{there exits two connected components } \omega_1^* \text{ and } \omega_2^* \text{ of } \omega \\
\text{such that }
 \omega_1^* \Rc_k \, \omega_2^*  \text{ and } \omega_1^* \subset \Omega^{(1)} , \, \omega_2^* \subset \Omega^{(2)} .
\end{array}
\end{equation}
The conclusion then follows from an iterative use of this argument. 

By definition of $\Rc_k$, there exist $(x,v_1,v_2) \in \T^d \times \R^d \times \R^d$ with $(x,v_1) \in\Omega^{(1)}$ and $(x,v_2) \in \Omega^{(2)}$ such that $k(x,v_1,v_2)>0$ or $k(x,v_2,v_1)>0$.

Note in particular that this implies  $(x,v_1) ,(x,v_2) \in  \omega$. Denoting by $\omega_1^*$ (resp.  $\omega_2^*$)  the connected component of $\omega$ such that $(x,v_1) \in \omega_1^*$ (resp. $(x,v_2) \in \omega_2^*$), we hence have $\omega_1^* \Rc_k \, \omega_2^*$. The conclusion of~\eqref{equivequivbis} follows from the fact that $\omega_j^* \subset\Omega^{(j)}$, for $j=1,2$ and the definition of $\Rc_k$.
\end{proof}

\subsection{Almost everywhere geometric control conditions and connectedness}
\label{appaeitgccconnected}

\begin{prop}
\label{aeitgccconnected}
Here, $\Omega$ is either $\T^d$ or an open subset of $\R^d$.
Let $\omega \subset \Omega \times \R^d$ be an open subset. Consider the following geometric properties:
\begin{enumerate}[(i)]
\item There exists $\tilde{\omega} \subset \omega$, $\tilde{\omega}$ {\em connected} satisfying the a.e.i.t. Geometric Control Condition;
\item The set $\omega$ satisfies the a.e.i.t. GCC and for any connected components $(\omega_1,\omega_2)$ of $\omega$, there exists $(x_0,v_0) \in \omega_1$ and $s \in \R$ such that $\phi_s(x_0,v_0) \in \omega_2$;
\item The set $\omega$ satisfies the a.e.i.t. GCC and $\bigcup_{s\in \R^+}\phi_{-s}(\omega)$ is connected.
\end{enumerate}

Then $(i) \implies (iii)$ and $(ii) \implies (iii)$.

\end{prop}

\begin{proof}[Proof of Proposition \ref{aeitgccconnected}] Before starting the proof, let us remark that is $\omega$ is a connected open subset of $\Omega \times \R^d$, then $ \bigcup_{s\in \R^+}\phi_{-s}(\omega)$ is also a connected open subset. Indeed it is first an open subset of $\Omega \times \R^d$, and it is equivalent to show that it is path-connected. Let $y_1, y_2 \in \bigcup_{s\in \R^+}\phi_{-s}(\omega)$; there exists $s_1,s_2 \geq 0$ and $z_1, z_2 \in \omega$ such that $y_1 = \phi_{-s_1} (z_1)$ and $y_2 = \phi_{-s_2} (z_2)$.  Since $\omega$ is a connected open subset of $\Omega \times \R^d$, it is also path-connected and one can find a continuous path in $\omega$ between $z_1$ and $z_2$ in $\omega$. Using the application $\phi_{-s}$, we also get a continuous path between $y_1$ and $z_1$ in $\bigcup_{s\in \R^+}\phi_{-s}(\omega)$ (resp. between $y_2$ and $z_2$).

Gluing these paths together, this yields a continuous path in $\bigcup_{s\in \R^+}\phi_{-s}(\omega)$ between $y_1$ and $y_2$.

\bigskip

\noindent $\bullet$ $(i) \implies (iii)$. Since $\tilde{\omega} \subset \omega$, we have $\bigcup_{s\in \R^+}\phi_{-s}(\tilde\omega) \subset \bigcup_{s\in \R^+}\phi_{-s}(\omega)$. Denote by $(\Omega_i)_{i \in I}$ the connected components of $\bigcup_{s\in \R^+}\phi_{-s}(\omega)$ . The sets $\Omega_i$ are connected open sets so that the inclusion $\bigcup_{s\in \R^+}\phi_{-s}(\tilde\omega) \subset \bigcup_{s\in \R^+}\phi_{-s}(\omega)$ together with the connectedness of $\bigcup_{s\in \R^+}\phi_{-s}(\tilde\omega)$ yields the existence of $i_0 \in I$ such that $\bigcup_{s\in \R^+}\phi_{-s}(\tilde\omega) \subset \Omega_{i_0}$. Since $\bigcup_{s\in \R^+}\phi_{-s}(\tilde\omega)$ is of full measure, this is also the case for $\Omega_{i_0}$. As $\Omega_i$ is open, we obtain that $\Omega_i = \emptyset$ for $i \neq i_0$, so that $\Omega_{i_0} = \bigcup_{s\in \R^+}\phi_{-s}(\omega)$ is connected (and of full measure).

\bigskip

\noindent $\bullet$ $(ii) \implies (iii)$.  Let $y_1, y_2 \in \bigcup_{s\in \R^+}\phi_{-s}(\omega)$; there exists $s_1,s_2 \geq 0$ and $z_1, z_2 \in \omega$ such that $y_1 = \phi_{-s_1} (z_1)$ and $y_2 = \phi_{-s_2} (z_2)$. 
If $z_1, z_2$ belong to the same connected component $\tilde \omega$ of $\omega$, then since $\bigcup_{s\in \R^+}\phi_{-s}(\tilde\omega)$ is connected, one can find a continuous path between $z_1$ and $z_2$. 

If $z_1, z_2$ belong to two different connected components $\omega_1$ and $\omega_2$, apply $(i)$ to find, up to a permutation between the indices $1$ and $2$, $u \in \omega_1$ and $s \in \R^+$ such that $\phi_{-s}(u) \in \omega_2$. Then one can find a continuous path between $u$ and  $z_1$ in $\omega_1$, and another between  $\phi_{-s}(u)$ and $z_2$ in $\omega_2$. We conclude as in the previous sub case by gluing the paths together.

\end{proof}

Note in particular that $(i)-(ii)-(iii)$ hold as soon as $\omega$ is connected and satisfies a.e.i.t. GCC.

\section{Proof of Proposition \ref{CEXucp}}
\label{proofCEXucp}

In this section, we prove Proposition \ref{CEXucp}.

   \begin{proof}[Proof of Proposition \ref{CEXucp}]
Let $x_0 \in  \T^d \setminus \overline{p_x(\omega)} \neq \emptyset$ and take $\eta>0$ such that $B(x_0 , 2\eta) \cap \overline{p_x(\omega)} = \emptyset$.
Define the potential $V (x) :=  \frac{|x-x_0|^2}{2} \Psi(x)$, where $\Psi$ is a ``corrector'' to ensure $V\in C^\infty (\T^d)$, and such that $\Psi \equiv 1$ on $B(x_0 , 2\eta)$ (reduce $\eta$ if necessary). Denote $V_\eps = \eps V$ and notice that $\nabla V_\eps(x) = \eps (x-x_0)$ on $B(x_0 , 2\eta)$.
As a consequence, the hamiltonian flow $(\phi_t)_{t \in \R}$ associated to the vector field $v \cdot \nabla_x - \nabla_x V_\eps \cdot \nabla_v$ may be explicited in the set $B(x_0,2\eta) \times \R^d$: we have
 $$
 \phi_t(x,v) = \left(x_0 + (x-x_0) \cos(\sqrt{\eps} t) + \frac{v}{\sqrt{\eps}}\sin (\sqrt{\eps} t) , - (x-x_0) \sqrt{\eps} \sin(\sqrt{\eps} t) + v \cos (\sqrt{\eps} t) \right) ,
 $$
 as long as $\phi_t(x,v) \in B(x_0,2\eta) \times\R^d$. In particular, note that if $(x,v)\in B(x_0,\eta) \times B(0, \sqrt{\eps}\eta)$, then $\phi_t(x,v)$ remains in $B(x_0,2\eta) \times B(0, 2 \sqrt{\eps}\eta)$ for all $t \in \R^+$. This reads $\phi_t(B(x_0,\eta) \times B(0,\sqrt{\eps}\eta)) \subset  B(x_0,2\eta) \times B(0,2\sqrt{\eps}\eta)$, i.e. in particular $\phi_t(B(x_0,\eta) \times B(0,\sqrt{\eps}\eta)) \cap \omega = \emptyset$ for all $t \in \R^+$. This proves that a.e.i.t. GCC is not satisfied.
 \end{proof}

\begin{rque}
Notice that in the previous proof, to handle small potential, we consider small speeds, i.e. $v \in B(0,\sqrt{\eps}\eta)$. In the opposite direction, if one fixes the speeds in a large Hamiltonian sphere $v \in S_H(0,R)$ (note that with the particular potential used in the proof, on the set $\{\Psi=1\}$ we have $S_H(0,R) = S(0,R)$) for some $R>0$, then one can find a (large) potential (namely $R^2/\eta^2 V$ where $V$ is that of the previous proof) such that a.e.i.t. GCC fails.
\end{rque}

\section{Other linear Boltzmann type equations}
\label{Other}

The goal of this appendix is to show that the methods developed in Part \ref{LTB} can be adapted to handle other types of Boltzmann-like equations.

\subsection{Generalization to a wider class of kinetic transport equations}

Consider now the equation
\begin{equation}
\label{B-general}
\partial_t f + a(v) \cdot \nabla_x f - \nabla_x V \cdot \nabla_v f= \int_{\R^d} \left[k(x,v' ,  v) f(v') - k(x,v ,  v') f(v)\right] \, dv', 
\end{equation}
where $a(v) = \nabla_v A(v)$ with $A: \R^d \to \R$ is such that $\int_{\R^d} e^{-A(v)} \, dv<+\infty$.

For simplicity, we assume that \eqref{B-general} is set on $\T^d\times \R^d$ (but it is also possible to consider the case of bounded domains with specular reflection, as in Section \ref{Boundary}).

Assume that $a(v)$ satisfies a \emph{non degeneracy} property: there exists $\gamma \in (0,2)$ and $C>0$ such that, for all $\xi \in \S^{d-1}$, 
\begin{equation}
\label{nondege}
\Leb \left\{v \in \R^d, \, |a(v) \cdot \xi| \leq \eps  \right\} \leq C \eps^\gamma.
\end{equation}
This prevents concentrations of $a(v)$ in any direction of $\S^{d-1}$.

The hamiltonian associated to the transport equation is then the following:
\begin{equation}
H(x,v)= A(v) + V(x).
\end{equation}

Define the global Maxwellian associated to $a(v)$:
\begin{equation}
 \Mc_A(v) = C_A e^{-A(v)},
\end{equation}
with $C_A = 1/ \left( \int_{\R^d} e^{-A(v)} \, dv \right)$.

In addition to the usual assumption {\bf A1} on the collision kernel $k$, we shall assume the following (which replace {\bf A2} --{\bf A3}):

\noindent {\bf A2'.} We assume that $\Mc$ cancels the collision operator, that is
\begin{equation}
\label{M_Aannule}
\text{for all }  (x,v) \in \Omega \times \R^d, \quad  \int_{\R^d} \left[k(x,v' ,  v)  \Mc_A(v') - k(x,v ,  v')  \Mc_A(v)\right] \, dv'  = 0.
\end{equation}

\noindent {\bf A3'.} We assume that
\begin{equation}
\label{bornek-A}
\tilde{k}(x,v',v) := \frac{k(x,v' ,  v) }{\Mc_A(v)} \in L^\infty(\T^d \times \R^d \times \R^d).
\end{equation}

For $a(v)=v$ (for which $\gamma=1$ in \eqref{nondege} ), we recover the framework which has been already treated before.
One physically relevant case is 
$$a_{rel}(v):= \frac{v}{\sqrt{1+ |v|^2}},$$
 (for which we also have $\gamma=1$ in \eqref{nondege}), which allows to model relativistic transport. Note that in this case, we have $A_{rel}(v) :=\sqrt{1+|v|^2}$, and the related Maxwellian is then the so-called \emph{relativistic Maxwellian}:
\begin{equation}
  \Mc_{rel}(v) := C_{rel} e^{-\sqrt{1+|v|^2}},
\end{equation}
where $C_{rel}$ is a normalizing constant, so that $\int_{\R^d}  \Mc_{rel}(v) \, dv =1$.

\bigskip

Our aim in this paragraph is to show that the methods developed in Part \ref{LTB} are still relevant here. The characteristics of the equation are defined in the following way:
  \begin{deft}
  \label{def-carac-mod}
 Let $V \in  W^{2,\infty}_{loc} (\T^d)$. Let $(x_0,v_0) \in \T^d \times \R^d$. The characteristics  $\phi_t(x_0,v_0) := (X_t (x_0,v_0), \, \Xi_t(x_0,v_0))$ associated to the hamiltonian $H(x,v) =  A(v) + V(x)$ are defined as the solutions to the system:
 \begin{equation}
 \left\{
 \begin{aligned}
 &\frac{dX_t}{dt} = a(\Xi_t), \quad \frac{d\Xi_t}{dt} = - \nabla_x V(X_t), \\
  &X_{t=0}=x_0, \quad \Xi_{t=0}=v_0.
 \end{aligned}
 \right.
 \end{equation}
\end{deft}

With this definition of characteristics, we can then properly define
\begin{itemize}
\item the set $\omega$ where collisions are effective, as in Definition \ref{def-om},
\item the Unique Continuation Property, as in Definition \ref{def:UCP}
\item $C^-(\infty)$, as in Definition \ref{definitionCinfini},
\item a.e.i.t. GCC, as in Definition \ref{defaeitgcc}, 
\item the equivalence relation $\sim$, as in Definition \ref{def-sim}.
\end{itemize}

The next thing to do concerns the local well-posedness of \eqref{B-general} in some relevant weighted spaces, which we introduce below.
\begin{deft}[Weighted $L^p$ spaces]
 We define the Banach spaces $\LL^2_A$ and $\LL^\infty_A$ by 
 \begin{equation*}
 \begin{aligned}
  &\LL^2_A := \Big\{f \in L^1_{loc}(\T^d \times \R^d), \, \int_{\T^d \times \R^d} | f|^2 \frac{e^V}{ \Mc_{A}(v)} \, dv \, dx < + \infty \Big\} , \\
 &\qquad \qquad \qquad \|f\|_{\LL^2_A} = \left(\int_{\T^d \times \R^d} | f|^2 \frac{e^V}{ \Mc_{A}(v)} \, dv \, dx \right)^{1/2}, \\
 & \LL^\infty_A : =  \Big\{f \in L^1_{loc}(\T^d \times \R^d), \, \sup_{\T^d \times \R^d} | f| \frac{e^V}{ \Mc_A(v)}  < + \infty \Big\} , 
 \quad \|f\|_{\LL^\infty_A} = \sup_{\T^d \times \R^d} | f| \frac{e^V}{ \Mc_A(v)} 
  \end{aligned}
  \end{equation*}
 The space $\LL^2_A$ is a Hilbert space endowed with the inner product
 $$
 \langle f, g \rangle_{\LL^2_A} := \int_{\T^d \times \R^d} e^{V} \frac{f \, g}{ \Mc_A(v)} \, dv \, dx.
 $$
 \end{deft}
As usual, we have
  \begin{prop}[Well-posedness of the linear Boltzmann equation with modified transport]
 \label{prop:WP-mod}
Assume that $f_0 \in \LL^2_A$. Then there exists a unique $f\in C^0(\R ;\LL^2_A)$ solution of~\eqref{B-general} satisfying $f|_{t = 0} =f_0$, and we have
\begin{equation}
\text{ for all } t \geq 0, \quad  \frac{d}{dt} \| f(t)\|_{\LL^2_A}^2 = - D_A(f(t)),
\end{equation}
 where 
  \begin{equation}
 \label{defD-A}
 D_A(f) =  \frac{1}{2} \int_{\Omega} e^{V} \int_{\R^d} \int_{\R^d} \left( \frac{k(x,v' ,  v)}{ \Mc_A(v)} + \frac{k(x,v ,  v')}{ \Mc_A(v')} \right)  \Mc_A(v) \Mc_A(v') \left(\frac{f(v)}{ \Mc_A(v)}- \frac{f(v')}{ \Mc_A(v')}\right)^2 \, dv' \, dv \, dx.
 \end{equation}
  If moreover $f_0 \geq 0$ a.e., then for all $t \in \R$ we have $f(t, \cdot,\cdot)\geq 0$ a.e. (Maximum principle).
 \end{prop}

Then the analogues of Theorems \ref{thmconvgene-intro}, \ref{thmconv-intro} and \ref{thmexpo-intro} hold in this setting (with some obvious modifications); for the sake of conciseness, we omit these statements.

Such results can be proved  exactly as Theorems \ref{thmconvgene-intro},  \ref{thmconv-intro} and \ref{thmexpo-intro}.
The crucial additional ingredient is the fact that averaging lemmas for the operator $a(v) \cdot \na_x$ still hold, precisely when $a$ satisfies the non degeneracy condition \eqref{nondege}, see \cite{GLPS}. Note that in this case, the gain of regularity on averages depends on the index $\gamma$ in \eqref{nondege}, but in any case, this is always sufficient to optain compactness.

\subsection{Generalization to linearized BGK operators}
\label{sec-BGK}
Once again, for simplicity, we assume that $(x,v) \in \T^d \times \R^d$. Let $V \in W^{2,\infty}_{loc}(\T^d)$.
Let $\varphi: \R \to \R^+_*$ be a function in $L^\infty(\R)$ such that
$$
\int_{\T^d \times \R^d} \varphi\left( \frac{|v|^2}{2} + V(x) \right)\, dv \, dx < +\infty,
$$
Denote $F(x,v) = \varphi\left( \frac{|v|^2}{2} + V(x) \right)$ and $\rho_F(x) = \int F(x,v) \, dv$.

Let $\sigma \in C^0(\T^d)$ be a non-negative function. We study in this paragraph the following degenerate linearized BGK equation:
\begin{equation}
\label{BGK}
\partial_t f + v \cdot \nabla_x f - \nabla_x V \cdot \nabla_v f = \sigma(x) \left( \frac{\int_{\R^d} f\, dv}{\rho_F(x)} F(x,v) - f \right).
\end{equation}
with an initial condition $f_0$ at time $0$. The natural equilibrium is given by 
$$
(x,v) \mapsto \int_{\T^d \times \R^d} f_0 \, dv \, dx  \frac{F(x,v)}{\int_{\T^d \times \R^d} F(x,v) \, dv \, dx}. 
$$
The main feature of this equilibrium is that there is no separation of variables contrary to the Maxwellian case.

Our aim in this paragraph is again to show that the methods developed in Part \ref{LTB} are still relevant here.
For what concerns well-posedness, we introduce the relevant weighted $L^p$ spaces and have the usual result.
\begin{deft}[Weighted $L^p$ spaces]
We define the Banach spaces $\LL^2_{bgk}$ and $\LL^\infty_{bgk}$ by 
 \begin{equation*}
 \begin{aligned}
  &\LL^2_{bgk} := \Big\{f \in L^1_{loc}(\T^d \times \R^d), \, \int_{\T^d \times \R^d} | f|^2\frac{1}{F(x,v)} \, dv \, dx < + \infty \Big\} , \\
 &\qquad \qquad \qquad \|f\|_{\LL^2_{bgk}} = \left(\int_{\T^d \times \R^d} | f|^2 \frac{1}{F(x,v)}\, dv \, dx \right)^{1/2}, \\
 & \LL^\infty_{bgk} : =  \Big\{f \in L^1_{loc}(\T^d \times \R^d), \, \sup_{\T^d \times \R^d} | f| \frac{1}{F(x,v)}  < + \infty \Big\} , 
 \quad \|f\|_{\LL^\infty} = \sup_{\T^d \times \R^d} | f| \frac{1}{F(x,v)} 
  \end{aligned}
  \end{equation*}
 The space $\LL^2_{bgk}$ is a Hilbert space endowed with the inner product
 $$
 \langle f, g \rangle_{\LL^2_{bgk}} := \int_{\T^d \times \R^d}  \frac{f \, g}{F(x,v)} \, dv \, dx.
 $$
 \end{deft}

  \begin{prop}[Well-posedness of the linearized BGK equation]
 \label{prop:WP-BGK}
Assume that $f_0 \in \LL^2_{bgk}$. Then there exists a unique $f\in C^0(\R ;\LL^2_{bgk})$ solution of~\eqref{BGK} satisfying $f|_{t = 0} =f_0$, and we have
\begin{equation}
\text{ for all } t \geq 0, \quad  \frac{d}{dt} \| f(t)\|_{\LL^2_{bgk}}^2 = - {D}_{bgk}(f(t)),
\end{equation}
 where 
 $${D}_{bgk}(f)= \int_{\T^d} e^V\sigma(x) \int_{\R^d} \int_{\R^d} \frac{F(x,v) F(x,v')}{\rho_F(x)} \left( \frac{f(v)}{F(x,v)} - \frac{f(v')}{F(x,v')}\right)^2 \, dv' \, dv \, dx.$$ 
 If moreover $f_0 \geq 0$ a.e., then for all $t \in \R$ we have $f(t, \cdot,\cdot)\geq 0$ a.e. (Maximum principle).

 \end{prop}

With the same geometric definitions of Section \ref{main}, we have the following results.
Note that the set $\omega$ where the collisions are effective is equal to $\omega_x \times \R^d$, where
$$
\omega_x := \{x \in \T^d, \, \sigma(x)>0\}.
$$

\begin{thm}[Convergence to equilibrium]
\label{thmconv-BGK}
 The following statements are equivalent.
\begin{enumerate}[(i.)]
\item The set $\omega$ satisfies the Unique Continuation Property.

\item The set $\omega$ satisfies the a.e.i.t. GCC and $\bigcup_{s\in \R^+}\phi_{-s}(\omega)$ is connected.

\item  For all $f_0 \in \LL^2_{bgk}$, denote by $f(t)$ the unique solution to \eqref{BGK} with initial datum $f_0$. We have
\begin{equation}
 \label{convergeto0-BGK}
\left\|f(t)-\left(\int_{\T^d \times \R^d} f_0 \, dv \,dx\right)\frac{F(x,v)}{\int_{\T^d \times \R^d} F(x,v) \, dv \,dx}\right\|_{\LL^2_{bgk}} \to_{t \to +\infty} 0,
 \end{equation}

\end{enumerate}

\end{thm}

\begin{thm}[Exponential convergence to equilibrium]
\label{thmexpo-BGK}
The two following statements are equivalent:
\begin{enumerate}[(i.)]
\item $C^-(\infty) > 0$. 
\item There exists $C>0, \gamma>0$ such that for any $f_0 \in \LL^2(\T^d \times \R^d)$, the unique solution to \eqref{BGK} with initial datum $f_0$ satisfies
\begin{multline}
\label{decexpo-BGK} 
\left\|f(t)-\left(\int_{\T^d \times \R^d} f_0 \, dv \,dx \right) \frac{F(x,v)}{\int_{\T^d \times \R^d} F(x,v) \, dv \,dx} \right\|_{\LL^2} \\
\leq C e^{-\gamma t} \left\|f_0-\left(\int_{\T^d \times \R^d} f_0 \, dv \,dx \right)  \frac{F(x,v)}{\int_{\T^d \times \R^d} F(x,v) \, dv \,dx}\right\|_{\LL^2}.
\end{multline}
\end{enumerate}
\end{thm}

We shall not dwell on the proofs of Theorems \ref{thmconv-BGK} and \ref{thmexpo-BGK}, since they are very similar to those of Theorems \ref{thmconv-intro} and \ref{thmexpo-intro}. Indeed, we note that the structure of the equation \eqref{BGK} is similar to that of \eqref{B}, in the sense that the ``degenerate dissipative'' part is still made of a dissipative term plus a relatively compact term. This compactness, as usual, comes from averaging lemmas.

Let us just underline a crucial point in the proof $(ii.)$ implies $(i.)$ of Theorem \ref{thmconv-BGK}. This comes from the fact that  $F(x,v)$ does not separate the $x$ and $v$ variables, contrary to the Maxwellian equilibrium of \eqref{B} and thus we have to be careful. Let us check that the proof we gave in the Boltzmann case is still relevant (see the proof of $(ii.) \implies (i.)$ of Theorem \ref{thmconv-intro}).

 Let $f \in C^0_t(\LL^2_{bgk})$ be a solution to
 \begin{align}
\label{eqhypo1-BGK} \partial_t f + v \cdot \nabla_x f - \nabla_x V \cdot \nabla_v f= 0, \\
\label{eqhypo2-BGK} f = \rho(t,x) F(x,v) \text{  on  } \R^+ \times \omega.
\end{align}
Assume that $\int_{\T^d \times \R^d} f \, dv \,dx=0$. The goal is to show that $f=0$.

 To this purpose, as before, consider for $(t,x, v) \in \R^+ \times \omega$,  $g(t,x) := \frac{1}{F(x,v)} \, f$ (note that by \eqref{eqhypo2-BGK}, $g$ does not depend on $v$). 
 We have, for $(t,x, v) \in \R^+ \times \omega$:
 $$
 \partial_t g + v \cdot \nabla_x g =   \frac{1}{F(x,v)} \left[\partial_t f + v \cdot \nabla_x f - v\cdot  \nabla_x F \, \frac{f}{F}\right].
 $$
 Since $f$ satisfies \eqref{eqhypo1-BGK} and \eqref{eqhypo2-BGK}, 
 $$\partial_t f + v \cdot \nabla_x f  =  \nabla_x V \cdot \nabla_v f =\nabla_v F \cdot \nabla_x V \, \frac{f}{F}.
 $$
 By definition of $F$, we have
 $$
\nabla_x V \cdot  \nabla_v F  = v \cdot \nabla_x F,
 $$
 from which we deduce that $g$ satisfies the free transport equation on $\omega$:
 \begin{equation}
 \label{transg-BGK}
  \text{ for all } (t,x, v) \in \R^+ \times \omega, \quad \partial_t g + v \cdot \nabla_x g =0.
 \end{equation}
We then conclude as in the proof of $(ii.) \implies (i.)$ of Theorem \ref{thmconv-intro}, \emph{mutatis mutandis}.


  \bibliographystyle{plain}
\bibliography{HKL4}
 
 \end{document}